\documentclass[12pt,reqno,a4paper]{amsart}

\usepackage{amssymb}
\usepackage{dsfont} 
\usepackage{mathrsfs}
\usepackage[hidelinks]{hyperref}
\usepackage{mathabx}
\usepackage[all]{xy}
\usepackage{graphicx} 
\usepackage{appendix}

\usepackage{geometry}
\newtheorem{proposition}{Proposition}[section]
\newtheorem{theorem}[proposition]{Theorem}
\newtheorem{corollary}[proposition]{Corollary}
\newtheorem{lemma}[proposition]{Lemma}
\theoremstyle{definition}
\newtheorem{definition}[proposition]{Definition}
\newtheorem{remark}[proposition]{Remark}
\newtheorem{remarks}[proposition]{Remarks}

\newcommand{\cst}{\ensuremath{\mathrm{C}^*}}
\newcommand{\comp}{\circ}
\newcommand{\ee}{\mathrm{e}}
\newcommand{\eps}{\varepsilon}
\newcommand{\ph}{\varphi}
\newcommand{\hh}[1]{\widehat{#1}}
\newcommand{\I}{\mathds{1}}
\newcommand{\id}{\mathrm{id}}
\newcommand{\ii}{\mathrm{i}}
\newcommand{\is}[2]{{\left\langle{#1}\,\vline\,#2\right\rangle}}
\newcommand{\ket}[1]{{\left|#1\right\rangle}}
\newcommand{\bra}[1]{{\left\langle#1\right|}}
\newcommand{\tens}{\otimes}
\newcommand{\algtens}{\otimes_{\textrm{alg}}}
\newcommand{\algboxtimes}{\boxtimes_{\textrm{alg}}}

\newcommand{\CC}{\mathbb{C}}
\newcommand{\GG}{\mathbb{G}}
\newcommand{\HH}{\mathbb{H}}
\newcommand{\NN}{\mathbb{N}}
\newcommand{\RR}{\mathbb{R}}
\newcommand{\TT}{\mathbb{T}}
\newcommand{\ZZ}{\mathbb{Z}}

\newcommand{\cA}{\mathscr{A}}
\newcommand{\cB}{\mathscr{B}}
\newcommand{\cK}{\mathscr{K}}

\newcommand{\sA}{\mathsf{A}}
\newcommand{\sB}{\mathsf{B}}

\newcommand{\sH}{\mathsf{H}}
\newcommand{\sK}{\mathsf{K}}
\newcommand{\sM}{\mathsf{M}}
\newcommand{\sN}{\mathsf{N}}
\newcommand{\sX}{\mathsf{X}}
\newcommand{\sY}{\mathsf{Y}}

\newcommand{\bh}{{\boldsymbol{h}}}
\newcommand{\bphi}{{\boldsymbol{\phi}}}
\newcommand{\brho}{{\boldsymbol{\rho}}}
\newcommand{\biota}{{\boldsymbol{\iota}}}
\newcommand{\bn}{{\boldsymbol{n}}}
\newcommand{\bz}{{\boldsymbol{z}}}

\newcommand{\uu}{\text{\rm u}}
\newcommand{\ur}{\text{\rm u,r}}
\newcommand{\ww}{\mathrm{W}}
\newcommand{\vv}{\mathrm{V}}

\newcommand{\Ww}{\mathds{W}}
\newcommand{\wW}{\text{\reflectbox{$\Ww$}}\:\!}

\DeclareMathOperator{\A}{A}
\DeclareMathOperator{\B}{B}
\DeclareMathOperator{\C}{C}
\DeclareMathOperator{\D}{D}
\DeclareMathOperator{\End}{End}

\DeclareMathOperator{\Ltwo}{\mathnormal{L}^2}
\DeclareMathOperator{\Linf}{\mathnormal{L}^{\infty}}
\DeclareMathOperator{\M}{M}
\DeclareMathOperator{\Mat}{\mathsf{Mat}}
\DeclareMathOperator{\Mor}{Mor}
\DeclareMathOperator{\Pol}{Pol}
\DeclareMathOperator{\cb}{cb}
\DeclareMathOperator{\ev}{ev}
\DeclareMathOperator{\Sp}{Sp}

\newcommand{\SU}{\mathrm{SU}}
\newcommand{\cJhat}{\hat{\mathscr{J}}}
\newcommand{\kmap}{\kappa}

\numberwithin{equation}{section}

\title[Braided quantum $\mathrm{SU}(2)$ group]{Braided quantum $\mathrm{SU}(2)$ group -- a case study}

\author{Jacek Krajczok}
\address{Vrije Universiteit Brussel}
\email{jacek.krajczok@vub.be}

\author{Piotr M.~So{\l}tan}
\address{University of Warsaw}
\email{piotr.soltan@fuw.edu.pl}


\makeatletter
\@namedef{subjclassname@2020}{\textup{2020} Mathematics Subject Classification}
\makeatother

\begin{document}

\begin{abstract}
We continue the study of the braided compact quantum group $\mathrm{SU}_q(2)$ for complex $q$ satisfying $0<|q|<1$ introduced by Kasprzak, Meyer, Roy and Woronowicz (\cite{BraidedSU2}). We address such aspects as existence of the Haar measure, construct the scaling group, the antipode and its polar decomposition and describe the related braided Hopf algebra. We also study when the braided flip extends to a completely bounded map and establish equivalence between the two approaches to bosonization and braided tensor product taken in the literature (\cite{BraidedSU2} vs.~\cite{MeyerRoyWoronowiczI,RoyBraided, DeCommerKrajczok}).
\end{abstract}

\maketitle

\tableofcontents

\allowdisplaybreaks

\section{Introduction}

The general theory of compact quantum groups was introduced in \cite{CQGs} as an extension of the theory of compact groups to the framework of non-commutative topology (let us also note here that the paper \cite{CQGs} circulated in preprint form many years before its actual publication). Later the scope of the theory was enlarged to encompass locally compact objects (\cite{lcqg,mnw}). A prime example of compact quantum group is the quantum group $\SU_q(2)$ (\cite{su2}), with \emph{real} deformation parameter satisfying $0<|q|<1$. It is described by the universal unital \cst-algebra $\C(\SU_q(2))$ generated by elements $\alpha$ and $\gamma$ making the matrix $u=\left[\begin{smallmatrix}\alpha&-q\gamma^*\\\gamma&\alpha^*\end{smallmatrix}\right]$ unitary, together with a comultiplication $\C(\SU_q(2))\to\C(\SU_q(2))\tens\C(\SU_q(2))$ characterized by the property that $u$ is a unitary representation of $\SU_q(2)$ on $\CC^2$. The authors of \cite{BraidedSU2} extended this construction to \emph{complex} deformation parameter $0<|q|<1$ -- this time however, $\SU_q(2)$ is a \emph{braided} compact quantum group. This means that the comultiplication is a map $\C(\SU_q(2))\to\C(\SU_q(2))\boxtimes\C(\SU_q(2))$, where $\boxtimes$ is the braided tensor product -- a construction of which depends on $q$ and an action of $\TT$ on $\C(\SU_q(2))$ (alternatively we might look at the action of Drinfeld double $\D(\TT)$ on $\C(\SU_q(2))$, see Section \ref{sect-SUq2} and Appendix \ref{app-BraidedTP}). This objects fits into a more general framework of \cst-algebraic braided quantum groups studied in \cite{ExampleBraided,MeyerRoyWoronowiczII,RoyBraided} and was also studied in the framework of quantum groups with projection on the von Neumann algebra level in \cite{KaspSol2020} (see also \cite{KaspSol2015}).

In our work we continue the study of the braided compact quantum group $\SU_q(2)$ undertaken in \cite{BraidedSU2}. Our results go roughly in three directions.

First, we show that the dense unital $*$-algebra $\Pol(\SU_q(2))$ generated by $\alpha$ and $\gamma$ coincides with the space spanned by matrix coefficients of finite dimensional unitary representations (Proposition \ref{prop-prop7}). On this space we can introduce a version of the antipode map $S$ and a braided flip $\chi^{\boxtimes}$. Together with the counit (Proposition \ref{prop-lemma19}) it can be understood as a braided Hopf algebra, i.e.~a Hopf algebra in the braided monoidal category $\sideset{_\ZZ^\ZZ}{}{\operatorname{\mathcal{YD}}}$ of Yetter-Drinfeld modules over $\CC[\ZZ]$ (Section \ref{sect-BrHopf}). We also prove that $S$ acts in the way known from non-braided setting on matrix coefficients of representations (Proposition \ref{propprop5}).

Next, as in the non-braided setting, we introduce scaling automorphisms $(\tau_t)_{t\in\RR}$ which form a group of automorphisms of $\C(\SU_q(2))$, and define the unitary antipode $R$ which is a bounded map on $\C(\SU_q(2))$ (sections \ref{sect-tau}, \ref{sect-UApolar}). Despite a number of similarities, the equation $S=R\comp\tau_{-\ii/2}$ does \emph{not} hold when $q\not\in\RR$ and we need to introduce a \emph{residual mapping} $\vartheta$ on $\Pol(\SU_q(2))$ (Section \ref{sect-UApolar}). It provides us with a polar decomposition $S=R\comp\tau_{-\ii/2}\comp\vartheta^{-1}$ (Definition \ref{def-R}). We also include Proposition \ref{prop1} which says that the braided flip $\chi^\boxtimes$ extends to a completely bounded map on $\C(\SU_q(2))\boxtimes\C(\SU_q(2))$ if and only if $\zeta=q/\overline{q}$ is a root of unity.

Finally, we study the Haar measure $\bh$ on $\SU_q(2)$ which is defined via the bosonization (see Section \ref{sect-HaarMeasure}). As expected, it is a $\D(\TT)$-invariant KMS state on $\C(\SU_q(2))$, and using it one can define the von Neumann algebra $\Linf(\SU_q(2))$ via the GNS representation. The Haar measure is left and right invariant: $(\id\boxtimes\bh)\Delta_A(a)=\bh(a)\I=(\bh\boxtimes\id)\Delta_A(a)$ for all $a\in A$. We provide two arguments for this result.

In the first appendix, we carefully check that the two approaches to braided tensor product (the one used in \cite{BraidedSU2} and the more general one of \cite{MeyerRoyWoronowiczI,MeyerRoyWoronowiczII,DeCommerKrajczok}) are equivalent in an appropriate sense -- see Theorem \ref{thm2} and Remark \ref{rem1}. The second appendix establishes a canonical isomorphism between the two definitions of bosonization from \cite{BraidedSU2} and \cite{MeyerRoyWoronowiczII}. The third appendix is devoted to the proof of Proposition \ref{prop1} describing the conditions for the braided flip to extend to a completely bounded map on $\C(\SU_q(2))\boxtimes \C(\SU_q(2))$.

In order to lighten the notation, throughout the paper we will write $A$, $\cA$ and $M$ to denote the algebras $\C(\SU_q(2))$, $\Pol(\SU_q(2))$, and $\Linf(\SU_q(2))$ respectively. The comultiplication on $A=\C(\SU_q(2))$ will be denoted by $\Delta_A$.

Whenever we refer to a map $\Psi$ as a \emph{morphism} from a \cst-algebra $\sA$ to a \cst-algebra $\sB$, we mean that $\Psi$ is a non-degenerate $*$-homomorphism from $\sA$ to the multiplier algebra $\M(\sB)$ of $\sB$. For details of this notion and its use we refer the reader to \cite{unbo,Lance}. The set of morphisms from $\sA$ to $\sB$ will be denoted by the symbol $\Mor(\sA,\sB)$.

For a  locally compact quantum group $\GG$, we denote by $\ww^{\GG}\in\M(\C_0(\GG)\tens\C_0(\hh{\GG}))$ the Kac-Takesaki operator of $\GG$, i.e.~its left regular representation. A \emph{left action} of $\GG$ on (the quantum space underlying) a \cst-algebra $\sA$ is described by a morphism $\alpha\in\Mor(\sA,\C_0(\GG)\tens\sA)$ such that $(\id\tens\alpha)\comp\alpha=(\Delta_\GG\tens\id)\comp\alpha$ and $(\C_0(\GG)\tens\I)\alpha(\sA)$ is dense in $\C_0(\GG)\tens\sA$ (Podle\'s condition). We will always assume that $\alpha$ is injective. If $\GG$ happens to be a classical locally compact group (i.e.~the \cst-algebra $\C_0(\GG)$ is commutative) then out of $\alpha$ one can produce an actual action of $\GG$ on $\sA$ by $*$-automorphisms via $\GG\ni g\mapsto\alpha_g\in\operatorname{Aut}(\sA)$, where $\alpha_g=(\operatorname{ev}_{g^{-1}}\tens\id)\comp\alpha$. The inverse in this formula is needed, so that $g\mapsto\alpha_g$ is multiplicative. We will stick to this convention throughout the paper.

If $\GG=\D(\HH)$ is the Drinfeld double of a locally compact quantum group $\HH$ (\cite[Section 4]{PodlesWoronowiczQLG}, \cite{Yamanouchi}, \cite[Section 8]{mnw}, \cite[Section 8]{BaajVaes}), then action of $\D(\HH)$ on $\sA$ amounts to a pair of actions, namely of $\HH$ and of $\hh{\HH}$ on $\sA$ satisfying the Yetter-Drinfeld condition (\cite[Proposition 3.2]{NestVoigt}). We will be interested in the case of $\GG=\D(\TT)$ in which this condition simplifies: the action $\brho^{\sA}\in\Mor(\sA,\D(\TT)\tens\sA)$ corresponds to a pair of commuting actions $\rho^{\sA}\in\Mor(\sA,\C(\TT)\tens\sA)$ and $\tilde{\rho}^{\sA}\in\Mor(\sA,\mathrm{c}_0(\ZZ)\tens\sA)$. This (bijective) correspondence is given by restriction of the action of $\D(\TT)=\TT\times\ZZ$ to the closed subgroups $\TT,\ZZ\subseteq\D(\TT)$ in one direction, and by equation $\brho^{\sA}=(\id\tens\tilde{\rho}^{\sA})\comp\rho^{\sA}$ in the other. We will use the same notational convention (of $\brho$, $\rho$, and $\tilde{\rho}$) throughout the paper.

\subsection{Further notational conventions}\label{sect-notation}
Whenever $X$ and $Y$ are subspaces of an algebra, the symbol $XY$ will denote $XY=\operatorname{span}\{xy\,|\,x\in X,\:y\in Y\}$.

In the first paper \cite{ExampleBraided} dealing with braided tensor products on \cst-algebra level (referred to ``crossed products'' in that article) written by S.L.~Woronowicz, a central role was played by the inclusion morphisms of the factor \cst-algebras into the braided product. The notation used in that paper was $j_1\in\Mor(\sA,\sA\tens_{\scriptscriptstyle\mathcal{C}}\!\sB)$ and $j_2\in\Mor(\sB,\sA\tens_{\scriptscriptstyle\mathcal{C}}\!\sB)$. The theory of braided tensor products of \cst-algebras was further developed in \cite{BraidedSU2,MeyerRoyWoronowiczI,MeyerRoyWoronowiczII} and \cite{DeCommerKrajczok}, where the notation for the braided product utilizes the symbol $\boxtimes$ (or $\boxtimes_\zeta$ in \cite{BraidedSU2}, see Section \ref{app-BraidedTP}) and variously $\iota_1$, $\iota_2$ etc.~or $\iota_A$, $\iota_B$ for the inclusion morphisms. The useful convention of denoting the inclusion morphism with a symbol which contains a symbol designating the \cst-algebra in question, like $\iota_\sA$, adopted e.g.~in \cite{MeyerRoyWoronowiczI,MeyerRoyWoronowiczII,DeCommerKrajczok} is not sufficiently precise when we are dealing with the braided tensor product of a given \cst-algebra $\sA$ with itself: $\sA\boxtimes\sA$. In that case the notation $j_1$ and $j_2$ seems more appropriate, but it becomes increasingly clumsy when multiple braided products are considered. The solution of using notation of the form $\iota_{\sA,1}$, $\iota_{\sB,3}$ etc.~(e.g.~to denote inclusion morphisms from $\sA$ and $\sB$ into $\sA\boxtimes\sB\boxtimes\sB$) solves the notational dilemmas, but is rather heavy and most of the time unnecessary. We propose to pass seamlessly between the convention of including the symbol for the \cst-algebra in the notation and only indicating its position in the braided product depending on the context. Although this is not a coherent notational scheme, we feel it will not cause any confusion to the reader, as the context always clearly explains which algebra is being included at what position in the braided tensor product. Sometimes (e.g.~in the proof of Proposition \ref{prop_BBDel}) we will use a combination of these two variants of notation simultaneously.

\section{The braided quantum group \texorpdfstring{$\SU_q(2)$}{SUq(2)}}\label{sect-SUq2}

\subsection{Definition}
For $q\in\CC\setminus\{0\}$ the (\cst-algebraic) braided compact quantum group $\SU_q(2)$ was introduced in \cite{BraidedSU2}: the \cst-algebra $A$ is defined as the universal unital \cst-algebra generated by $\alpha,\gamma$ such that
\begin{equation}\label{eq-SUq2rels}
\begin{aligned}
\alpha^*\alpha+\gamma^*\gamma&=\I,&\alpha\gamma&=\overline{q}\gamma\alpha,\\
\alpha\alpha^*+|q|^2\gamma^*\gamma&=\I,&\gamma\gamma^*&=\gamma^*\gamma.
\end{aligned}
\end{equation}
or, equivalently, the matrix $\left[\begin{smallmatrix}\alpha&-q\gamma^*\\\gamma&\alpha^*\end{smallmatrix}\right]$ is unitary (note that the relation $\alpha\gamma^*=q\gamma^*\alpha$ follows from the remaining ones, see e.g.~\cite[Remark 3.1]{BraidedPodlesSpheres}). The \cst-algebra $A$ could also be denoted by the symbol $\C(\SU_q(2))$, but we will stick to $A$ in this paper to lighten the notation. Before defining the comultiplication we need to introduce the action $\rho^A$ of $\TT$ on $A$:
\begin{equation}\label{eq-eq5}
\rho^A(\alpha)=\I\tens\alpha,\quad
\rho^A(\gamma)=\bz\tens\gamma
\end{equation}
where $\bz$ is the canonical generator of $\C(\TT)$. This way $A$ becomes an object of the category of \cst-algebras endowed with an action of $\TT$ (and equivariant morphisms of \cst-algebras) on which the \emph{braided tensor product} is defined. This construction depends on the choice of number $\zeta\in \TT$ -- we discuss this notion and its more general version in detail in Appendix \ref{app-BraidedTP}. At this point let us use the simplified approach to the braided tensor product as presented in \cite[Section 1]{BraidedSU2}: first we put $\zeta=q/\overline{q}$ (where $q$ is the deformation parameter in \eqref{eq-SUq2rels}). Then, given two \cst-algebras $\sX$ and $\sY$ equipped with actions $\rho^\sX$ and $\rho^\sY$ of $\TT$, we define $\iota_{\sX}\in\Mor(\sX,\C(\TT^2_\zeta)\tens\sX\tens\sY)$ and $\iota_{\sY}\in\Mor(\sY,\C(\TT^2_\zeta)\tens\sX\tens\sY)$ by
\[
\iota_{\sX}(x)=U^n\tens x\tens\I,\quad
\iota_{\sY}(y)=V^m\tens\I\tens y,
\]
where $x\in\sX$ and $y\in\sY$ are homogeneous elements of degrees $n$ and $m$ respectively ($\rho^\sX(x)=\bz^n\tens x$, $\rho^\sY(y)=\bz^m\tens y$) and $\C(\TT^2_\zeta)$ is the quantum torus:
\[
\C(\TT^2_\zeta)=\cst\bigl\langle U,V,\I\,\bigr|\bigl.\,U^*U=UU^*=V^*V=VV^*=\I,\:UV=\zeta VU\bigr\rangle.
\]
Then the \emph{braided tensor product} $\sX\boxtimes\sY$ of $\sX$ and $\sY$ is defined as the closed linear span of elements of the form $\iota_\sX(x)\iota_\sY(y)$ in $\M(\C(\TT^2_\zeta)\tens\sX\tens\sY)$, in symbols $\sX\boxtimes\sY=\overline{\iota_\sX(\sX)\iota_\sY(\sY)}$. Having defined $\sX\boxtimes\sY$ we consider $\iota_\sX$ and $\iota_\sY$ as morphisms from $\sX$ and $\sY$ to $\sX\boxtimes\sY$ respectively. By definition
\[
\iota_{\sX}(x)\iota_{\sY}(y)=\zeta^{\deg(x)\deg(y)} \iota_{\sY}(y)\iota_{\sX}(x)
\]
for homogeneous $x\in\sX$ and $y\in\sY$ of degrees respectively $\deg(x),\deg(y)\in\ZZ$.

The comultiplication on $A$ is the unital $*$-homomorphism $\Delta_A$ from $A$ to $A\boxtimes A$ defined by
\[
\Delta_A(\alpha)=\iota_1(\alpha)\iota_2(\alpha)-q\iota_1(\gamma^*)\iota_2(\gamma),\quad
\Delta_A(\gamma)=\iota_1(\gamma)\iota_2(\alpha)+\iota_1(\alpha^*)\iota_2(\gamma).
\]
As explained in Section \ref{sect-notation}, we have switched to the notation $\iota_1$, $\iota_2$ as in the case of braided tensor product $A\boxtimes A$ having the same factors, the symbol $\iota_A$ becomes ambiguous. The braided tensor product is a bifunctor on the category of $\TT$-\cst-algebras and hence we can take braided tensor product of (equivariant) morphisms (we can take several versions of morphisms, in particular equivariant c.p.~maps -- see \cite[Section 5.2]{MeyerRoyWoronowiczI} and \cite[Proposition 5.1]{DeCommerKrajczok}). This allows us to state that $\Delta_A$ is coassociative: $(\Delta_A\boxtimes\id)\comp\Delta_A=(\id\tens\Delta_A)\comp\Delta_A$. Moreover the ``cancellation laws'' hold, i.e.~$\overline{\iota_1(A)\Delta_A(A)}=A\boxtimes A=\overline{\Delta_A(A)\iota_2(A)}$.

A further study of the braided quantum group $\SU_q(2)$ including description of some of its representations, actions and homogeneous spaces is contained in \cite{BraidedPodlesSpheres}.

\subsection{Bosonization}\label{sect-boson}

The authors of \cite{BraidedSU2} define the bosonization of $\SU_q(2)$ as the compact quantum groups described by $(B,\Delta_B)$, where $B$ is universal unital \cst-algebra $B$ with generators $\alpha,\gamma,z$ where $\alpha,\gamma$ satisfy relations \eqref{eq-SUq2rels} together with
\[
z\alpha z^*=\alpha,\quad
z\gamma z^*=\zeta^{-1}\gamma,\quad
zz^*=z^*z=\I
\]
and $\Delta_B$ acts as follows:
\begin{equation}\label{eq-eq32}
\Delta_B(\alpha)=\alpha\tens\alpha-q\gamma^*z\tens\gamma,\quad
\Delta_B(\gamma)=\gamma\tens\alpha+\alpha^*z\tens\gamma,\quad
\Delta_B(z)=z\tens z.
\end{equation}
We will discuss an alternative approach to the bosonization in Section \ref{sect-BosonBoson}.

In fact, as noted in \cite[Section 6]{BraidedSU2}, the compact quantum group described by $(B,\Delta_B)$ is the co-opposite of the compact quantum group $\mathrm{U}_q(2)$ studied in \cite{ZhangZhao} via $\alpha\mapsto a^*$, $\gamma\mapsto b^*$, $z\mapsto D^*$.

Let us denote by $\kmap$ the obvious map $A=\C(\SU_q(2))\to B$. Furthermore let $\bn$ be the canonical (unbounded) generator of $\mathrm{c}_0(\ZZ)$, i.e.~the function $\ZZ\ni n\mapsto n\in\CC$.

\begin{proposition}\label{prop-injkap}
The map $\kmap$ is injective.
\end{proposition}

\begin{proof}
We have a natural isomorphism $B\simeq \ZZ\ltimes A$, for the action $\tilde{\rho}^{A}$ of $\ZZ$ on $A$ given by $\alpha\mapsto \I\tens\alpha$, $\gamma\mapsto\zeta^\bn\tens\gamma$. Indeed $\ZZ\ltimes A=\overline{\tilde{\rho}^A(A)(\C(\TT)\tens\I)}$ is generated by $\mathsf{x}=\tilde{\rho}^A(\alpha)$, $\mathsf{y}=\tilde{\rho}^A(\gamma)$ and $\mathsf{z}=(z\tens\I)$ and it is the universal \cst-algebra for the relations
\begin{align*}
\mathsf{x}^*\mathsf{x}+\mathsf{y}^*\mathsf{y}&=\I,&\mathsf{x}\mathsf{y}&=\overline{q}\mathsf{y}\mathsf{x},
&\mathsf{x}\mathsf{x}^*+|q|^2\mathsf{y}^*\mathsf{y}&=\I,&\mathsf{y}\mathsf{y}^*&=\mathsf{y}^*\mathsf{y},\\
\mathsf{z}\mathsf{x}\mathsf{z}^*&=\mathsf{x},
&\mathsf{z}\mathsf{y}\mathsf{z}^*&=\zeta^{-1}\mathsf{y},
&\mathsf{z}\mathsf{z}^*&=\I,&\mathsf{z}^*\mathsf{z}&=\I
\end{align*}
because $\ZZ$ is amenable (\cite[Theorem 7.13]{Williams}). Under this isomorphism $\kmap(A)\subseteq B$ corresponds to $A\subseteq \ZZ\ltimes A$ and $z$ corresponds to the generator $\lambda(1)\in\cst_{\text{r}}(\ZZ)\subseteq \ZZ\ltimes A$. It follows that $\kmap$ is injective.
\end{proof}

In what follows we will denote by $\cA$ the algebra generated by $\alpha,\alpha^*,\gamma$ and $\gamma^*$ inside $A$. This is a dense unital $*$-subalgebra and we have

\begin{corollary}\label{cor-bazacA}
The system
\begin{equation}\label{eq-bazacA}
\bigl\{\alpha^n\gamma^m{\gamma^*}^k\,\bigr|\bigl.\,n,m,k\in\ZZ_+\bigr\}\cup\bigl\{{\alpha^*}^n\gamma^m{\gamma^*}^k\,\bigr|\bigl.\,n\in\NN,\:m,k\in\ZZ_+\bigr\}
\end{equation}
is a basis of $\cA$.
\end{corollary}

\begin{proof}
Since $\kmap\colon A\to B$ is injective and the system \eqref{eq-bazacA} is linearly independent in $B\simeq\C(\mathrm{U}_q(2))$ (\cite[Theorem 2.2]{ZhangZhao}), it is also linearly independent in $A$. Furthermore it is easy to see that the space spanned by \eqref{eq-bazacA} is invariant under taking adjoints and multiplication by $\alpha$ and $\gamma$ which ends the proof.
\end{proof}

There is an action of $\TT$ on $B$ given by
\[
\rho^B(\alpha)=\I\tens\alpha,\quad
\rho^B(\gamma)=\bz\tens\gamma,\quad
\rho^B(z)=\bz\tens z,
\]
which has the following relation to the comultiplication $\Delta_B$: $(\rho^B\tens\id)\comp\Delta_B=(\id\tens\Delta_B)\comp\rho^B$. It follows from this that if $\bh_B$ denotes the Haar measure of the quantum group described by $(B,\Delta_B)$ then $\bh_B$ is $\TT$-invariant: since
\[
(\id\tens\bh_B\tens\id)(\rho^B\tens\id)\Delta_B(b)=(\id\tens\bh_B\tens\id)(\id\tens\Delta_B)\rho^B(b)
=\bigl((\id\tens\bh_B)\rho^B(b)\bigr)\tens\I
\]
for any $b\in B$, for any state $\omega\in\C(\TT)^*$ we have
\[
\bigl(((\omega\tens\bh_B)\comp\rho^B)\tens\id\bigr)\Delta_B(b)
=\bigl((\omega\tens\bh_B)\rho^B(b)\bigr)\I,\qquad b\in B,
\]
so the state $(\omega\tens\bh_B)\comp\rho^B$ is right invariant. By uniqueness of the Haar measure $(\omega\tens\bh_B)\comp\rho^B=\bh_B$ (for all $\omega$) and hence $(\id\tens\bh_B)\rho^B(b)=\bh_B(b)\I$ for all $b\in B$. Also by \cite[Theorem 2.8]{GuinSaurabh} the Haar measure $\bh_B$ is faithful if $q$ is not a root of unity. We will from now on assume that $0<|q|<1$.

The modular group $\sigma^{\bh_B}$ of the Haar measure can be calculated based on the results of \cite{GuinSaurabh} and the isomorphism $B\to\C(\mathrm{U}_q(2)^{\operatorname{op}})$ taking $\alpha$ to $a^*$, $\gamma$ to $b^*$, and $z$ to $D^*$ (see above):
\[
\sigma^{\bh_B}_t(\alpha)=|q|^{-2\ii t}\alpha,\quad
\sigma^{\bh_B}_t(\gamma)=\gamma,\quad
\sigma^{\bh_B}_t(z)=z,\qquad t\in\RR.
\]

\section{Further results}

\subsection{The Haar measure}\label{sect-HaarMeasure}

We define a state $\bh$ on $A$ by $\bh=\bh_B\comp\kmap$, where $\kmap$ is the inclusion $A\hookrightarrow B$ discussed in Subsection \ref{sect-boson}. The state $\bh$ is faithful. Also, the modular group $\sigma^{\bh_B}$ restricts to a point-norm continuous group of automorphism of $A$ considered as a subalgebra of $B$ (via the embedding $\kmap$) which clearly satisfies the KMS condition. Explicitly, the modular group of $\bh$ is given by
\[
\sigma^{\bh}_t(\alpha)=|q|^{-2\ii t}\alpha,\quad
\sigma^{\bh}_t(\gamma)=\gamma
,\qquad t\in\RR.
\]
It follows that $\bh$ extends to a faithful, normal state on the von Neumann algebra $\Linf(\SU_q(2))$ generated by the image of the GNS representation of $A$ with respect to $\bh$. We denote it by $M$ in order to lighten the notation.

\begin{proposition}
The action of $\TT$ on $A$ extends to a point-weak${}^*$ continuous action on $M$.
\end{proposition}

\begin{proof}
For each $\mu\in\TT$ the automorphism $\rho^A_\mu$ preserves $\bh$, hence it extends to unital, normal $*$-homomorphism $\rho^M_\mu\colon M\to M$. One easily checks that this gives a group action and $\bh$ (extended to a state on $M$) is $\rho^M$-invariant.

Point-weak${}^*$ continuity follows by density. Indeed, take $x\in M$, $\omega\in M_*$, and let $\mu_n\xrightarrow[n\to\infty]{}\mu$ in $\TT$. Next given $\eps>0$, by a standard approximation argument, we can find $y,y'\in\cA$ such that $\|\omega-\omega_{\Lambda_\bh(y),\Lambda_\bh(y')}\|\leq\eps$, where $\Lambda_\bh$ is the GNS map for $\bh$. Then using the fact that $\bh$ is $\rho^M$-invariant we have
\begin{align*}
\Bigl|\bigl\langle\rho^M_\mu(x)-\rho^M_{\mu_n}(x),\omega\bigr\rangle\Bigr|
&\leq 2\|x\|\eps+\Bigl|\bigl\langle\rho^M_\mu(x)-\rho^M_{\mu_n}(x),\omega_{\Lambda_\bh(y),\Lambda_\bh(y')}\bigr\rangle\Bigr|\\
&=2\|x\|\eps+\Bigl|\bh\Bigl(y^*\bigl(\rho^M_\mu(x)-\rho^M_{\mu_n}(x)\bigr)y'\Bigr)\Bigr|\\
&=2\|x\|\eps+\Bigr|\bh\Bigr(\sigma^\bh_\ii(y')y^*\bigl(\rho^M_\mu(x)-\rho^M_{\mu_n}(x)\bigr)\Bigr)\Bigr|\\
&=2\|x\|\eps+\biggl|\bh\biggl(\Bigl(
\rho^M_{\overline{\mu}}\bigl(\sigma^\bh_\ii(y')y^*\bigr)-
\rho^M_{\overline{\mu_n}}\bigl(\sigma^\bh_\ii(y')y^*\bigr)\Bigr)x\biggr)\biggr|,
\end{align*}
but $\sigma^\bh_\ii(y')y^*\in\cA\subseteq A$ and hence $\mu'\mapsto\rho^A_{\mu'}(\sigma^\bh_\ii(y')y^*)$ is norm continuous and we obtain
\[
\limsup_{n\to\infty}\bigl|\langle\rho^M_\mu(x)-\rho^M_{\mu_n}(x),\omega\rangle\bigr|\leq 2\|x\|\eps.\qedhere
\]
\end{proof}

We will now prove that $\bh$ is a Haar measure on the braided quantum group $\SU_q(2)$. This question was addressed in the preliminary version \cite{BraidedSU2-arXiv} of \cite{BraidedSU2}. Our proof follows the same lines as that presented in \cite{BraidedSU2-arXiv}. In addition, in Section \ref{sect-Antipode}, we give a new argument for right invariance of $\bh$ (see Proposition \ref{prop-hRightInv}) which together with Proposition \ref{prop-hLeftInv} gives an alternative proof of Corollary \ref{cor-hHaar}

We begin with Corollary \ref{cor-idbth} which says that the braided slice maps $\bh\boxtimes\id$ and $\id\boxtimes\bh$ are well defined normal u.c.p.~maps $M\,\overline{\boxtimes}\,M\to M$. The fact that $\bh$ is the Haar measure of $\SU_q(2)$ is expressed as
\begin{equation}\label{eq-HaarMeasure}
(\bh\boxtimes\id)\Delta(a)=(\id\boxtimes\bh)\Delta(a)=\bh(a)\I,\qquad a\in M.
\end{equation}
The slice maps restrict to equivariant u.c.p.~maps $A\boxtimes A\to A$ (cf.~also \cite[Section 5.2]{MeyerRoyWoronowiczI}). Note that the von Neumann algebraic formulation of the fact that $\bh$ is the Haar measure of $\SU_q(2)$ given by \eqref{eq-HaarMeasure} is equivalent to the analogous statement on the \cst-algebra level thanks to normality of the maps involved.

\begin{proposition}\label{prop-hLeftInv}
The state $\bh$ is left invariant, i.e.~$(\id\boxtimes\bh)\Delta(a)=\bh(a)\I$ for all $a\in M$.
\end{proposition}

\begin{proof}
Define
\[
\psi\colon A\boxtimes A\longrightarrow B\tens B
\]
as the composition
\[
\xymatrix{
A\boxtimes A\ar[r]
&A\boxtimes A\boxtimes\C(\TT)\ar[r]^(.32){\Psi}
&\bigl(A\boxtimes\C(\TT)\bigr)\tens\bigl(A\boxtimes\C(\TT)\bigr)
=B\tens B,
}
\]
where the first map is the canonical inclusion and $\Psi$ is introduced in Proposition \ref{prop-prop2Cstar}. Observe that
\[
\psi\comp\Delta_A=\Delta_B\comp\iota_A.
\]
Indeed, this follows from the calculations:
\begin{align*}
\psi\bigl(\Delta_A(\alpha)\bigr)
&=\psi\bigl(\iota_1(\alpha)\iota_2(\alpha)-q\iota_1(\gamma)^*\iota_2(\gamma)\bigr)\\
&=\bigl(\iota_A(\alpha)\tens\I\bigr)\bigl((\iota_{\C(\TT)}\tens\iota_A)\rho^A(\alpha)\bigr)
-q(\iota_A(\gamma^*)\tens\I)\bigl((\iota_{\C(\TT)}\tens\iota_A)\rho^A(\gamma)\bigr)\\
&=\iota_A(\alpha)\tens\iota_A(\alpha)
-q\iota_A(\gamma^*)\iota_{\C(\TT)}(\bz)\tens\iota_A(\gamma)=\Delta_B\bigl(\iota_A(\alpha)\bigr)
\end{align*}
and
\begin{align*}
\psi\bigl(\Delta_A(\gamma)\bigr)
&=\psi\bigl(\iota_1(\gamma)\iota_2(\alpha)+\iota_1(\alpha^*)\iota_2(\gamma)\bigr)\\
&=\bigl(\iota_A(\gamma)\tens\I\bigr)\bigl((\iota_{\C(\TT)}\tens\iota_A)\rho^A(\alpha)\bigr)
+\bigl(\iota_A(\alpha^*)\tens\I\bigr)\bigl((\iota_{\C(\TT)}\tens\iota_A)\rho^A(\gamma)\bigr)\\
&=\iota_A(\gamma)\tens\iota_A(\alpha)+\iota_A(\alpha^*)\iota_{\C(\TT)}(\bz)\tens\iota_A(\gamma)
=\Delta_B\bigl(\iota_A(\gamma)\bigr).
\end{align*}
It follows that for any $a\in A$
\begin{equation}\label{eq-eq34}
(\id\tens\bh_B)\psi\bigl(\Delta_A(a)\bigr)
=(\id\tens\bh_B)\Delta_B\bigl(\iota_A(a)\bigr)
=\bh_B\bigl(\iota_A(a)\bigr)\I=\bh(a)\I.
\end{equation}
Next we claim that
\begin{equation}\label{eq-eq36}
(\id\tens\bh_B)\comp\psi=\iota_A\comp(\id\boxtimes\bh)
\end{equation}
as maps $A\boxtimes A\to B$. It is enough to check equality \eqref{eq-eq36} on elements of the form $\iota_1(x)\iota_2(y)$ with $x,y$ are homogeneous of degrees $\deg(x),\deg(y)\in\ZZ$. Then we have
\begin{align*}
(\id\tens\bh_B)\psi\bigl(\iota_1(x)\iota_2(y)\bigr)
&=(\id\tens\bh_B)
\bigl((\iota_A(x)\tens\I)(\iota_{\C(\TT)}\tens\iota_A)\rho^A(y)\bigr)\\
&=(\id\tens\bh_B)\bigl(\iota_A(x)\iota_{\C(\TT)}(\bz^{\deg(y)})\tens\iota_A(y)\bigr)\\
&=\delta_{\deg(y),0}\,\iota_A(x)\bh(y)
=\iota_A(\id\boxtimes\bh)\bigl(\iota_1(x)\iota_2(y)\bigr)
\end{align*}
since both $\bh$ and $\bh_B$ are invariant under the actions of $\TT$. This proves equality \eqref{eq-eq36}. Consequently, using \eqref{eq-eq36} and \eqref{eq-eq34}
\[
\iota_A\bigl((\id\boxtimes\bh)\Delta_A(a)\bigr)
=(\id\tens\bh_B)\psi\bigl(\Delta_A(a)\bigr)
=\bh(a)\I,\qquad a\in A
\]
and since $\iota_A$ is injective, we obtain $(\id\boxtimes\bh)\Delta_A(a)=\bh(a)\I$ for all $a\in A$.
\end{proof}

\begin{corollary}\label{cor-hHaar}
The state $\bh$ is the Haar measure of $\SU_q(2)$.
\end{corollary}

\begin{proof}
As observed in \cite[Proof of Proposition 7.2]{BraidedSU2-arXiv} the proof of existence of the Haar measure on second countable compact quantum groups (\cite[Section 3]{CQGs}) carries over to the braided case (one only needs to change tensor products to braided tensor products and use appropriate embeddings instead of maps of the form $a\mapsto a\tens\I$ or $a\mapsto\I\tens a$). Hence the Haar measure $\widetilde{\bh}$ on $\SU_q(2)$ exists and it satisfies $\omega\comp(\widetilde{\bh}\boxtimes\id)\comp\Delta_A=\widetilde{\bh}$ for any state $\omega\in A^*$. Taking $\omega=\bh$ and using $\bh\comp(\widetilde{\bh}\boxtimes\id)=\widetilde{\bh}\comp(\id\boxtimes \bh)$ yields $\bh=\widetilde{\bh}$.
\end{proof}

\subsection{The scaling group}\label{sect-tau}

The name \emph{scaling group} of a locally compact quantum group $\GG$ comes from the fact that this one-parameter group of automorphisms of $\C_0(\GG)$ (or $\Linf(\GG)$) scales the Haar weights by a constant. The group itself is introduced in terms of the polar decomposition of the antipode (\cite[Section 5]{lcqg}, \cite{mnw}). In case $\GG$ is compact, the scaling is actually trivial and the name does not reflect the nature of this object. In this section we will define the appropriate analog of the scaling group which will, in particular, leave the Haar measure invariant as in the case of non-braided compact quantum groups. The role of the scaling group in the polar decomposition of the antipode will be described in Section \ref{sect-UApolar}.

The scaling group of the quantum group $\SU_q(2)$ (for real $q$) is well known to be defined by
\begin{equation}\label{eq-eq62}
\tau_t(\alpha)=\alpha,\quad\tau_t(\gamma)=|q|^{2\ii t}\gamma,\qquad t\in\RR
\end{equation}
(cf.~e.g..~\cite[Example 1.7.8]{NeshveyevTuset}). We will use the same formulas to introduce this group in the case of complex $q$ (satisfying $0<|q|<1$).

\begin{proposition}\label{prop-tau}
There is a point-norm continuous one-parameter group $\tau=\{\tau_t\}_{t\in\RR}$ of automorphisms of $A$ such that
\begin{enumerate}
\item for any $t\in\RR$ we have $\tau_t(\alpha)=\alpha$ and $\tau_t(\gamma)=|q|^{2\ii t}\gamma$,
\item each automorphism $\tau_t$ globally preserves the subalgebra $\cA$,
\item we have $\tau_t=\rho^A_{|q|^{-2\ii t}}$ for all $t\in\RR$, in particular $\tau_t$ is $\TT$ and $\D(\TT)$-equivariant for each $t$,
\item for all $t\in\RR$ we have $\bh\comp\tau_t=\bh$ and hence $\tau$ extends to a point-weak${}^*$ continuous one-parameter group of automorphisms of $M$,
\item $\tau_t$ and $\sigma^\bh_s$ commute for all $s,t\in\RR$,
\item\label{prop-tau1} we have
\begin{subequations}
\begin{align}
\Delta_A\comp\tau_t&=(\tau_t\boxtimes\tau_t)\comp\Delta_A
=(\sigma^\bh_t\boxtimes\sigma^\bh_{-t})\comp\Delta_A,\label{eq-prop-tau1}\\
\Delta_A\comp\sigma^\bh_t&=(\tau_t\boxtimes\sigma^\bh_t)\comp\Delta_A
=(\sigma^\bh_{t}\boxtimes\tau_{-t})\comp\Delta_A.\label{eq-prop-tau2}
\end{align}
\end{subequations}
\end{enumerate}
\end{proposition}

\begin{proof}
Existence of the automorphisms $\{\tau_t\}_{t\in\RR}$ of $A$ satisfying \eqref{eq-eq62}, the point-norm continuity of the group and the fact that it preserves the subalgebra $\cA$ are standard. Also it is clear from the definition that $\tau_t=\rho^A_{|q|^{-2\ii t}}$ for all $t$ and then existence of point-weak${}^*$ continuous extension to $M$ follows from general principles: $\tau_t$ is spatially implemented by the unitary $\Lambda_\bh(a)\mapsto\Lambda_\bh(\tau_t(a))$. Finally the fact that the groups $\tau$ and $\sigma^\bh$ pointwise commute is also obvious.

To prove the formulas in point \eqref{prop-tau1} it is enough to check them on generators. Note that they are well defined -- automorphisms $\tau_t$ and $\sigma^\bh_t$ are equivariant. For $t\in\RR$ we have
\[
\resizebox{\textwidth}{!}{\ensuremath{\displaystyle
\begin{aligned}
\Delta_A\bigl(\tau_t(\alpha)\bigr)
&=\Delta_A(\alpha)=\iota_1(\alpha)\iota_2(\alpha)-q\iota_1(\gamma^*)\iota_2(\gamma),\\
(\tau_t\boxtimes\tau_t)\Delta_A(\alpha)
&=\iota_1\bigl(\tau_t(\alpha)\bigr)\iota_2\bigl(\tau_t(\alpha)\bigr) -q\iota_1\bigl(\tau_t(\gamma)^*\bigr)\iota_2\bigl(\tau_t(\gamma)\bigr)
=\iota_1(\alpha)\iota_2(\alpha)-q\iota_1(\gamma^*)\iota_2(\gamma),\\
(\sigma^\bh_t\boxtimes\sigma^\bh_{-t})\Delta_A(\alpha)
&=\iota_1\bigl(\sigma^\bh_t(\alpha)\bigr)\iota_2\bigl(\sigma^\bh_{-t}(\alpha)\bigr)
-q\iota_1\bigl(\sigma^\bh_t(\gamma)^*\bigr)\iota_2\bigl(\sigma^\bh_{-t}(\gamma)\bigr)
=\iota_1(\alpha)\iota_2(\alpha)-q\iota_1(\gamma^*)\iota_2(\gamma)
\end{aligned}
}}
\]
and
\[
\resizebox{\textwidth}{!}{\ensuremath{\displaystyle
\begin{aligned}
\Delta_A\bigl(\tau_t(\gamma)\bigr)
&=|q|^{2\ii t}\bigl(\iota_1(\gamma)\iota_2(\alpha)+\iota_1(\alpha^*)\iota_2(\gamma)\bigr),\\
(\tau_t\boxtimes\tau_t)\Delta_A(\gamma)
&=\iota_1\bigl(\tau_t(\gamma)\bigr)\iota_2\bigl(\tau_t(\alpha)\bigr)
+\iota_1\bigl(\tau_t(\alpha)^*\bigr)\iota_2\bigl(\tau_t(\gamma)\bigr)
=|q|^{2\ii t}\bigl(\iota_1(\gamma)\iota_2(\alpha) + \iota_1(\alpha^*)\iota_2(\gamma)\bigr),\\
(\sigma^\bh_t\boxtimes\sigma^\bh_{-t})\Delta_A(\gamma)
&=\iota_1\bigl(\sigma^\bh_t(\gamma)\bigr)\iota_2\bigl(\sigma^\bh_{-t}(\alpha)\bigr)
+\iota_1\bigl(\sigma^\bh_t(\alpha)^*\bigr)\iota_2\bigl(\sigma^\bh_{-t}(\gamma)\bigr)
=|q|^{2\ii t}\bigl(\iota_1(\gamma)\iota_2(\alpha) + \iota_1(\alpha^*)\iota_2(\gamma)\bigr)
\end{aligned}
}}
\]
which proves \eqref{eq-prop-tau1}. Next
\[
\resizebox{\textwidth}{!}{\ensuremath{\displaystyle
\begin{aligned}
\Delta_A\bigl(\sigma^\bh_t(\alpha)\bigr)
&=|q|^{-2\ii t}\bigl(\iota_1(\alpha)\iota_2(\alpha) -q\iota_1(\gamma^*)\iota_2(\gamma)\bigr),\\
(\tau_t\boxtimes\sigma^\bh_t)\Delta_A(\alpha)
&=\iota_1\bigl(\tau_t(\alpha)\bigr)\iota_2\bigl(\sigma^\bh_t(\alpha)\bigr) -q\iota_1\bigl(\tau_t(\gamma)^*\bigr)\iota_2\bigl(\sigma^\bh_t(\gamma)\bigr)
=|q|^{-2\ii t}\bigl(\iota_1(\alpha)\iota_2(\alpha) -q\iota_1(\gamma^*)\iota_2(\gamma)\bigr),\\
(\sigma^\bh_t\boxtimes\tau_{-t})\Delta_A(\alpha)
&=\iota_1\bigl(\sigma^\bh_t(\alpha)\bigr)\iota_2\bigl(\tau_{-t}(\alpha)\bigr) -q\iota_1\bigl(\sigma^\bh_t(\gamma)^*\bigr)\iota_2\bigl(\tau_{-t}(\gamma)\bigr)
=|q|^{-2\ii t}\bigl(\iota_1(\alpha)\iota_2(\alpha) -q\iota_1(\gamma^*)\iota_2(\gamma)\bigr)
\end{aligned}
}}
\]
and
\[
\resizebox{\textwidth}{!}{\ensuremath{\displaystyle
\begin{aligned}
\Delta_A\bigl(\sigma^\bh_t(\gamma)\bigr)
&=\iota_1(\gamma)\iota_2(\alpha)+\iota_1(\alpha^*)\iota_2(\gamma),\\
(\tau_t\boxtimes\sigma^\bh_t)\Delta_A(\gamma)
&=\iota_1\bigl(\tau_t(\gamma)\bigr)\iota_2\bigl(\sigma^\bh_t(\alpha)\bigr)
+\iota_1\bigl(\tau_t(\alpha)^*\bigr)\iota_2\bigl(\sigma^\bh_t(\gamma)\bigr)
=\iota_1(\gamma)\iota_2(\alpha)+\iota_1(\alpha^*)\iota_2(\gamma),\\
(\sigma^\bh_t\boxtimes\tau_{-t})\Delta_A(\gamma)
&=\iota_1\bigl(\sigma^\bh_t(\gamma)\bigr)\iota_2\bigl(\tau_{-t}(\alpha)\bigr)
+\iota_1\bigl(\sigma^\bh_t(\alpha)^*\bigr)\iota_2\bigl(\tau_{-t}(\gamma)\bigr)
=\iota_1(\gamma)\iota_2(\alpha)+\iota_1(\alpha^*)\iota_2(\gamma)
\end{aligned}
}}
\]
and \eqref{eq-prop-tau2} follows.
\end{proof}

\subsection{The antipode}

Recall that the $*$-algebra $\cA$ was defined in Section \ref{sect-SUq2} as the $*$-subalgebra of $A$ generated (algebraically) by $\alpha$ and $\gamma$. We first note the following fact:

\begin{proposition}\label{prop-lemma13}
$\cA$ is the universal unital algebra generated by $\alpha$, $\alpha^*$, $\gamma$ and $\gamma^*$ satisfying \eqref{eq-SUq2rels} and $\alpha\gamma^*=q\gamma^*\alpha$.
\end{proposition}

\begin{proof}
For real $q$ this is \cite[Theorem 1.3]{su2}. Denote by $\widetilde{\cA}$ the universal unital algebra generated by elements $\tilde{\alpha}$, $\tilde{\alpha}^*$, $\tilde{\gamma}$, $\tilde{\gamma}^*$ satisfying
\begin{align*}
\tilde{\alpha}^*\tilde{\alpha}+\tilde{\gamma}^*\tilde{\gamma}&=\I,&\tilde{\alpha}\tilde{\gamma}&=\overline{q}\tilde{\gamma}\tilde{\alpha},&\tilde{\alpha}\tilde{\gamma}^*&=q\tilde{\gamma}^*\tilde{\alpha},\\
\tilde{\alpha}\tilde{\alpha}^*+|q|^2\tilde{\gamma}^*\tilde{\gamma}&=\I,&\tilde{\gamma}\tilde{\gamma}^*&=\tilde{\gamma}^*\tilde{\gamma}.
\end{align*}
We have the canonical unital algebra homomorphism $\pi\colon\widetilde{\cA}\to\cA$ which is surjective. Take $x\in\ker{\pi}$. Then we can write
\[
x=\sum_{n,m,k\ge 0}x_{n,m,k}\tilde{\alpha}^n\tilde{\gamma}^m(\tilde{\gamma}^*)^k+
\sum_{n>0,\:m,k\ge 0}x_{n,m,k}(\tilde{\alpha}^*)^n\tilde{\gamma}^m(\tilde{\gamma}^*)^k
\]
for some $x_{n,m,k}\in\CC$ (we do not claim linear independence here) and consequently
\[
0=\pi(x)=
\sum_{n,m,k\ge 0}x_{n,m,k}\alpha^n\gamma^m(\gamma^*)^k+
\sum_{n>0,\:m,k\ge 0}x_{n,m,k}(\alpha^*)^n\gamma^m(\gamma^*)^k.
\]
But from Corollary \ref{cor-bazacA} we know that the system
\[
\bigl\{\alpha^n\gamma^m{\gamma^*}^k\,\bigr|\bigl.\,n,m,k\in\ZZ_+\bigr\}\cup\bigl\{{\alpha^*}^n\gamma^m{\gamma^*}^k\,\bigr|\bigl.\,n\in\NN,\:m,k\in\ZZ_+\bigr\}
\]
is a basis of $\cA$, which shows that $x_{n,m,k}=0$ for all $n\in\ZZ$, $m,k\in\ZZ_+$, and hence $x=0$.
\end{proof}

\subsubsection{Definition of the antipode}\label{sect-Antipode}

The goal of this section is to introduce the antipode of the braided quantum group $\SU_q(2)$. In analogy to the non-braided case we want to have $S\colon\cA\to\cA$ such that
\begin{enumerate}
\item\label{enum-S1} $S$ is unital, linear and $(S\tens\id)u=u^*$, where $u=\left[\begin{smallmatrix}\alpha&-q\gamma^*\\\gamma&\alpha^* \end{smallmatrix}\right]$ is the fundamental representation od $\SU_q(2)$ (cf.~Sections \ref{sect-SUq2}, \ref{sect-RepTh} and \cite[Section 5]{BraidedPodlesSpheres}),
\item\label{enum-S2} $S$ is braided-anti-multiplicative, i.e.~$S(ab)=\zeta^{-\deg(a)\deg(b)}S(b)S(a)$ for homogeneous $a,b\in\cA$ (see equation \eqref{eq-BrFlip}).
\end{enumerate}
We note that \eqref{enum-S1} forces
\begin{equation}\label{eq-eq41}
S(\alpha)=\alpha^*,\quad
S(\alpha^*)=\alpha,\quad
S(\gamma)=-\overline{q}\gamma,\quad
S(\gamma^*)=-\tfrac{1}{q}\gamma^*,
\end{equation}
while \eqref{enum-S2} is satisfied in algebraic setting \cite[Proposition 3.2.12]{HeckenbergerSchneider}.

\begin{theorem}\label{thm-prop6}
There exists a unique $S\colon\cA\to\cA$ which is unital, linear, preserves the degree, satisfies \eqref{eq-eq41} and $S(ab)=\zeta^{-\deg(a)\deg(b)} S(b)S(a)$ for homogeneous $a,b\in\cA$.
\end{theorem}

\begin{proof}
Uniqueness of $S$ is clear. Let $\cB$ be a vector space isomorphic to $\cA$ with isomorphism $\cA\ni a\mapsto\flat(a)\in \cB$ and let us equip $\cB$ with multiplication
\[
\flat(a)\flat(b)=
\zeta^{-\deg(a)\deg(b)}\flat(ba)
\]
for homogeneous $a,b\in\cA$. This makes $\cB$ into a unital associative algebra generated by elements $\bigl\{\flat(\alpha^*),\flat(\alpha),\flat(-\overline{q}\gamma),\flat(-q^{-1}\gamma^*)\bigr\}$ with unit $\flat(\I)$. These elements satisfy
\begin{align*}
\flat(\alpha)\flat(\alpha^*)+\flat(-q^{-1}\gamma^*)\flat(-\overline{q}\gamma)-\flat(\I)
&=\flat\bigl(\alpha^*\alpha +(\overline{q}/q)\zeta\gamma\gamma^*-\I\bigr)=0,\\
\flat(\alpha^*)\flat(\alpha)+|q|^2\flat(-\overline{q}\gamma)\flat(-q^{-1}\gamma^*)-\flat(\I)
&=\flat\bigl(\alpha\alpha^*+|q|^2(\overline{q}/q)\zeta\gamma^*\gamma-\I\bigr)=0,\\
\flat(-\overline{q}\gamma)\flat(-q^{-1} \gamma^*)-\flat(-q^{-1}\gamma^*)\flat(-\overline{q}\gamma)
&=(\overline{q}/q)\flat(\zeta\gamma^*\gamma-\zeta\gamma\gamma^*)=0,\\
\flat(\alpha^*)\flat(-\overline{q}\gamma)-\overline{q}\flat(-\overline{q}\gamma)\flat(\alpha^*)
&=-\overline{q}\flat(\gamma \alpha^*-\overline{q}\alpha^*\gamma)=0,\\
\flat(\alpha^*)\flat(-q^{-1}\gamma^*)-q\flat(-q^{-1}\gamma^*)\flat(\alpha^*)
&=-q^{-1}\flat(\gamma^*\alpha^*-q\alpha^*\gamma^*)=0.
\end{align*}

The universal property of $\cA$ (Proposition \ref{prop-lemma13}) gives us a linear, unital, multiplicative map $\tilde{S}\colon\cA\to\cB$ such that
\[
\tilde{S}(\alpha)=\flat(\alpha^*),\quad
\tilde{S}(\alpha^*)=\flat(\alpha),\quad
\tilde{S}(\gamma)=\flat(-\overline{q}\gamma),\quad
\tilde{S}(\gamma^*)=\flat(-q^{-1}\gamma^*).
\]
Now we define $S=\flat^{-1}\comp\tilde{S}\colon\cA\to\cA$ which, by its construction, satisfies \eqref{eq-eq41}. Let us check that $S$ is braided-anti-multiplicative: take homogeneous $a,b\in\cA$. Then
\begin{equation}\label{eq2}
\begin{aligned}
\flat\bigl(S(ab)\bigr)&=\tilde{S}(ab)=\tilde{S}(a)\tilde{S}(b)
=\flat\bigl((\flat^{-1}\comp\tilde{S})(a)\bigr)\flat\bigl((\flat^{-1}\comp\tilde{S})(b)\bigr)\\
&=\zeta^{-\deg((\flat^{-1}\comp\tilde{S})(a))\deg((\flat^{-1}\comp\tilde{S})(b))}
\flat\bigl((\flat^{-1}\comp\tilde{S})(b)(\flat^{-1}\comp\tilde{S})(a)\bigr)\\
&=\zeta^{-\deg(S(a))\deg(S(b))}
\flat\bigl(S(b)S(a)\bigr).
\end{aligned}
\end{equation}
To finish the proof, we need to show that $\deg(S(a))=\deg(a)$ for any homogeneous $a$. This is easily shown on the basis \eqref{eq-bazacA} of $\cA$ using equation \eqref{eq2}.
\end{proof}

We note a corollary of the above proof that $S$ defined in Theorem \ref{thm-prop6} is equivariant and preserves the degree.

In the following sections we will analyze $S$ and prove a number of its properties which include its relationship with representation theory of the braided compact quantum group $\SU_q(2)$ (Section \ref{sect-RepTh}) and the notion of a braided Hopf algebra (Section \ref{sect-BrHopf}).

\subsubsection{Braided flip}\label{sect-BrFlip}

We already observed in the Proof of Theorem \ref{thm-prop6} that the action of $\TT$ on $A$ considered in this paper restricts to an action on $\cA$. Clearly so does the action $\tilde{\rho}^\cA$ of $\ZZ$ given by restricting the action provided by Proposition \ref{prop-lemma1}. By definition $\tilde{\rho}^\cA_n(\alpha)=\alpha$ and $\tilde{\rho}^\cA_n(\gamma)=\zeta^{-n}\gamma$ for all $n\in\ZZ$ (cf.~Section \ref{sect-BrHopf}). With this action of $\ZZ$ the algebra $\cA$ becomes a Yetter-Drinfeld module over the group algebra $\CC[\ZZ]$ as in \cite[Definition 1.4.1]{HeckenbergerSchneider} (we describe this structure in more detail at the beginning of Section \ref{sect-BrHopf}). As such it comes equipped with additional structure such as the flip map $c_{\cA,\cA}\colon\cA\algtens\cA\to\cA\algtens\cA$ which satisfies
\[
c_{\cA,\cA}(a\tens b)=\tilde{\rho}^\cA_{\deg(a)}(b)\tens a
\]
for all $a,b\in\cA$ with $a$ homogeneous (\cite[Definition 1.4.10]{HeckenbergerSchneider}). In other words for homogeneous elements
\begin{equation}\label{eq-BrFlip}
c_{\cA,\cA}(a\tens b)=\zeta^{-\deg(a)\deg(b)}b\tens a.
\end{equation}
We will transfer this flip to the braided tensor product of $\cA$ with itself. To this end we will use the next lemma in which $\cA\algboxtimes\cA$ is defined to be the subspace of $A\boxtimes A$ spanned by $\bigl\{\iota_1(a)\iota_2(b)\,\bigr|\bigl.\,a,b\in\cA\bigr\}$.

\begin{proposition}\label{prop-lemma10}
There is a well defined linear isomorphism
\begin{equation}\label{eq42}
\phi_{\tens,\boxtimes}\colon\cA\algtens\cA\ni a\tens b\longmapsto\iota_1(a)\iota_2(b)\in\cA\algboxtimes\cA.
\end{equation}
\end{proposition}

\begin{proof}
A linear map \eqref{eq42} is well defined and unique by the universal property of the algebraic tensor product. It is also surjective, so we only need to prove that it is injective.

Take $x\in\ker{\phi_{\tens,\boxtimes}}$ and write $x=\sum\limits_{n,m=-\infty}^{\infty} x_{n,m}$ where
\[
x_{n,m}=\sum_{i=1}^{N_{n,m}}a_{n,m}^i\tens b_{n,m}^i
\]
for some $a_{n,m}^i,b_{n,m}^i\in\cA$ with $\deg(a_{n,m}^i)=n$ and $\deg(b_{n,m}^i)=m$.
Then viewing $A\boxtimes A$ as a subalgebra of $\C(\TT^2_\zeta)\tens A\tens A$ (Section \ref{sect-SUq2}) we have
\[
0=\phi_{\tens,\boxtimes}(x)
=\sum_{n,m=-\infty}^{\infty}\sum_{i=1}^{N_{n,m}}\iota_1(a^i_{n,m})\iota_2(b^i_{n,m})
=\sum_{n,m=-\infty}^{\infty}\sum_{i=1}^{N_{n,m}}U^nV^m\tens a^i_{n,m}\tens b^i_{n,m}.
\]
This implies $\sum\limits_{i=1}^{N_{n,m}}a^i_{n,m}\tens b^i_{n,m}=0$ for all $n,m\in\ZZ$, and hence $x_{n,m}=0$ for all $n,m$ which shows that $\phi_{\tens,\boxtimes}$ is injective.
\end{proof}

\begin{proposition}
There is a unique linear map $S\boxtimes S\colon\cA\algboxtimes\cA\to\cA\algboxtimes\cA$ such that
\[
(S\boxtimes S)\bigl(\iota_1(a)\iota_2(b)\bigr)
=\iota_1\bigl(S(a)\bigr)\iota_2\bigl(S(b)\bigr),\qquad a,b\in\cA.
\]
\end{proposition}

\begin{proof}
Uniqueness of $S\boxtimes S$ is obvious and its existence is proved by defining  $S\boxtimes S=\phi_{\tens,\boxtimes}\comp(S\tens S)\comp\phi_{\tens,\boxtimes}^{-1}$.
\end{proof}

\begin{definition}
We define the \emph{braided flip} $\chi^\boxtimes$ on $\cA\algboxtimes\cA$ by $\chi^\boxtimes=\phi_{\tens,\boxtimes}\comp c_{\cA,\cA}\comp\phi_{\tens,\boxtimes}^{-1}$.
\end{definition}

The mapping $\chi^\boxtimes\colon\cA\algboxtimes\cA\to\cA\algboxtimes\cA$ satisfies
\begin{equation}\label{eq1} 
\chi^\boxtimes\bigl(\iota_1(a)\iota_2(b)\bigr)=\zeta^{-\deg(a)\deg(b)}\iota_1(b)\iota_2(a)
=\iota_2(a)\iota_1(b)
\end{equation}
for all homogeneous $a,b\in\cA$. Braided flip is invertible with
\[
(\chi^{\boxtimes})^{-1} \bigl(\iota_1(a)\iota_2(b)\bigr)=\zeta^{\deg(a)\deg(b)}\iota_1(b)\iota_2(a).
\]

\begin{remark}
Formula \eqref{eq1} extends the braided flip $\chi^{\boxtimes}$ to a map on a larger space $A^{\operatorname{Pol}}\algboxtimes A^{\operatorname{Pol}}$, where $A^{\operatorname{Pol}}$ is the dense subspace of $A$ spanned by homogeneous elements. Since we will have no use for this map, we restrict ourselves to working with the smaller space $\cA$.
\end{remark}

Whether $\chi^\boxtimes$ extends to a completely bounded map on $A\boxtimes A$, depends on properties of $\zeta=q/\overline{q}\in \TT$.

\begin{proposition}\label{prop1}
\noindent
\begin{enumerate}
\item If $\zeta$ is a root of unity with $\zeta=\ee^{2\pi\ii n/d}$ for $n\in\ZZ_+$ and $d\in\NN$ with $0\le n\le d-1$ and $\gcd(n,d)=1$, then $\chi^{\boxtimes}$ extends to a completely bounded linear map $A\boxtimes A\to A\boxtimes A$ with the c.b.~norm not greater than $d^2$.
\item If $\zeta$ is not a root of unity, then $\chi^{\boxtimes}$ does not extend to a completely bounded map on $A\boxtimes A$.
\end{enumerate}
\end{proposition}

We delegate the proof of Proposition \ref{prop1} to Appendix \ref{app-BraidedFlip}.

\subsection{Algebraic properties of the antipode}

\begin{proposition}
We have $(S\boxtimes S)\comp\chi^\boxtimes=\chi^\boxtimes\comp(S\boxtimes S)$.
\end{proposition}

\begin{proof}
Take homogeneous $a,b\in\cA$. Then
\begin{align*}
(S\boxtimes S)\chi^\boxtimes(\iota_1(a)\iota_2(b))
&=\zeta^{-\deg(a)\deg(b)}(S\boxtimes S)(\iota_1(b)\iota_2(a))\\
&=\zeta^{-\deg(a)\deg(b)}\iota_1\bigl(S(b)\bigr)\iota_2\bigl(S(a)\bigr)\\
&=\chi^{\boxtimes}\bigl(\iota_1(S(a))\iota_2(S(b))\bigr)
=\chi^{\boxtimes}(S\boxtimes S)\bigl(\iota_1(a)\iota_2(b)\bigr).\qedhere
\end{align*}
\end{proof}

\begin{proposition}\label{prop-lemma12}
The antipode satisfies
\begin{equation}\label{eq-eq45}
\Delta_A\comp S=(S\boxtimes S)\comp\chi^{\boxtimes}\comp\Delta_A.
\end{equation}
\end{proposition}

\begin{proof}
First we check that both sides of \eqref{eq-eq45} coincide on the elements $\I$ (trivial), $\alpha$, $\alpha^*$, $\gamma$ and $\gamma^*$. We have
\begin{align*}
\Delta_A\bigl(S(\alpha)\bigr)
&=\Delta_A(\alpha^*)
=\bigl(\iota_1(\alpha)\iota_2(\alpha)-q\iota_1(\gamma)^* \iota_2(\gamma)\bigr)^*
=\iota_2(\alpha^*)\iota_1(\alpha^*)-\overline{q}\iota_2(\gamma^*)\iota_1(\gamma)\\
&=\iota_1(\alpha^*)\iota_2(\alpha^*)-\overline{q}\zeta\iota_1(\gamma)\iota_2(\gamma^*)
=\iota_1(\alpha^*)\iota_2(\alpha^*)-q\iota_1(\gamma)\iota_2(\gamma^*)
\end{align*}
and
\begin{align*}
(S\boxtimes S)\chi^{\boxtimes}\Delta_A(\alpha)
&=(S\boxtimes S)\chi^{\boxtimes}(\iota_1(\alpha)\iota_2(\alpha)-q\iota_1(\gamma^*)\iota_2(\gamma))\\
&=(S\boxtimes S)(\iota_1(\alpha)\iota_2(\alpha)-q\zeta\iota_1(\gamma)\iota_2(\gamma^*))\\
&=\iota_1(\alpha^*)\iota_2(\alpha^*)-q\zeta\iota_1(-\overline{q}\gamma)\iota_2(-q^{-1}\gamma^*)
=\iota_1(\alpha^*)\iota_2(\alpha^*)-q\iota_1(\gamma)\iota_2(\gamma^*).
\end{align*}
Similarly $\Delta_A\bigl(S(\alpha^*)\bigr)
=\Delta_A(\alpha)
=\iota_1(\alpha)\iota_2(\alpha)-q\iota_1(\gamma)^*\iota_2(\gamma)$ and
\begin{align*}
(S\boxtimes S)\chi^{\boxtimes}\Delta_A(\alpha^*)
&=(S\boxtimes S)\chi^{\boxtimes}\bigl((\iota_1(\alpha)\iota_2(\alpha)-q\iota_1(\gamma^*)\iota_2(\gamma))^*\bigr)\\
&=(S\boxtimes S)\chi^{\boxtimes}\bigl(\iota_2(\alpha^*)\iota_1(\alpha^*)-\overline{q}\iota_2(\gamma^*)\iota_1(\gamma)\bigr)\\
&=(S\boxtimes S)\chi^{\boxtimes}\bigl(\iota_1(\alpha^*)\iota_2(\alpha^*)-\overline{q}\zeta\iota_1(\gamma)\iota_2(\gamma^*)\bigr)\\
&=(S\boxtimes S)\bigl(\iota_1(\alpha^*)\iota_2(\alpha^*)-q\zeta\iota_1(\gamma^*)\iota_2(\gamma)\bigr)\\
&=\iota_1(\alpha)\iota_2(\alpha)-q\zeta\iota_1(-q^{-1}\gamma^*)\iota_2(-\overline{q}\gamma)\\
&=\iota_1(\alpha)\iota_2(\alpha)-q\iota_1(\gamma^*)\iota_2(\gamma),
\end{align*}
as well as $\Delta_A\bigl(S(\gamma)\bigr)
=-\overline{q}\Delta_A(\gamma)=-\overline{q}\iota_1(\gamma)\iota_2(\alpha)
-\overline{q}\iota_1(\alpha^*)\iota_2(\gamma)$ and
\begin{align*}
(S\boxtimes S)\chi^{\boxtimes}\Delta_A(\gamma)
&=(S\boxtimes S)\chi^{\boxtimes}\bigl(\iota_1(\gamma)\iota_2(\alpha)+\iota_1(\alpha^*)\iota_2(\gamma)\bigr)\\
&=(S\boxtimes S)\bigl(\iota_1(\alpha)\iota_2(\gamma)+\iota_1(\gamma)\iota_2(\alpha^*)\bigr)\\
&=\iota_1(\alpha^*)\iota_2(-\overline{q}\gamma)+\iota_1(-\overline{q}\gamma)\iota_2(\alpha)
\end{align*}
and finally
\begin{align*}
\Delta_A\bigl(S(\gamma^*)\bigr)
&=-q^{-1}\Delta_A(\gamma^*)
=-q^{-1}\bigl(\iota_1(\gamma)\iota_2(\alpha)+\iota_1(\alpha^*)\iota_2(\gamma)\bigr)^*\\
&=-q^{-1}\iota_2(\alpha^*)\iota_1(\gamma^*)-q^{-1}\iota_2(\gamma^*)\iota_1(\alpha)\\
&=-q^{-1}\iota_1(\gamma^*)\iota_2(\alpha^*)-q^{-1} \iota_1(\alpha)\iota_2(\gamma^*)
\end{align*}
and
\begin{align*}
(S\boxtimes S)\chi^{\boxtimes}\Delta_A(\gamma^*)
&=(S\boxtimes S)\chi^{\boxtimes}\bigl(\iota_1(\gamma^*)\iota_2(\alpha^*)+\iota_1(\alpha)\iota_2(\gamma^*)\bigr)\\
&=(S\boxtimes S)\bigl(\iota_1(\alpha^*)\iota_2(\gamma^*)+\iota_1(\gamma^*)\iota_2(\alpha)\bigr)\\
&=\iota_1(\alpha)\iota_2(-q^{-1}\gamma^*)+\iota_1(-q^{-1}\gamma^*)\iota_2(\alpha^*).
\end{align*}

Now let us take $x,y\in\cA$ and assume that both sides of \eqref{eq-eq45} have equal value for $x$ and $y$. Write (in Sweedler notation) $\Delta_A(x)=\sum \iota_1(x_{(1)})\iota_2(x_{(2)})$ assuming additionally that each $x_{(1)}$, $x_{(2)}$ is homogeneous, and similarly for $y$. Note that since $\Delta_A$ is $\TT$-equivariant, we can additionally assume $\deg(x)=\deg(x_{(1)})+\deg(x_{(2)})$ for each summand. Then writing $\lambda_{p,q}$ for the numerical factor $\zeta^{-\deg(p)\deg(q)}$ we compute
\begin{align*}
\Delta_A\bigl(S(xy)\bigr)
&=\lambda_{x,y}\Delta_A\bigl(S(y)S(x)\bigr)
=\lambda_{x,y}\Delta_A\bigl(S(y)\bigr)\Delta_A\bigl(S(x)\bigr)\\
&=\lambda_{x,y}\bigl((S\boxtimes S)\chi^\boxtimes\Delta_A(y)\bigr)
\bigl((S\boxtimes S)\chi^\boxtimes\Delta_A(x)\bigr)\\
&=\lambda_{x,y}\sum
\lambda_{y_{(1)},y_{(2)}}
(S\boxtimes S)\bigl(\iota_1(y_{(2)})\iota_2(y_{(1)})\bigr)
\lambda_{x_{(1)},x_{(2)}}(S\boxtimes S)\bigl(\iota_1(x_{(2)})\iota_2(x_{(1)})\bigr)\\
&=\lambda_{x,y}
\sum\lambda_{y_{(1)},y_{(2)}}
\lambda_{x_{(1)},x_{(2)}}
\iota_1\bigl(S(y_{(2)})\bigr)\iota_2\bigl(S(y_{(1)})\bigr)
\iota_1\bigl(S(x_{(2)})\bigr)\iota_2\bigl(S(x_{(1)})\bigr)\\
&=\lambda_{x,y}
\sum\lambda_{y_{(1)},y_{(2)}}
\lambda_{x_{(1)},x_{(2)}}\lambda_{x_{(2)},y_{(1)}}\iota_1\bigl(S(y_{(2)})\bigr)
\iota_1\bigl(S(x_{(2)})\bigr)\iota_2\bigl(S(y_{(1)})\bigr)\iota_2\bigl(S(x_{(1)})\bigr)
\end{align*}
and
\begin{align*}
(S\boxtimes S)&\chi^\boxtimes\Delta_A(xy)
=(S\boxtimes S)\chi^\boxtimes\bigl(\Delta_A(x)\Delta_A(y)\bigr)\\
&=\sum(S\boxtimes S)\chi^\boxtimes
\bigl(\iota_1(x_{(1)})\iota_2(x_{(2)})\iota_1(y_{(1)})\iota_2(y_{(2)})\bigr)\\
&=\sum\lambda_{x_{(2)},y_{(1)}}
(S\boxtimes S)\chi^\boxtimes\bigl(\iota_1(x_{(1)}y_{(1)})\iota_2(x_{(2)}y_{(2)})\bigr)\\
&=\sum\lambda_{x_{(2)},y_{(1)}}\lambda_{x_{(1)}y_{(1)},x_{(2)}y_{(2)}}
(S\boxtimes S)\bigl(\iota_1( x_{(2)}y_{(2)})\iota_2(x_{(1)}y_{(1)})\bigr)\\
&=\sum\lambda_{x_{(2)},y_{(1)}}
\lambda_{x_{(1)}y_{(1)},x_{(2)}y_{(2)}}
\iota_1\bigl(S(x_{(2)}y_{(2)})\bigr)\iota_2\bigl(S(x_{(1)}y_{(1)})\bigr)\\
&=\sum\lambda_{x_{(2)},y_{(1)}}
\lambda_{x_{(1)}y_{(1)},x_{(2)}y_{(2)}}
\lambda_{x_{(2)},y_{(2)}}
\lambda_{x_{(1)},y_{(1)}}
\iota_1\bigl(S(y_{(2)})S(x_{(2)})\bigr)\iota_2\bigl(S(y_{(1)})S(x_{(1)})\bigr).
\end{align*}
Now the right-hand sides of the two equations above coincide since (for every summand)
\begin{align*}
\deg(x)\deg(y)&+\deg(y_{(1)})\deg(y_{(2)})+
\deg(x_{(1)})\deg(x_{(2)})+\deg(x_{(2)})\deg(y_{(1)})\\
&=\deg(x_{(1)})\deg(y_{(1)})+\deg(x_{(1)})\deg(y_{(2)})\\
&\qquad+2\deg(x_{(2)})\deg(y_{(1)})+\deg(x_{(2)})\deg(y_{(2)})\\
&\qquad\quad+
\deg(y_{(1)})\deg(y_{(2)})+\deg(x_{(1)})\deg(x_{(2)})
\end{align*}
and
\begin{align*}
\deg(x_{(2)})\deg(y_{(1)})
&+\deg(x_{(1)}y_{(1)})\deg(x_{(2)}y_{(2)})
+\deg(x_{(2)})\deg(y_{(2)})+\deg(x_{(1)})\deg(y_{(1)})\\
&=2\deg(x_{(2)})\deg(y_{(1)})+\deg(x_{(1)})\deg(x_{(2)})+\deg(x_{(1)})\deg(y_{(2)})\\
&\qquad+
\deg(y_{(1)})\deg(y_{(2)})+\deg(x_{(2)})\deg(y_{(2)})+\deg(x_{(1)})\deg(y_{(1)}).
\end{align*}
Thus induction and Corollary \ref{cor-bazacA} can be used to easily finish the proof.
\end{proof}

\begin{proposition}\label{prop-prop10}
We have $*\comp S\comp *\comp S=\id$. In particular $S\colon\cA\to\cA$ is invertible.
\end{proposition}

\begin{proof}
Equality $*\comp S\comp *\comp S=\id$ holds on the generators $\I$, $\alpha$, $\alpha^*$, $\gamma$ and $\gamma^*$. This is trivial for the first three and on $\gamma$ and $\gamma^*$ we have
\begin{align*}
S\bigl(S(\gamma)^*\bigr)^*
&=S(-q\gamma^*)^*
=-\overline{q}(-q^{-1}\gamma^*)^*
=\gamma,\\
S\bigl(S(\gamma^*)^*\bigr)^*
&=S(-\overline{q}^{-1}\gamma)^*
=-q^{-1}(-q\gamma^*)=\gamma^*.
\end{align*}
Next, for homogeneous $a,b\in\cA$
\[
\resizebox{\textwidth}{!}{\ensuremath{\displaystyle
\begin{aligned}
S\bigl(S(ab)^*\bigr)^*
&=S\bigl((\zeta^{-\deg(a)\deg(b)}S(b)S(a))^*\bigr)^*
=\zeta^{-\deg(a)\deg(b)}S\bigl(S(a)^*S(b)^*\bigr)^*\\
&=\zeta^{-\deg(a)\deg(b)}
\bigl(\zeta^{-\deg(S(a)^*)\deg(S(b)^*)}S(S(b)^*)S(S(a)^*)\bigr)^*
=S\bigl(S(a)^*\bigr)^*S\bigl(S(b)^*\bigr)^*
\end{aligned}
}}
\]
and induction ends the proof.
\end{proof}

The antipode $S$ can be used to give an alternative proof of existence of the Haar measure on the braided compact quantum group $\SU_q(2)$, i.e.~one which avoids repeating the steps of the proof of existence of the Haar measure on compact quantum groups as given in Section \ref{sect-HaarMeasure} (see the proof of Corollary \ref{cor-hHaar}).

\begin{proposition}\label{prop-hRightInv}
The state $\bh$ introduced in Section \ref{sect-HaarMeasure} is the Haar measure of $\SU_q(2)$ and we have $\bh\comp S=\bh$ on $\cA$.
\end{proposition}

\begin{proof}
Left invariance of $\bh$ was established in Proposition \ref{prop-hLeftInv}, so we only need to prove right invariance.

Consider $\TT$-invariant map $\bh\comp S\colon\cA\to\CC$. Then by Lemma \ref{prop-lemma12}
\[
\bigl((\bh\comp S)\boxtimes\id\bigr)\comp\Delta_A
=S^{-1}\comp(\bh\boxtimes\id)\comp(S\boxtimes S)\comp\Delta_A
=S^{-1}\comp(\bh\boxtimes\id)\comp(\chi^\boxtimes)^{-1}\comp\Delta_A\comp S
\]
and as
\begin{align*}
(\bh\boxtimes\id)(\chi^\boxtimes)^{-1}\bigl(\iota_2(a)\iota_1(b)\bigr)
&=(\bh\boxtimes\id)\bigl(\iota_1(a)\iota_2(b)\bigr)=\bh(a)b,\\
(\id\boxtimes\bh)\bigl(\iota_2(a)\iota_1(b)\bigr)
&=\zeta^{-\deg(a)\deg(b)}(\id\boxtimes\bh)\bigl(\iota_1(b)\iota_2(a)\bigr)
=\zeta^{-\deg(a)\deg(b)}\bh(a)b
\end{align*}
for homogeneous $a,b\in\cA$, we get $(\bh\boxtimes\id)\comp(\chi^\boxtimes)^{-1}=\id\boxtimes\bh$ and consequently for any $x\in\cA$
\[
\bigl((\bh\comp S)\boxtimes\id\bigr)\Delta_A(x)
=S^{-1}(\id\boxtimes\bh)\Delta_A\bigl(S(x)\bigr)
=(\bh\comp S)(x)S^{-1}(\I)
=(\bh\comp S)(x)\I
\]
i.e.~$\bh\comp S$ is right invariant.

Finally this right invariance of $\bh\comp S$ and left invariance of $\bh$ gives
\[
\bh\comp S=\bh(\I)\cdot(\bh\comp S)
=\bigl((\bh\comp S)\boxtimes\bh\bigr)\comp\Delta_A
=(\bh\comp S)\comp(\id\boxtimes\bh)\comp\Delta_A
=(\bh\comp S)(\I)\cdot\bh=\bh
\]
which gives $\bh\comp S=\bh$ and consequently shows that $\bh$ is right invariant.
\end{proof}

 \begin{remark}
The canonical embedding $\kappa\colon A\to B$ introduced in Section \ref{sect-boson} does not respect antipodes. Indeed, using \cite[Theorem 2.4]{ZhangZhao} we have $S_{(B,\Delta_B)}(\gamma)=-\overline{q}^{-1}\gamma z^*$ and thus Lemma \ref{lem-lemma20} from the next section implies that $S_{(B,\Delta_B)}(\gamma)\not\in\kappa(A)$.
\end{remark}

\subsection{Representation theory}\label{sect-RepTh}

In order to prove further properties of the antipode (e.g. that it is the convolution inverse of the identity, cf.~Theorem \ref{thm-prop8}) we need to discuss some elements of representation theory of the braided compact quantum group $\SU_q(2)$.

Given a Hilbert space $\sH$ with a representation $U^\sH$ of $\TT$ viewed as an element of $\M(\C(\TT)\tens\cK(\sH))$, we have the associated action $\rho^{\cK(\sH)}\in\Mor(\cK(\sH),\C(\TT)\tens\cK(\sH))$:
\[
\rho^{\cK(\sH)}(x)={U^{\sH}}^*(\I\tens x)U^\sH,\qquad{x}\in\cK(\sH).
\]
Then Proposition \ref{prop-lemma1} provides an action $\tilde{\rho}^{\cK(\sH)}$ of $\ZZ$ on $\cK(\sH)$ associated with the fixed constant $\zeta=q/\overline{q}$ and the two actions combine to an action $\brho^{\cK(\sH)}$ of $\D(\TT)$. As noted in \cite[Example 5.17]{MeyerRoyWoronowiczI} and \cite[Section 8.2]{DeCommerKrajczok} (in the von Neumann algebra setting), for a \cst-algebra $\sX$ with a $\TT$-action we can identify $\sX\boxtimes\cK(\sH)$ with $\sX\tens\cK(\sH)$. In particular $\cK(\sH)\boxtimes\cK(\sK)\simeq\cK(\sH\tens\sK)$. In fact, with our definitions we have equality -- the argument is essentially that of \cite[Proposition 8.1]{DeCommerKrajczok} modified to the \cst-setting (in particular one needs to use the universal lift of the \cst-algebraic action \cite{FischerPreprint}). The precise statement is the following:

\begin{proposition}
Let $(\GG,\hh{R})$ be a quasi-triangular locally compact quantum group acting on \cst-algebras $\sX\subseteq\B(\sH)$ and $\sY\subseteq\B(\sK)$ with actions $\rho^\sX$ and $\rho^\sY$ respectively. Assume that the action on $\sY$ is inner, i.e.~$\rho^\sY(y)=(U^\sY)^*(\I\tens y)U^\sY$ for a unitary representation $U^\sY\in\M(\C_0(\GG)\tens\sY)$. Then $\sX\boxtimes\sY=\sX\tens\sY$.
\end{proposition}

We take the following definitions from \cite[Section 5]{BraidedSU2}, but change \emph{right} representations to \emph{left} representations\footnote{This is the natural choice, when working with left actions.}, i.e.~in this paper representations are elements of $\M(A\tens\cK(\sH))$ not of $\M(\cK(\sH)\tens A)$.

\begin{definition}\label{def-def1}
Let $\sH$ be a Hilbert space with a representation of $\TT$ and $v\in\M(A\tens\cK(\sH))$ a unitary element such that $(\rho^A_\lambda\tens\rho^{\cK(\sH)}_\lambda)v=v$ for all $\lambda\in\TT$. We say that $v$ is a representation of $\SU_q(2)$ on $\sH$ if
\[
(\Delta_A\tens\id)v=\bigl((\iota_1\tens\id)v\bigr)\bigl((\iota_2\tens\id)v\bigr).
\]
\end{definition}

\begin{definition}
Let $\sH$ and $\sK$ be Hilbert spaces with representations of $\TT$ and let $u\in\M(A\tens\cK(\sH))$ and $v\in\M(A\tens\cK(\sK))$ be representations of $\SU_q(2)$. The \emph{tensor product} of $u$ and $v$ is defined as
\[
u\otop v=
\bigl((\id\tens\iota_1)u\bigr)\bigl((\id\tens\iota_2)v\bigr)
\in\M\bigl(A\tens(\cK(\sH)\boxtimes\cK(\sK))\bigr)
=\M\bigl(A\tens\cK(\sH\tens\sK)\bigr)
\]
see \cite[Proposition 5.3]{BraidedSU2}.
\end{definition}

\begin{remarks}\hspace*{\fill}
\begin{enumerate}
\item The representation of $\TT$ on $\sH\tens\sK$ used in the construction of the tensor product of representations of $\SU_q(2)$ is the tensor product of representations on $\sH$ and $\sK$ respectively.
\item Representations of braided locally compact quantum groups appear also in \cite[Definition 2.2]{Braidedsymgraph} with reference to \cite{MeyerRoyBraided}.
\end{enumerate}
\end{remarks}

In what follows, by a \emph{matrix element} of a representation $u\in\M(A\tens\cK(\sH))$ of $\SU_q(2)$ we will understand an element of $A$ of the form $(\id\tens\omega)u$, where $\omega\in\cK(\sH)^*$.

\begin{proposition}\label{prop-lemma19}
There exists a unique character $\eps\colon A\to\CC$ satisfying $\eps(\alpha)=1$ and $\eps(\gamma)=0$. It is $\TT$-equivariant and
\begin{equation}\label{eq-eq50}
(\eps\boxtimes\id)\comp\Delta_A=(\id\boxtimes\eps)\comp\Delta_A=\id.
\end{equation}
Moreover, $(\eps\tens\id)u=\I$ for any unitary representation $u$ of $\SU_q(2)$.
\end{proposition}


\begin{proof}
Existence of a unique character $\eps$ satisfying $\eps(\alpha)=1$ and $\eps(\gamma)=0$ follows from the definition of $A$ as the universal \cst-algebra for the relations \eqref{eq-SUq2rels}. The $\TT$-equivariance of $\eps$ holds by the definition of the action \eqref{eq-eq5} and \eqref{eq-eq50} follows from a simple calculation on generators. (Note that we can form $\eps\boxtimes\id$ and $\id\boxtimes\eps$ because $\eps$ is $\D(\TT)$--equivariant). Now take a unitary representation $u$ of $\SU_q(2)$ on $\sH$. By definition
\[
(\Delta_A\tens\id)u=\bigl((\iota_1\tens\id)u\bigr)\bigl((\iota_2\tens\id)u\bigr)
\in\M\bigl((A\boxtimes A)\tens\cK(\sH)\bigr)
\]
and applying the $*$-homomorphism $(\eps\boxtimes\id)\tens\id$ to both sides gives
\begin{align*}
u&=\Bigl(\bigl((\eps\boxtimes\id)\comp\Delta_A\bigr)\tens\id\Bigr)u\\
&=\Bigl(\bigl(((\eps\boxtimes\id)\comp\iota_1)\tens\id\bigr)u\Bigr)
\Bigl(\bigl(((\eps\boxtimes\id)\comp\iota_2)\tens\id\bigr)u\Bigr)\\
&=\bigl(\I\tens((\eps\tens\id)u)\bigr)u.\qedhere
\end{align*}
\end{proof}

In what follows, we will use \cite[Theorem 6.1]{BraidedSU2}. Let us state and prove it here as there are differences in conventions and we provide a few more details than the reasoning in \cite{BraidedSU2}. We will mostly use the bosonization of $\SU_q(2)$ in the form $\widetilde{B}=A\boxtimes\C(\TT)$ (see Section \ref{sect-BosonBoson}). We will use a combination of notations described in Section \ref{sect-notation} as follows:
\begin{equation}\label{eq-notconv}
\vcenter{\xymatrix@R=2pt{
\quad\quad A\,\ar@{^{(}->}[rr]^(.41){\iota_A}&&A\boxtimes\C(\TT),\hphantom{\boxtimes A,}\\
\:\,\,\C(\TT)\,\ar@{^{(}->}[rr]^(.41){\iota_{\C(\TT)}}&&A\boxtimes\C(\TT),\hphantom{\boxtimes A,}\\
\quad\quad A\,\ar@{^{(}->}[rr]^(.41){\iota_1}&&A\boxtimes A\boxtimes\C(\TT)\hphantom{,}&\text{onto the first factor},\\
\quad\quad A\,\ar@{^{(}->}[rr]^(.41){\iota_2}&&A\boxtimes A\boxtimes\C(\TT)\hphantom{,}&\text{\quad\;\!\:\!onto the second factor},\\
\:\,\,\C(\TT)\,\ar@{^{(}->}[rr]^(.41){\iota_3}&&A\boxtimes A\boxtimes\C(\TT),\\
A\boxtimes A\,\ar@{^{(}->}[rr]^(.41){\iota_{12}}&&A\boxtimes A\boxtimes\C(\TT),\\
\quad\quad A\,\ar@{^{(}->}[rr]^(.41){\jmath_1}&&A\boxtimes A\hphantom{\boxtimes\C(\TT),\!,}&\text{onto the first factor},\\
\quad\quad A\,\ar@{^{(}->}[rr]^(.41){\jmath_2}&&A\boxtimes A\hphantom{\boxtimes\C(\TT),\!,}&\text{\quad\;\!\:\!onto the second factor}.\\
}}
\end{equation}

From the alternative description of bosonization given in Section \ref{sect-boson} we see that there is a unique unital $*$-homomorphism $\pi\colon\widetilde{B}\to\C(\TT)$ such that $\pi(\iota_A(\alpha))=\I$, $\pi(\iota_A(\gamma))=0$ and $\pi(\iota_{\C(\TT)}(\bz))=\bz$. Clearly it satisfies $\pi\comp\iota_{\C(\TT)}=\id$ and
\[
\pi\bigl(\iota_A(a)\bigr)=\eps(a)\I,\qquad a\in A.
\]
In Section \ref{sect-BosonBoson} we introduce the comultiplication $\Delta_{\widetilde{B}}$ on $\widetilde{B}$ and by checking on generators we find that
\begin{equation}\label{eq-eq54}
(\id\tens\pi)\bigl(\Delta_{\widetilde{B}}(\iota_A(a))\bigr)=\iota_A(a)\tens\I,\qquad a\in A
\end{equation}
and that
\begin{equation}\label{eq-eq55}
\rho^{\widetilde{B}}=(\pi\tens\id)\comp\Delta_{\widetilde{B}}.
\end{equation}
Next we need a small lemma.

\begin{lemma}\label{lem-lemma20}
If $x\in\widetilde{B}$, then $(\id\tens(\bh_{\TT}\comp\pi))\Delta_{\widetilde{B}}(x)=x$ if and only if $x$ belongs to the range of $\iota_A$.
\end{lemma}

\begin{proof}
Take $x\in \widetilde{B}$ and assume that $(\id\tens(\bh_{\TT}\comp\pi))\Delta_{\widetilde{B}}(x)=x$. Approximate
\[
x=\lim_{n\to\infty}\sum_{k=-\infty}^{\infty}\iota_A(x_{n,k})\iota_{\C(\TT)}(\bz^k)
\]
for some $x_{n,k}\in\cA$ (finitely many non-zero for each $n$). We see from \eqref{eq-eq54} that
\[
(\id\tens\pi)\Delta_{\widetilde{B}}\iota_A(x_{n,k})\in\widetilde{B}\tens\I
\]
and hence
\begin{align*}
x&=\bigl(\id\tens(\bh_{\TT}\comp\pi)\bigr)\Delta_{\widetilde{B}}(x)\\
&=\lim_{n\to\infty}\sum_{k=-\infty}^{\infty}
(\id\tens\bh_{\TT})
\Bigl(
\bigl(
(\id\tens\pi)\Delta_{\widetilde{B}}(\iota_A(x_{n,k}))
\bigr)
\bigl(
(\id\tens\pi)\Delta_{\widetilde{B}}\bigl(\iota_{\C(\TT)}(\bz^k)\bigr)
\bigr)\Bigr)\\
&=\lim_{n\to\infty}\sum_{k=-\infty}^{\infty}
\Bigl(
\bigl(\id\tens(\bh_{\TT}\comp\pi)\bigr)
\Delta_{\widetilde{B}}\bigl(\iota_A(x_{n,k})\bigr)
\Bigr)
\Bigl(
\iota_{\C(\TT)}\bigl(\bz^k\bh_{\TT}(\bz^k)\bigr)
\Bigr)\\
&=\lim_{n\to\infty}
\bigl(\id\tens(\bh_{\TT}\comp\pi)\bigr)\Delta_{\widetilde{B}}\bigl(\iota_A(x_{n,0})\bigr)
=\lim_{n\to\infty}\iota_A(x_{n,0})
\end{align*}
using again \eqref{eq-eq54} in the last equality. The converse is clear by \eqref{eq-eq54}.
\end{proof}

\begin{theorem}[{\cite[Theorem 6.1]{BraidedSU2}}]\label{thm-thm1}
Let $\sH$ be a Hilbert space equipped with unitary representation $U\in\M(\C(\TT)\tens\cK(\sH))$. There is a bijection between unitary representations $u$ of $\SU_q(2)$ on $\sH$ and unitary representations $v$ of the compact quantum group described by $(\widetilde{B},\Delta_{\widetilde{B}})$ satisfying $(\pi\tens\id)v=U$. This bijection is given by
\begin{equation}\label{eq-eq49}
v=\bigl((\iota_A\tens\id)u\bigr)
\bigl((\iota_{\C(\TT)}\tens\id)U\bigr).
\end{equation}
\end{theorem}

\begin{proof}
Let $u$ be a unitary representation of $\SU_q(2)$ on $\sH$ and define $v$ via \eqref{eq-eq49}. It is a unitary element of $\M(\widetilde{B}\tens\cK(\sH))$. We have
\[
(\pi\tens\id)v
=\bigl((\pi\iota_A\tens\id)u\bigr)
\bigl((\pi\iota_{\C(\TT)}\tens\id)U\bigr)
=\bigl(\I\tens(\eps\tens\id)u\bigr)U=U
\]
as claimed. Next we check that $v$ is a representation. The invariance condition on $u$ (Definition \ref{def-def1}) reads $U^*_{13}((\rho^A\tens\id)u)U_{13}=u_{23}$, so
\[
 (\rho^A\tens\id)u =U_{13}u_{23}U_{13}^*.
\]
Using this, the expression $\Delta_{\widetilde{B}}=\Psi\comp(\Delta_A\boxtimes\id)$ (see Section \ref{sect-BosonBoson}) and Proposition \ref{prop-prop2Cstar} we have
\begin{equation}\label{eq-eq59}
\resizebox{0.9\textwidth}{!}{\ensuremath{\displaystyle
\begin{aligned}
(\Delta_{\widetilde{B}}&\tens\id)v\\
&=
(\Delta_{\widetilde{B}}\tens\id)\Bigl(
\bigl((\iota_A\tens\id)u\bigr)
\bigl((\iota_{\C(\TT)}\tens\id)U\bigr)
\Bigr)\\
&=\Bigl((\Psi\tens\id)\bigl(
((\iota_{12}\comp\Delta_A)\tens\id)u
\bigr)\Bigr)
\Bigl((\Psi\tens\id)\bigl((\iota_3\tens\id)U\bigr)\Bigr)\\
&=\Bigl(
(\Psi\tens\id)
(\iota_{12}\tens\id)\bigl(((\jmath_1\tens\id)u)((\jmath_2\tens\id)u)\bigr)
\Bigr)
\Bigl((\Psi\tens\id)
\bigl((\iota_3\tens\id)U\bigr)
\Bigr)\\
&=\Bigl(
(\Psi\tens\id)
\bigl((\iota_1\tens\id)u\bigr)
\Bigr)
\Bigl(
(\Psi\tens\id)
\bigl((\iota_2\tens\id)u\bigr)
\Bigr)
\Bigl(
(\Psi\tens\id)
\bigl(\iota_3\tens\id)U\bigr)
\Bigr)\\
&=\bigl((\iota_A\tens\id)u\bigr)_{13}
\Bigl(
\bigl(((\iota_{\C(\TT)}\tens\iota_A)\comp\rho^A)\tens\id\bigr)u
\Bigr)
\Bigl(
\bigl(((\iota_{\C(\TT)}\tens\iota_{\C(\TT)})\comp\Delta_\TT)\tens\id\bigr)U
\Bigr)\\
&=\bigr((\iota_A\tens\id)u\bigr)_{13}
\bigl((\iota_{\C(\TT)}\tens\iota_A\tens\id)(U_{13}u_{23}U_{13}^*)\bigr)
\bigl((\iota_{\C(\TT)}\tens\iota_{\C(\TT)}\tens\id)(U_{13}U_{23})\bigr)\\
&=\bigl((\iota_A\tens\id)u\bigr)_{13}
\bigl((\iota_{\C(\TT)}\tens\id)U\bigr)_{13}
\bigl((\iota_A\tens\id)u\bigr)_{23}
\bigl(\iota_{\C(\TT)}\tens\id)U\bigr)_{23}
=v_{13}v_{23}
\end{aligned}
}}
\end{equation}
which proves that $v$ is a representation of the compact quantum group described by $(\widetilde{B},\Delta_{\widetilde{B}})$.

Conversely, assume that $v$ is a unitary representation of $(\widetilde{B},\Delta_{\widetilde{B}})$ on $\sH$ satisfying $(\pi\tens\id)v=U$ and define the unitary
\[
\widetilde{u}=v\bigl((\iota_{\C(\TT)}\tens\id)U\bigr)^*\in
\M\bigl(\widetilde{B}\tens\cK(\sH)\bigr).
\]
Next observe that for $\xi,\eta\in\sH$
\begin{align*}
\bigl(\id\tens(\bh_\TT\comp\pi)\bigr)
&\Delta_{\widetilde{B}}\bigl((\id\tens\omega_{\xi,\eta})\widetilde{u}\bigr)\\
&=\bigl(\id\tens(\bh_{\TT}\comp\pi)\tens\omega_{\xi,\eta}\bigr)
\Bigl(v_{13}v_{23}
\bigl((\iota_{\C(\TT)}\tens\id)U\bigr)^*_{23}
\bigl((\iota_{\C(\TT)}\tens\id)U\bigr)^*_{13}\Bigr)\\
&=(\id\tens\bh_\TT\tens\omega_{\xi,\eta})
\Bigl(v_{13}U_{23}U^*_{23}
\bigl((\iota_{\C(\TT)}\tens\id)U\bigr)^*_{13}
\Bigr)
=(\id\tens\omega_{\xi,\eta})\widetilde{u}
\end{align*}
and hence Lemma \ref{lem-lemma20} gives $(\id\tens\omega_{\xi,\eta})\widetilde{u}\in \iota_A(A)$. It easily follows that
\[
\widetilde{u}\in\M\bigl(\iota_A(A)\tens\cK(\sH)\bigr)
\]
and we can define $u\in\M(A\tens\cK(\sH))$ such that
\[
(\iota_A\tens\id)u=\widetilde{u}
\]
(recall that the minimal tensor product of injective maps is injective, and strict extension of injective $*$-homomorphism is injective, see also \cite[Proposition 2.3]{Lance}).

Thus we have $v=((\iota_A\tens\id)u)((\iota_{\C(\TT)}\tens\id)U)$. Clearly $u$ is unitary, we need to check that it is a representation. First let us check that it is $\TT$-invariant, or equivalently that
\begin{equation}\label{eq-eq57}
U_{13}^*\bigl((\rho^A\tens\id)u\bigr)U_{13}=u_{23}.
\end{equation}
Using \eqref{eq-eq55} we compute
\begin{align*}
(\id\tens\iota_A\tens\id)
&
\Bigl(U_{13}^*\bigl((\rho^A\tens\id)u\bigr)U_{13}\Bigr)
\\
&
=U_{13}^*
\Bigl(\bigl((\id\tens\iota_A)\comp\rho^A\tens\id\bigr)u\Bigr)U_{13}
=U_{13}^*\Bigl(\bigl((\rho^{\widetilde{B}}\comp\iota_A)\tens\id\bigr)u\Bigr)U_{13}
\\
&
=U_{13}^*\Bigl(
\bigl(((\pi\tens\id)\comp\Delta_{\widetilde{B}})\tens\id\bigr)\widetilde{u}\Bigr)U_{13}
\\
&
=U_{13}^*\Bigl(\bigl((\pi\tens\id)\comp\Delta_{\widetilde{B}}\tens\id\bigr)
\bigl(v((\iota_{\C(\TT)}\tens\id)U^*)\bigr)\Bigr)U_{13}
\\
&
=U_{13}^*
\bigl((\pi\tens\id\tens\id)(v_{13}v_{23})\bigr)
\bigl((\pi\tens\id\tens\id)(\iota_{\C(\TT)}\tens\iota_{\C(\TT)}\tens\id)(U^*_{23}U^*_{13})\bigr)U_{13}
\\
&
=U_{13}^*\bigl((\pi\tens\id)v\bigr)_{13}v_{23}
\bigl((\iota_{\C(\TT)}\tens\id)U^*\bigr)_{23}\\
&=v_{23}
\bigl((\iota_{\C(\TT)}\tens\id)U^*\bigr)_{23}=\widetilde{u}_{23}
=\bigl((\iota_A\tens\id)u\bigr)_{23}
\end{align*}
(here we count $\widetilde{B}$ as having one leg). This proves \eqref{eq-eq57}, since the map $\id\tens\iota_A\tens\id$ is injective. To prove the remaining equation
\begin{equation}\label{eq-eq58}
(\Delta_A\tens\id)u
=\bigl((\jmath_1\tens\id)u\bigr)\bigl((\jmath_2\tens\id)u\bigr)
\end{equation}
we reverse calculation \eqref{eq-eq59}: we have
\begin{align*}
(\Delta_{\widetilde{B}}\tens\id)v
&=(\Delta_{\widetilde{B}}\tens\id)
\Bigl(
\bigl((\iota_A\tens\id)u\bigr)
\bigl((\iota_{\C(\TT)}\tens\id)U\bigr)
\Bigr)\\
&=(\Psi\tens\id)\Bigl(
\bigl((\iota_{12}\comp\Delta_A)\tens\id\bigr)u
\Bigr)
(\Psi\tens\id)
\bigl((\iota_3\tens\id)U\bigr)
\\
&=(\Psi\tens\id)
\Bigl(
\bigl((\iota_{12}\comp\Delta_A)\tens\id\bigr)u
\Bigr)
\bigl((\iota_{\C(\TT)}\tens\iota_{\C(\TT)}\tens\id)(U_{13}U_{23})\bigr)
\end{align*}
and this is equal to (using \eqref{eq-eq57})
\begin{align*}
v_{13}v_{23}
&=\bigl((\iota_A\tens\id)u\bigr)_{13}
\bigl((\iota_{\C(\TT)}\tens\id)U\bigr)_{13}
\bigl((\iota_A\tens\id)u\bigr)_{23}
\bigl((\iota_{\C(\TT)}\tens\id)U\bigr)_{23}\\
&=
\bigl((\iota_A\tens\id)u\bigr)_{13}
\bigl((\iota_{\C(\TT)}\tens\iota_A\tens\id)(U_{13}u_{23}U_{13}^*)\bigr)
\bigl((\iota_{\C(\TT)}\tens\iota_{\C(\TT)}\tens\id)(U_{13}U_{23})\bigr)\\
&=\bigl((\iota_A\tens\id)u\bigr)_{13}
\Bigl(\bigl(((\iota_{\C(\TT)}\tens\iota_A)\comp\rho^A)\tens\id\bigr)u\Bigr)
\bigl((\iota_{\C(\TT)}\tens\iota_{\C(\TT)}\tens\id)(U_{13}U_{23})\bigr)\\
&=(\Psi\tens\id)
\bigl((\iota_1\tens\id)u\bigr)
(\Psi\tens\id)
\bigl((\iota_2\tens\id)u\bigr)
\bigl((\iota_{\C(\TT)}\tens\iota_{\C(\TT)}\tens\id)(U_{13}U_{23})\bigr)\\
&=(\Psi\tens\id)\Bigl(
(\iota_{12}\tens\id)\bigl(((\jmath_1\tens\id)u)((\jmath_2\tens\id)u)\bigr)
\Bigr)
\bigl((\iota_{\C(\TT)}\tens\iota_{\C(\TT)}\tens\id)(U_{13}U_{23})\bigr).
\end{align*}
This implies
\[
(\Psi\tens\id)\Bigl(\bigl((\iota_{12}\comp\Delta_A)\tens\id\bigr)u\Bigr)=
(\Psi\tens\id)
\Bigl(
(\iota_{12}\tens\id)\bigl(((\jmath_1\tens\id)u)((\jmath_2\tens\id)u)\bigr)
\Bigr).
\]
Since $\Psi$ and $\iota_{12}$ are injective, this proves \eqref{eq-eq58}. Clearly both passages from $u$ to $v$ and from $v$ to $u$ are one another's inverses.
\end{proof}

As a corollary we get the following nice result.

\begin{proposition}\label{prop-prop7}
The $*$-subalgebra $\cA\subseteq A$ coincides with the set of matrix elements of all finite-dimensional unitary representations of $\SU_q(2)$.
\end{proposition}

\begin{proof}
Write $\mathcal{P}$ for the space of matrix elements of all finite-dimensional unitary representations of $\SU_q(2)$. The possibility of taking direct sums and tensor products implies that $\mathcal{P}$ is an algebra
and since $\cA$ is generated by $\{\I,\alpha,\alpha^*,\gamma,\gamma^*\}$ and these are matrix elements of the fundamental representation, we easily get $\cA\subseteq\mathcal{P}$ (note that we do not need to involve the notion of the contragredient representation here).

Conversely, let $u$ be a representation of $\SU_q(2)$ on a finite-dimensional Hilbert space $\sH$ with a representation $U$ of $\TT$ and take $(\id\tens\omega)u\in\mathcal{P}$ for some $\omega\in\B(\sH)^*$. Then by Theorem \ref{thm-thm1}
\[
v=\bigl((\iota_A\tens\id)u\bigr)
\bigl((\iota_{\C(\TT)}\tens\id)U\bigr)
\]
is a finite-dimensional unitary representation of the compact quantum group described by $(\widetilde{B},\Delta_{\widetilde{B}})$. Since $(\widetilde{B},\Delta_{\widetilde{B}})$ is isomorphic to $\mathrm{U}_q(2)^{\text{op}}$ (in a concrete way, see Section \ref{sect-boson}), we can use results about that quantum group. In particular we see from the classification result \cite[Theorem 4.17]{GuinSaurabh} that $\Pol(\widetilde{B},\Delta_{\widetilde{B}})=\iota_A(\cA)\iota_{\C(\TT)}(\Pol(\TT))$. Consequently, after decomposing $v$ into a direct sum of irreps, we conclude that matrix elements of $v$ belong to $\iota_A(\cA)\iota_{\C(\TT)}(\Pol(\TT))$. It follows that
\[
\iota_A\bigl((\id\tens\omega)u\bigr)
=(\id\tens\omega)\Bigl(
v\bigl((\iota_{\C(\TT)}\tens\id)U\bigr)^*\Bigr)
\in\iota_A(\cA)\iota_{\C(\TT)}\bigl(\Pol(\TT)\bigr)=\Pol(\widetilde{B},\Delta_{\widetilde{B}}),
\]
and $\iota_A((\id\tens\omega)u)$ is a polynomial in $\alpha,\alpha^*,\gamma,\gamma^*,z$ and $z^*$ (here $\alpha,\gamma$ and $z$ are generators of $\widetilde{B}\simeq B$ as in Section \ref{sect-boson}, cf.~also Section \ref{sect-BosonBoson}). Consider the action of $\TT$ on $\widetilde{B}\simeq B$ given by $z\mapsto \bz\tens z$ and $\alpha\mapsto\I\tens\alpha$, $\gamma\mapsto\I\tens\gamma$. In particular the action on $\iota_A((\id\tens\omega)u)\in \iota_A(A)$ is trivial. The conditional expectation onto the fixed point subalgebra for this action allows us to deduce that $(\id\tens\omega)u\in\cA$.
\end{proof}

To proceed further, let us introduce some more terminology and braided analogs of known results. From now on we will consider only finite-dimensional representations.

\begin{definition}
Let $u$ and $v$ be finite-dimensional, unitary representations of $\SU_q(2)$ on Hilbert spaces $\sH$ and $\sK$ equipped with representations for $U$ and $V$ of $\TT$. We say that $T\in\B(\sH,\sK)$ \emph{intertwines} $u$ and $v$ if
\begin{enumerate}
\item\label{def-reps1} $T\in\Mor_\TT(U,V)$, i.e.~$T$ is equivariant for the $\TT$-representations $U$ and $V$,
\item\label{def-reps2} we have $(\I\tens T)u=v(\I\tens T)$.
\end{enumerate}
The set of all intertwiners for $u$ and $v$ will be denoted by $\Mor(u,v)$ and by $\End(u)$ in case $v=u$. We say that $u$ is \emph{irreducible} if $\End(u)=\CC\I$.
\end{definition}

\begin{remarks}\hspace*{\fill}
\begin{enumerate}
\item The equivariance condition on intertwiners seems to be important -- without it e.g.~Proposition \ref{prop-lemma24} below fails.
\item It is easy to construct examples where \eqref{def-reps2} holds, but \eqref{def-reps1} fails (e.g. $u=v$ the trivial representation on $\CC^2$ with non-equivalent $U$ and $V$).
\item $T$ intertwines $u$ and $v$ iff it is a morphism of associated representations of $(\widetilde{B},\Delta_{\widetilde{B}})$. In particular, $u$ is irreducible iff the associated representation of $(\widetilde{B},\Delta_{\widetilde{B}})$ is irreducible.
\end{enumerate}
\end{remarks}


\begin{lemma}
Let $u$, $v$ and $w$ be finite-dimensional unitary representations of $\SU_q(2)$ on $\TT$-Hilbert spaces $\sH_u$, $\sH_v$ and $\sH_w$. Then
\begin{enumerate}
\item $\Mor(u,v)$ is a closed subspace of $\B(\sH_u,\sH_v)$ and hence it is a Banach space. If $T\in\Mor(u,v)$ then $T^*\in\Mor(v,u)$. If $S\in\Mor(v,w)$ then $ST\in \Mor(u,w)$. In particular, $\End(u)$ is a \cst-algebra.
\item (Schur's lemma) Assume that $u,v$ are irreducible. Then either $u$ and $v$ are unitarily equivalent and $\Mor(u,v)$ is one-dimensional or $\Mor(u,v)=\{0\}$.
\end{enumerate}
\end{lemma}

\begin{proof}
The first point is trivial. The second can be proven exactly as its non-braided version (see e.g.~\cite[Proposition 1.3.4]{NeshveyevTuset}).
\end{proof}

\begin{proposition}\label{prop-lemma24}
Let $u$ be a finite-dimensional unitary representation of $\SU_q(2)$. Then $u$ is equivalent to a direct sum of irreducible representations.
\end{proposition}

\begin{proof}
This decomposition is performed exactly as in the non-braided case (e.g.~\cite[Theorem 1.3.7]{NeshveyevTuset}). Let us remark however, that when we write $\End(u)=\bigoplus_{a=1}^{A}\operatorname{Mat}_{n_a}$ with basis of matrix units $\{e^a_{i,j}\}$ then the minimal projections $e^a_{i,i}$ are equivariant for the representation of $\TT$ (by assumption), hence we can indeed define smaller representations $(\I\tens e^a_{i,i})u$ of $\SU_q(2)$ on $e^a_{i,i}\sH$ (where $\sH$ is the carrier Hilbert space of $u$) with representations of $\TT$ given by $(\I\tens e^a_{i,i})U$.
\end{proof}

\begin{lemma}\label{lem-lemma14}
\hspace*{\fill}
\begin{enumerate}
\item\label{lem-lemma14-1} Let $u$ be the fundamental two-dimensional representation of $\SU_q(2)$. Then for any $\omega\in\B(\CC^2)^*$ we have $S((\id\tens\omega)u)=(\id\tens\omega)(u^*)$.
\item Let $u$ and $v$ be finite-dimensional unitary representations of $\SU_q(2)$ on $\sH$ and $\sK$ such that $S((\id\tens\omega)u)=(\id\tens\omega)(u^*)$ and $S((\id\tens\omega')v)=(\id\tens\omega')(v^*)$ for all $\omega\in\B(\sH)^*$, $\omega'\in\B(\sK)^*$. Then $S((\id\tens\Omega)(u\otop v))=(\id\tens\Omega)((u\otop v)^*)$ for any $\Omega\in\B(\sH\tens\sK)^*$.
\end{enumerate}
\end{lemma}

\begin{proof}
$S$ was defined so that \eqref{lem-lemma14-1} holds for $\omega=\omega_{\xi,\eta}$ with $\xi,\eta$ any of the basis vectors. The general case follows by sesquilinearity.

Next let us choose orthonormal bases $\{\xi_k\}$ and $\{\eta_l\}$ in $\sH$ and $\sK$ consisting of eigenvectors for the corresponding actions of $\TT$: $U^\sH(\I\tens\xi_k)=\bz^{n_k}\tens\xi_k$ and $U^\sK(\I\tens\eta_l)=\bz^{m_l}\tens\eta_l$. Then
\[
\rho^{\B(\sH)}\bigl(\ket{\xi_k}\bra{\xi_{k'}}\bigr)
=\bz^{n_{k'}-n_k}\tens\ket{\xi_k}\bra{\xi_{k'}}
\]
and similarly for the matrix units in $\B(\sK)$. If we write $u=\sum\limits_{k,k'}u_{k,k'}\tens\ket{\xi_k}\bra{\xi_{k'}}$ then
\[
\rho^A(u_{k,k'})=\bz^{-n_{k'}+n_k}\tens u_{k,k'}
\]
and similarly for $v_{l,l'}$. The claim follows from the calculation: for any $\zeta,\theta\in\sH\tens\sK$
\[
\resizebox{\textwidth}{!}{\ensuremath{\displaystyle
\begin{aligned}
S\bigl((\id\tens\omega_{\zeta,\theta})(u\otop v)\bigr)
&=(S\tens\omega_{\zeta,\theta})
\Bigl(\bigl((\id\tens\iota_1)u\bigr)\bigl((\id\tens\iota_2)v\bigr)\Bigr)\\
&=\sum_{k,k',l,l'}
(S\tens\omega_{\zeta,\theta})
\bigl(u_{k,k'}v_{l,l'}\tens\iota_1(\ket{\xi_k}\bra{\xi_{k'}})\iota_2(\ket{\eta_l}\bra{\eta_{l'}})\bigr)\\
&=\sum_{k,k',l,l'}\zeta^{-(-n_{k'}+n_k)(-m_{l'}+m_l)}
S(v_{l,l'})S(u_{k,k'})
\omega_{\zeta,\theta}\bigl(\iota_1(\ket{\xi_k}\bra{\xi_{k'}})\iota_2(\ket{\eta_l}\bra{\eta_{l'}})\bigr)\\
&=\sum_{k,k',l,l'}
\zeta^{-(-n_{k'}+n_k)(-m_{l'}+m_l)}
v_{l',l}^*u_{k',k}^*
\zeta^{(n_{k'}-n_k)(m_{l'}-m_l)}
\omega_{\zeta,\theta}\bigl(
\iota_2(\ket{\eta_l}\bra{\eta_{l'}})
\iota_1(\ket{\xi_k}\bra{\xi_{k'}})\bigr)\\
&=\sum_{k,k',l,l'}
v_{l',l}^*u_{k',k}^*
\omega_{\zeta,\theta}\bigl(
\iota_2(\ket{\eta_l}\bra{\eta_{l'}})
\iota_1(\ket{\xi_k}\bra{\xi_{k'}})\bigr)\\
&=\biggl(\sum_{k,k',l,l'}
u_{k',k}v_{l',l}
\omega_{\theta,\zeta}\bigl(
\iota_1(\ket{\xi_{k'}}\bra{\xi_k})
\iota_2(\ket{\eta_{l'}}\bra{\eta_l})\bigr)
\biggr)^*\\
&=(\id\tens\omega_{\theta,\zeta})
\Bigl(
\bigl((\id\tens\iota_1)u\bigr)\bigl((\id\tens\iota_2)v\bigr)
\Bigr)^*
=(\id\tens\omega_{\zeta,\theta})\bigl((u\otop v)^*\bigr)
\end{aligned}
}}
\]
in which we used braided anti-multiplicativity of $S$: $S(ab)=\zeta^{-\deg(a)\deg(b)}S(b)S(a)$ for homogeneous $a,b\in\cA$, see Theorem \ref{thm-prop6}. The case of the matrix elements $(\id\tens\Omega)(u\otop v)$ for general $\Omega\in\B(\sH\tens\sK)^*$ follows again by linearity.
\end{proof}

In the next result we will use the maps
\begin{align*}
S\boxtimes\id&\colon\cA\algboxtimes\cA\ni\iota_1(a)\iota_2(b)\longmapsto \iota_1(S(a))\iota_2(b)\in
\cA\algboxtimes\cA,\\
\id\boxtimes S&\colon \cA\algboxtimes\cA\ni \iota_1(a)\iota_2(b)\longmapsto \iota_1(a)\iota_2(S(b))\in
\cA\algboxtimes\cA,\\
\mu &\colon \cA\algboxtimes\cA\ni \iota_1(a) \iota_2(b)\longmapsto ab\in \cA
\end{align*}
which are all defined with help of Proposition \ref{prop-lemma10}: $S\boxtimes\id=\phi_{\tens,\boxtimes}\comp(S\tens\id)\comp\phi_{\tens,\boxtimes}^{-1}$, $\id\boxtimes S=\phi_{\tens,\boxtimes}\comp(\id\tens S)\comp\phi_{\tens,\boxtimes}^{-1}$, and
$\mu=m\comp\phi_{\tens,\boxtimes}^{-1}$, where $m$ is the multiplication $\cA\algtens\cA\to\cA$.

\begin{theorem}\label{thm-prop8}
For any $a\in\cA$ we have
\[
\mu (S\boxtimes \id)\Delta_A(a)=\mu(\id\boxtimes S)\Delta_A(a)=\eps(a)\I.
\]
\end{theorem}

\begin{proof}
Let $a$ be an arbitrary element of the basis \eqref{eq-bazacA} and assume $a\neq\I$ (this case is trivial). Furthermore let $u$ be the fundamental representation of $\SU_q(2)$. Then we can find $n\in\NN$ and $\omega\in\B((\CC^2)^{\tens n})^*$ such that $a=(\id\tens\omega)v$ for $v=u^{\otop n}$. Furthermore $\omega$ is a linear combination of vector functionals $\omega_{\xi,\eta}$ for some $\xi,\eta\in(\CC^2)^{\tens n}$.

By Lemma \ref{lem-lemma14}, we have $S((\id\tens\omega)v)=(\id\tens\omega)(v^*)$ and Proposition \ref{prop-lemma19} gives $\eps((\id\tens\omega)v)=\omega(\I)$. Choose an orthonormal basis $\{\xi_k\}_k$ in $(\CC^2)^{\tens n}$. The calculation
\[
\resizebox{\textwidth}{!}{\ensuremath{\displaystyle
\begin{aligned}
\mu\Bigl((S\boxtimes\id)\Delta_A\bigl((\id\tens\omega_{\xi,\eta})v\bigr)\Bigr)
&=\bigl(\mu\comp(S\boxtimes\id)\bigr)
\Bigl((\id\tens\id\tens\omega_{\xi,\eta})\bigl((\iota_1\tens\id)v\bigr)
\bigl((\iota_2\tens\id)v\bigr)\Bigr)\\
&=\sum_k\bigl(\mu\comp(S\boxtimes\id)\bigr)\bigl(\iota_1(v_{\xi,\xi_k})
\iota_2(v_{\xi_k,\eta})\bigr)
=\sum_k\mu\bigl(\iota_1(v_{\xi_k,\xi}^*)\iota_2(v_{\xi_k,\eta})\bigr)\\
&=\sum_k v_{\xi_k,\xi}^*v_{\xi_k,\eta}
=\sum_k\bigl((\id\tens\omega_{\xi_k,\xi})v\bigr)^*
\bigl((\id\tens\omega_{\xi_k,\eta})v\bigr)\\
&=\sum_k\bigl((\id\tens\omega_{\xi,\xi_k})v^*\bigr)
\bigl((\id\tens\omega_{\xi_k,\eta})v\bigr)\\
&=(\id\tens\omega_{\xi,\eta})(v^* v)=\is{\xi}{\eta}\I=\eps(v_{\xi,\eta})\I
\end{aligned}
}}
\]
proves the first equation for $\omega=\omega_{\xi,\eta}$ and the general case follows by linearity. The second equation holds by a similar reasoning.
\end{proof}

\begin{proposition}\label{propprop5}
For any finite-dimensional unitary representation $u$ of $\SU_q(2)$ and
$\omega\in\B(\sH)^*$ we have $S((\id\tens\omega)u)=(\id\tens\omega)(u^*)$.
\end{proposition}

\begin{proof}
Choose an orthonormal basis $\{\xi_k\}$ ($k\in\{1,\dotsc,\dim{\sH}\}$) of $\sH$, write $u$ in the form
\[
u=\sum_{i,j=1}^{\dim{\sH}}u_{i,j}\tens \ket{\xi_i}\bra{\xi_j}
\]
and take $n,m\in\{1,\dotsc,\dim{\sH}\}$. We have $(\Delta_A\tens\id)u=((\iota_1\tens\id)u)((\iota_2\tens\id)u)$, i.e.
\[
\Delta_A(u_{n,m})
=\sum_{k=1}^{\dim{\sH}}\iota_1(u_{n,k})\iota_2(u_{k,m}).
\]
By Theorem \ref{thm-prop8}
\[
\delta_{n,m}\I
=\eps\bigl(u_{n,m})\I
=\mu\bigl((S\boxtimes\id)\Delta_A(u_{n,m})\bigr)
=\sum_{k=1}^{\dim{\sH}}S(u_{n,k})u_{k,m}
\]
and hence
\[
\sum_{m=1}^{\dim{\sH}}\delta_{n,m}u_{p,m}^*
=\sum_{k,m=1}^{\dim{\sH}}S(u_{n,k})u_{k,m}u_{p,m}^*
\]
for $p\in\{1,\dotsc,\dim{\sH}\}$. Since $u$ is unitary, this simplifies to $u_{p,n}^*=S(u_{n,p})$. Writing an arbitrary functional $\omega$ as a linear combination of $\{\omega_{\xi_n,\xi_m}\}$ completes the proof.
\end{proof}

\subsection{Braided Hopf algebra structure on \texorpdfstring{$\cA$}{A}}\label{sect-BrHopf}

In this section we will show that $\cA$ is a braided Hopf algebra, i.e.~a Hopf algebra in $\mathcal{C}=\sideset{_\ZZ^\ZZ}{}{\operatorname{\mathcal{YD}}}$ in the sense of \cite[Definition 1.6.6]{HeckenbergerSchneider}.\footnote{Our \cst-algebraic approach favors the compact group $\TT$, but the definitions in \cite{HeckenbergerSchneider} are such that the discrete group $\ZZ$ is more prominent.}
This requires shifting point of view a little: for example in the algebraic setting of \cite{HeckenbergerSchneider} the algebraic tensor product is used (with twisted multiplication), while in analytic setting it is better to have twisted version $\algboxtimes$.

First, we need to understand $\cA$ as an object in $\sideset{_\ZZ^\ZZ}{}{\operatorname{\mathcal{YD}}}$, i.e.~we need to argue that $\cA$ has the structure of a Yetter-Drinfeld module over $\CC[\ZZ]$. First, we have the $\ZZ$-grading
\[
\cA=\bigoplus_{k\in\ZZ}
\bigl\{a\in\cA\,\bigr|\bigl.\,\deg(a)=k\bigr\},
\]
where $\deg$ corresponds to the $\TT$-action $\rho^\cA$. Next there is the $\CC[\ZZ]$-module structure given by
\[
k\cdot a=\rho^\cA_{\zeta^{k}}(a)
=(\operatorname{ev}_{\zeta^{-k}}\tens\id)\rho^\cA(a),\qquad a\in\cA.
\]
Let us explain why the $\CC[\ZZ]$-module structure is given by this formula. The action of $\D(\TT)=\TT\times\ZZ$ on $A$ restricts to an action on $\cA$ and then further to actions $\rho^\cA$ and $\tilde{\rho}^\cA$ of the subgroups $\TT$ and $\ZZ$ of $\D(\TT)$ (see Sections \ref{sect-Antipode} and \ref{sect-BrFlip}).

These structures make $\cA$ into Yetter-Drinfeld module over $\CC[\ZZ]$ (\cite[Definition 1.4.1]{HeckenbergerSchneider}) hence an object of $\sideset{_\ZZ^\ZZ}{}{\operatorname{\mathcal{YD}}}$. Morphisms in $\sideset{_\ZZ^\ZZ}{}{\operatorname{\mathcal{YD}}}$ are $\ZZ$-graded, $\ZZ$-equivariant linear maps (so in our case $\TT$-equivariant linear maps).
%

Next we claim that $\cA$ is an algebra, a coalgebra, a bialgebra, and finally a Hopf algebra in $\sideset{_\ZZ^\ZZ}{}{\operatorname{\mathcal{YD}}}$ (\cite[Page 40]{HeckenbergerSchneider}):
\begin{itemize}
\item Algebra: the multiplication
\[
\widetilde{\mu}\colon\cA\algtens\cA\ni a\tens b\longmapsto ab\in\cA
\]
and the unit map $\CC\ni z\mapsto z\I \in\cA$ are morphisms in $\sideset{_\ZZ^\ZZ}{}{\operatorname{\mathcal{YD}}}$ because the Yetter-Drinfeld structure on the tensor product is diagonal and the $\TT$-action is by homomorphisms \cite[Page 29]{HeckenbergerSchneider}. Algebra axioms clearly hold.
\item Coalgebra: as the comultiplication and counit we take
\[
\widetilde{\Delta}\colon\cA\ni a\longmapsto\phi_{\tens,\boxtimes}^{-1}\bigl(\Delta_A(a)\bigr)\in\cA\algtens\cA
\]
and $\eps\colon\cA\to\CC$. Since the comultiplication of $\SU_q(2)$ is equivariant, and so is $\phi_{\tens,\boxtimes}$, we conclude that above $\widetilde{\Delta}$ is equivariant. $\eps$ is equivariant by Proposition \ref{prop-lemma19}. Axioms of a coalgebra hold, because of coassociativity of $\Delta_A$ and Proposition \ref{prop-lemma19}.
\item Bialgebra: we need to argue that $\widetilde{\Delta}$ and $\eps$ are algebra maps. It is clear that $\eps$ is an algebra map (it is unital and multiplicative). To show that $\widetilde{\Delta}$ is an algebra map, we need to prove
\[
\widetilde{\Delta}\comp\widetilde{\mu}
=(\widetilde{\mu}\tens\widetilde{\mu})\comp(\id\tens c_{\cA,\cA}\tens\id)\comp(\widetilde{\Delta}\tens\widetilde{\Delta})
\]
(see \cite[Definition 1.6.3]{HeckenbergerSchneider}) which we check on simple tensors using Sweedler notation:
\begin{align*}
\widetilde{\Delta}\bigl(\widetilde{\mu}(a\tens b)\bigr)
&=\phi_{\tens,\boxtimes}^{-1}\Delta_A(ab)
=\phi_{\tens,\boxtimes}^{-1}\bigl(\Delta_A(a)\Delta_A(b)\bigr)\\
&=\sum \phi_{\tens,\boxtimes}^{-1}\bigl(
\iota_1(a_{(1)})\iota_2(a_{(2)})\iota_1(b_{(1)})\iota_2(b_{(2)})
\bigr)\\
&=\sum \zeta^{-\deg(a_{(2)})\deg(b_{(1)})}
\phi_{\tens,\boxtimes}^{-1}\bigl(\iota_1(a_{(1)}b_{(1)})\iota_2(a_{(2)}b_{(2)})\bigr)\\
&=\sum \zeta^{-\deg(a_{(2)})\deg(b_{(1)})}a_{(1)}b_{(1)}\tens a_{(2)}b_{(2)}
\end{align*}
and
\begin{align*}
(\widetilde{\mu}\tens\widetilde{\mu})&\Bigl(
(\id\tens c_{\cA,\cA}\tens\id)\bigl(
(\widetilde{\Delta}\tens\widetilde{\Delta})(a\tens b)
\bigr)
\Bigr)\\
&=\sum (\widetilde{\mu}\tens\widetilde{\mu})\bigl(
(\id\tens c_{\cA,\cA}\tens\id)
(a_{(1)}\tens a_{(2)}\tens b_{(1)}\tens b_{(2)})
\bigr)\\
&=\sum \zeta^{-\deg(a_{(2)})\deg(b_{(1)})}
(\widetilde{\mu}\tens\widetilde{\mu})(a_{(1)}\tens b_{(1)}\tens a_{(2)}\tens b_{(2)})\\
&=\sum \zeta^{-\deg(a_{(2)})\deg(b_{(1)})}
a_{(1)}b_{(1)}\tens a_{(2)}b_{(2)}.
\end{align*}
\item Hopf algebra: as the antipode we take $S\colon\cA\to\cA$, which we know to be linear and $\TT$-equivariant. The antipode condition is equivalent to the conclusion of Theorem \ref{thm-prop8}.
\end{itemize}

Thus we proved the following:

\begin{proposition}
The vector space $\cA$ with multiplication $\widetilde{\mu}$, comultiplication $\widetilde{\Delta}$, unit $\I$, counit $\eps$, the above Yetter-Drinfeld structure and antipode $S$ is a Hopf algebra in $\sideset{_\ZZ^\ZZ}{}{\operatorname{\mathcal{YD}}}$.
\end{proposition}

\begin{remark}
We are not aware of any literature proposing a definition of a braided Hopf $*$-algebra. In the case of $(\cA,\widetilde{\mu},\I,\widetilde{\Delta},\eps,S)$ we do have $*\comp S\comp *\comp S=\id$ (Proposition \ref{prop-prop10}), but $\phi_{\tens,\boxtimes}$ (and consequently $\widetilde{\Delta}$) does not work well with (the usual) involution on the tensor product.

However, if we define a new involution $\underline{\ast}$ on $\cA\algtens\cA$ so that $(a\tens b)^{\underline{\ast}}=\zeta^{-\deg(a)\deg(b)}a^*\tens b^*$ (for homogeneous elements $a,b\in\cA$), then
\[
\widetilde{\Delta}(a^*)
=\phi_{\tens,\boxtimes}^{-1}\biggl(\sum\iota_2(a_{(2)}^*)\iota_1(a_{(1)}^*)\biggr)
=\sum\zeta^{-\deg(a_{(1)})\deg(a_{(2)})}a_{(1)}^*\tens a_{(2)}^*
=\widetilde{\Delta}(a)^{\underline{\ast}}
\]
and
\begin{align*}
\widetilde{\mu}(a\tens b)^*&=(ab)^*=b^*a^*=
\zeta^{\deg(a)\deg(b)}\widetilde{\mu}\bigl(\zeta^{-\deg(a)\deg(b)} b^*\tens a^*\bigr)\\
&=\zeta^{\deg(a)\deg(b)}\widetilde{\mu}\bigl((b\tens a)^{\underline{\ast}}\bigr)
=\widetilde{\mu}\bigl(
(\zeta^{-\deg(a)\deg(b)}b\tens a)^{\underline{\ast}}\bigr)
=\widetilde{\mu}\bigl(
c_{\cA,\cA}(a\tens b)^{\underline{\ast}}\bigr)
\end{align*}
i.e.~$\underline{\ast}$ is anti-multiplicative in the braided sense.
Notice also that
\begin{align*}
\bigl(\phi_{\tens,\boxtimes}^{-1}\comp\ast\comp\phi_{\tens,\boxtimes}\bigr)&(a\tens b)
=\phi_{\tens,\boxtimes}^{-1}\Bigl(\bigl(\iota_1(a)\iota_2(b)\bigr)^*\Bigr)=
\phi_{\tens,\boxtimes}^{-1}\bigl(\iota_2(b^*)\iota_1(a^*)\bigr)\\
&=\zeta^{-\deg(a)\deg(b)}
\phi_{\tens,\boxtimes}^{-1}\bigl(\iota_1(a^*)\iota_2(b^*)\bigr)=
\zeta^{-\deg(a)\deg(b)}a^*\tens b^*=(a\tens b)^{\underline{\ast}},
\end{align*}
which justifies the formula for $\underline{\ast}$. Thus \emph{with the involution $\underline{\ast}$ on tensor product}, $\cA$ could be called a \emph{braided Hopf $*$-algebra}.
\end{remark}

\subsection{The unitary antipode and the polar decomposition of the antipode}\label{sect-UApolar}

For \emph{real} $q$ satisfying $0<|q|<1$, the unitary antipode $R$ of $\SU_q(2)$ is characterized by $R(\alpha)=\alpha^*$, $R(\gamma)=-\operatorname{sgn}(q)\gamma$. It is a bounded, linear, anti-multiplicative and $*$-preserving map on $\C(\SU_q(2))$, which provides a canonical decomposition of the antipode into its bounded and unbounded part: $S=R\comp\tau_{-\ii/2}$. We will establish a similar decomposition in the braided case. It will turn out, however, that there is a need for a third map.

Now we go back to our standing assumption that $q$ is complex with $0<|q|<1$. Recall that $\{\alpha^n\gamma^m{\gamma^*}^k\}_{n,m,k\ge 0}\cup\{{\alpha^*}^n\gamma^m{\gamma^*}^k\}_{n\ge 1,m,k\ge 0}$ is a basis of $\cA$ (Corollary \ref{cor-bazacA}).

\begin{lemma}
We have $\bh(\alpha^n\gamma^m{\gamma^*}^k)=
\bh({\alpha^*}^n\gamma^m{\gamma^*}^k)
=\delta_{n,0}\delta_{m,k}\tfrac{1-|q|^2}{1-|q|^{2(m+1)}}$ for all $n,m,k\in\ZZ_+$.
\end{lemma}

\begin{proof}
This follows from the definition of $\bh$ and \cite[Theorem 4.1]{ZhangZhao}.
\end{proof}

In what follows, we will use complex powers of $\zeta$. They are defined using the principal value of logarithm, i.e.~$\zeta^z=\ee^{z \operatorname{Log}(\zeta)}$ for $z\in\CC$, where $\operatorname{Log}(\zeta)\in\ii]-\pi,\pi]$.

\begin{definition}
Define linear map $\vartheta\colon\cA\to\cA$ by $\vartheta(a)=\zeta^{\frac{1}{2}\deg(a)^2}a$ for homogeneous elements $a$. We will call $\vartheta$ the \emph{residual mapping}.
\end{definition}

\begin{remarks}\hspace*{\fill}
\begin{enumerate}
\item The linear map $\vartheta$ is well defined and invertible with inverse $\vartheta^{-1}$ mapping each homogeneous $a\in\cA$ to $\zeta^{-\frac{1}{2} \deg(a)^2}a$.
\item The reason for the unusual name of $\vartheta$ is contained in Theorem \ref{thm-prop3} below. It is the ``residue'' left after composing the unitary antipode (see below), the inverse of the antipode and the analytic generator $\tau_{-\ii/2}$ of the scaling group. For compact quantum groups it is equal to the identity (\cite[Theorem 2.6(4)]{CQGs}).
\end{enumerate}
\end{remarks}

\begin{lemma}\label{lem-lemma18}
For homogeneous $a,b\in\cA$ we have
\begin{align*}
\deg\bigl(\vartheta(a)\bigr)&=\deg\bigl(\vartheta^{-1}(a)\bigr)=\deg(a),\\
\vartheta(ab)&=\zeta^{\deg(a)\deg(b)}\vartheta(a)\vartheta(b),\\
\vartheta^{-1}(ab)&=\zeta^{-\deg(a)\deg(b)}\vartheta^{-1}(a)\vartheta^{-1}(b),\\
\vartheta(a)^*&=\zeta^{-\deg(a)^2}\vartheta(a^*),\\
\vartheta^{-1}(a)^*&=\zeta^{\deg(a)^2}\vartheta^{-1}(a^*).
\end{align*}
Furthermore, $\bh\comp\vartheta=\bh\comp\vartheta^{-1}=\bh$ and $\vartheta$ is $\TT$-equivariant.
\end{lemma}

\begin{proof}
The first two equations are immediate. The next two properties follow from the calculation:
\[
\vartheta(ab)=\zeta^{\frac{1}{2}\deg(ab)^2}ab
=\zeta^{ \deg(a)\deg(b)}
\zeta^{\frac{1}{2}\deg(a)^2}a\zeta^{\frac{1}{2}\deg(b)^2}b
=\zeta^{\deg(a)\deg(b)}\vartheta(a)\vartheta(b)
\]
and hence
\begin{align*}
\vartheta^{-1}\bigl(\vartheta(a)\vartheta(b)\bigr)
&=\zeta^{-\deg(a)\deg(b)}\vartheta^{-1}\bigl(\vartheta(ab)\bigr)
=\zeta^{-\deg(a)\deg(b)}ab\\
&=\zeta^{-\deg(\vartheta^{-1}(a))\deg(\vartheta^{-1}(b))}
\vartheta^{-1}\bigl(\vartheta(a)\bigr)
\vartheta^{-1}\bigl(\vartheta(b)\bigr).
\end{align*}
Observe that as $\operatorname{Log}(\zeta)\in \ii \RR$, we have $\overline{\zeta^{x}}=\overline{e^{x \operatorname{Log}(\zeta)}}=e^{-x \operatorname{Log}(\zeta)}=\zeta^{-x}$  for $x\in\RR$. Using this and the fact that $\deg(a^*)=-\deg(a)$ we obtain
\[
\vartheta(a)^*
=\bigl(\zeta^{\frac{1}{2}\deg(a)^2}a\bigr)^*
=\zeta^{-\frac{1}{2}\deg(a)^2}a^*
=\zeta^{-\deg(a)^2}\bigl(\zeta^{\frac{1}{2}\deg(a^*)^2}a^*\bigr)
=\zeta^{-\deg(a)^2}\vartheta(a^*)
\]
and
\[
\vartheta^{-1}\bigl(\vartheta(a)\bigr)^*
=a^*=\zeta^{\deg(a)^2}\vartheta^{-1}
\bigl(\zeta^{-\deg(a)^2}\vartheta(a^*)\bigr)
=\zeta^{\deg(a)^2}\vartheta^{-1}\bigl(\vartheta(a)^*\bigr).
\]

To prove invariance of $\bh$ under $\vartheta$ observe that for all $n,m,k\ge 0$ we have
\[
\deg(\alpha^n\gamma^m{\gamma^*}^k)=m-k
\]
and hence by Lemma \ref{lem-lemma18}
\[
\bh\bigl(\vartheta(\alpha^n \gamma^m{\gamma^*}^k)\bigr)
=\zeta^{\frac{1}{2}(m-k)^2}\bh(\alpha^n\gamma^m{\gamma^*}^k)
=\zeta^{\frac{1}{2}(m-k)^2}\delta_{m,k}\bh(\alpha^n\gamma^m{\gamma^*}^k)=\bh(\alpha^n\gamma^m{\gamma^*}^k).
\]
Similarly we check equality on basic elements ${\alpha^*}^n\gamma^m{\gamma^*}^k$ and $\bh\comp\vartheta=\bh$ follows. Consequently also $\bh\comp\vartheta^{-1}=\bh$.
\end{proof}

It is clear from the definition (Section \ref{sect-tau}) that all elements of $\cA$ are \emph{analytic} for the scaling group, i.e.~for $a\in\cA$ the function $\RR\ni t\mapsto\tau_t(a)\in A$ has an entire analytic extension which we denote $\CC\ni z\mapsto\tau_z(a)\in A$, thus defining the operator $\tau_z\colon\cA\to\cA$. It is known that a continuous one-parameter group of automorphisms of a \cst-algebra is determined by $\tau_z$ for any $z\in\CC\setminus\RR$ (\cite{CioranescuZsido}) and hence $\tau_\ii$ is called the \emph{analytic generator} of $\tau$. Extending this terminology we will refer to $\tau_{-\ii/2}$ as the analytic generator of $\tau$.

\begin{definition}\label{def-R}
We define the \emph{unitary antipode} $R\colon\cA\to\cA$ of $\SU_q(2)$ by $R=S\comp\tau_{\ii/2}\comp\vartheta$.
\end{definition}

\begin{theorem}\label{thm-prop3}
The map $R$ is unital, linear, and $*$-anti-multiplicative. It satisfies
\begin{equation}\label{eq-eq63}
R(\alpha)=\alpha^*,\quad
R(\alpha^*)=\alpha,\quad
R(\gamma)=-\zeta^{\frac{1}{2}}\tfrac{\overline{q}}{|q|}\gamma,\quad
R(\gamma^*)=-\zeta^{\frac{1}{2}}\tfrac{\overline{q}}{|q|}\gamma^*.
\end{equation}
Moreover $R$ is invertible with $R^2=\id$.

We have $\bh\comp R=\bh$ and $R$ is $\TT$-equivariant. Furthermore $R$ extends to a bounded map on the \cst-algebra $A$ and a bounded, normal map on the von Neumann algebra $M$ with analogous properties.
\end{theorem}

\begin{remark}
With the conventions we chose for the powers of $\zeta$ we have the following: if $\zeta\neq 1$ (i.e.~$q$ is not real) then $\zeta^{\frac{1}{2}}\tfrac{\overline{q}}{|q|}=\operatorname{sgn}(\operatorname{Re}{q})$ except when $\operatorname{Re}{q}=0$. This formula also holds for $q=|q|\ee^{\pm\ii\frac{\pi}{2}}$ if we declare $\operatorname{sgn}(\operatorname{Re}{q})=\pm 1$ in those cases. This means that the formulas for the values of the unitary antipode on the generators $\alpha$, $\gamma$, $\alpha^*$ and $\gamma^*$ are in fact identical to those known from the case of real $q$ (see e.g.~\cite[Proof of Proposition 7.4]{modular}).
\end{remark}

\begin{proof}[Proof of Theorem \ref{thm-prop3}]
The linear map $R$ clearly is equivariant, as all the involved maps are $\TT$-equivariant. A similar argument shows that $R$ preserves the Haar measure. The equalities \eqref{eq-eq63} hold by construction.

Now take homogeneous $a,b\in\cA$. Then using already established properties we obtain
\begin{align*}
R(ab)
&=S\Bigl(\tau_{\ii/2}\bigl(\vartheta(ab)\bigr)\Bigr)
=\zeta^{\deg(a)\deg(b)}S\Bigl(\tau_{\ii/2}\bigl(\vartheta(a)\vartheta(b)\bigr)\Bigr)\\
&=\zeta^{\deg(a)\deg(b)}S\Bigl(\tau_{\ii/2}\bigl(\vartheta(a)\bigr)\tau_{\ii/2}\bigl(\vartheta(b)\bigr)\Bigr)\\
&=\zeta^{\deg(a)\deg(b)}
\zeta^{-\deg(\tau_{\ii/2}\vartheta(a))\deg(\tau_{\ii/2}\vartheta(b))}
S\Bigl(\tau_{\ii/2}\bigl(\vartheta(b)\bigr)\Bigr)
S\Bigl(\tau_{\ii/2}\bigl(\vartheta(a)\bigr)\Bigr)\\
&=S\Bigl(\tau_{\ii/2}\bigl(\vartheta(b)\bigr)\Bigr)
S\Bigl(\tau_{i/2}\bigl(\vartheta(a)\bigr)\Bigr)=R(b)R(a).
\end{align*}

Assume that for $a,b\in\cA$ we have $R(a)^*=R(a^*)$ and $R(b)^*=R(b^*)$. Then
\[
R(ab)^*=\bigl(R(b)R(a)\bigr)^*=R(a)^*R(b)^*=R(a^*)R(b^*)=R(b^* a^*)
=R\bigl((ab)^*\bigr).
\]
Since also
\[
R(\I)=\I,\quad
R(\alpha)^*=R(\alpha^*),\quad
R(\alpha^*)^*=R(\alpha)
\]
and
\begin{align*}
R(\gamma)^*
&=\bigl(-\zeta^{\frac{1}{2}}\tfrac{\overline{q}}{|q|}\gamma\bigr)^*
=-\zeta^{-\frac{1}{2}}\tfrac{ q}{|q|}\gamma^*
=-\zeta^{\frac{1}{2}}\tfrac{\overline{q}}{|q|}\gamma^*
=R(\gamma^*)\\
R(\gamma^*)^*
&=\bigl(-\zeta^{\frac{1}{2}}\tfrac{\overline{q}}{|q|}\gamma^*\bigr)^*
=-\zeta^{-\frac{1}{2}}\tfrac{q}{|q|}\gamma
=-\zeta^{\frac{1}{2}}\tfrac{\overline{q}}{|q|}\gamma
=R(\gamma)
\end{align*}
(as $\zeta^{-\frac{1}{2}}q=\zeta^{-\frac{1}{2}} \zeta\, \overline{q}=\zeta^{\frac{1}{2}}\overline{q}$) and we conclude that $R$ is $*$-preserving.

Next, since clearly $R^2(\alpha)=\alpha$, $R^2(\alpha^*)=\alpha^*$ and
\[
R^2(\gamma)=\bigl(-\zeta^{\frac{1}{2}}\tfrac{\overline{q}}{|q|}\bigr)^2\gamma
=\zeta\tfrac{\overline{q}\,\overline{q}}{q\overline{q}}\gamma=\gamma,\qquad
R^2(\gamma^*)=R^2(\gamma)^*=\gamma^*
\]
we have $R^2=\id$.

Next let us argue that $R$ extends to continuous map on $A$ and a normal, continuous map on $M$. Define
an antilinear map $\cJhat\colon\Ltwo(\SU_q(2))\to\Ltwo(\SU_q(2))$ via
\[
\cJhat\Lambda_\bh(x)=\Lambda_\bh\bigl(R(x)^*\bigr),\qquad x\in\cA.
\]
First, since
\[
\bigl\|\Lambda_\bh\bigl(R(x)^*\bigr)\bigr\|^2
=\bh\bigl(R(x)R(x)^*\bigr)=\bh\bigl(R(x^*x)\bigr)
=\bh(x^*x)=\bigl\|\Lambda_\bh(x)\bigr\|^2,
\]
we see that $\cJhat$ extends to a well defined, antilinear isometry. It has dense image, hence $\cJhat$ is antiunitary. As $R^2=\id$ we have $\cJhat^*=\cJhat^{-1}=\cJhat$.

For $x,y\in\cA$ we calculate:
\begin{align*}
\cJhat x^*\cJhat\Lambda_\bh(y)
&=\cJhat\Lambda_\bh\bigl(x^*R(y^*)\bigr)
=\Lambda_\bh\Bigl(R\bigl(x^*R(y)^*\bigr)^*\Bigr)\\
&=\Lambda_\bh\Bigl(R\bigl(R(y)x\bigr)\Bigr)
=\Lambda_\bh\bigl(R(x)y\bigr)=R(x)\Lambda_\bh(y)
\end{align*}
which proves that for any $x\in\cA$ we have
\begin{equation}\label{eq-eq37}
R(x)=\cJhat x^*\cJhat.
\end{equation}
Now we can simply define the extension of $R$ to $A$ and $M$ by \eqref{eq-eq37}. Clearly these maps are (normal), unital, linear, $*$-preserving, anti-multiplicative and involutive.
\end{proof}

\begin{remark}
The expected equality
\begin{equation}\label{eq-fail}
\Delta_A\comp R=(R\boxtimes R)\comp\chi^\boxtimes\comp\Delta_A
\end{equation}
(with ``obvious'' $R\boxtimes R$) does \emph{not} hold. Indeed, we have
\begin{align*}
\Delta_A\bigl(R(\alpha)\bigr)&=\Delta_A(\alpha^*)
=\bigl(\iota_1(\alpha)\iota_2(\alpha)-q\iota_1(\gamma^*)\iota_2(\gamma)\bigr)^*
=\iota_2(\alpha^*)\iota_1(\alpha^*)-\overline{q}\iota_2(\gamma^*)\iota_1(\gamma)\\
&=\iota_1(\alpha^*)\iota_2(\alpha^*)-\overline{q}
\zeta^{-\deg(\gamma^*)\deg(\gamma)}\iota_1(\gamma)\iota_2(\gamma^*)
=\iota_1(\alpha^*)\iota_2(\alpha^*)-q\iota_1(\gamma)\iota_2(\gamma^*)
\end{align*}
while
\begin{align*}
\bigl((R\boxtimes R)\comp\chi^\boxtimes\:\comp\:&\Delta_A\bigr)(\alpha)
=(R\boxtimes R)\Bigl(\chi^\boxtimes\bigl(
\iota_1(\alpha)\iota_2(\alpha)-q\iota_1(\gamma^*)\iota_2(\gamma)\bigr)\Bigr)\\
&=(R\boxtimes R)\bigl(
\iota_1(\alpha)\iota_2(\alpha)
-q\zeta^{-\deg(\gamma)\deg(\gamma^*)}
\iota_1(\gamma)\iota_2(\gamma^*)
\bigr)\\
&=(R\boxtimes R)\bigl(
\iota_1(\alpha)\iota_2(\alpha)-q\zeta\,\iota_1(\gamma) \iota_2(\gamma^*)
\bigr)\\
&=\iota_1(\alpha^*)\iota_2(\alpha^*)
-q\zeta\,\iota_1\bigl(-\zeta^{\frac{1}{2}}\tfrac{\overline{q}}{|q|}\gamma\bigr)
\iota_2\bigl(-\zeta^{\frac{1}{2}}\tfrac{\overline{q}}{|q|}\gamma^*\bigr)\\
&=\iota_1(\alpha^*)\iota_2(\alpha^*)-q\zeta\,\tfrac{q}{\overline{q}}
\,\tfrac{\overline{q}\overline{q}}{q\overline{q}}
\iota_1(\gamma)\iota_2(\gamma^*)
=\iota_1(\alpha^*)\iota_2(\alpha^*)-q\zeta\iota_1(\gamma)\iota_2(\gamma^*).
\end{align*}
The expected equalities hold for the antipode (Proposition \ref{prop-lemma12}), and for analytic generator of the scaling group (Proposition \ref{prop-tau}\eqref{prop-tau1}). So the failure of
\eqref{eq-fail} is due to properties of the residual mapping $\vartheta$.
\end{remark}

\begin{appendices}

\section{Appendices}

\subsection{Braided tensor products}\label{app-BraidedTP}

In this section we will show that the definitions of braided tensor product given in \cite{BraidedSU2} and \cite{MeyerRoyWoronowiczI,MeyerRoyWoronowiczII} are equivalent in the situation when both constructions apply. In the end, we want to be compatible with the general von Neumann algebraic approach to braided tensor product introduced in \cite{DeCommerKrajczok}. It is very close to the conventions of \cite{MeyerRoyWoronowiczII}, however \cite{DeCommerKrajczok} uses left actions, whereas \cite{MeyerRoyWoronowiczI,MeyerRoyWoronowiczII} use right actions. This difference is not essential, see \cite[Section 9.2]{DeCommerKrajczok}, where it is explained how to translate construction of \cite{DeCommerKrajczok} to the setting of right actions. We note that \cite{BraidedSU2} also uses left actions.

The definition of the braided tensor product according to \cite{BraidedSU2} was already recalled in Section \ref{sect-SUq2}. It depends on a fixed complex number $\zeta$ of absolute value $1$ (in the context of $\SU_q(2)$, one takes it to be $\zeta=q/\overline{q}$, where $q$ is the deformation parameter). The product is a bifunctor on the category of \cst-algebras equipped with an action of $\TT$ with equivariant morphisms of \cst-algebras. In order to distinguish it from the more general tensor product described below, in this section we will denote this braided tensor product by $\boxtimes_\zeta$. The inclusion morphisms for this construction of the braided tensor product will be in this section denoted by $\jmath_\sX$ and $\jmath_\sY$.

The approach to the notion of braided tensor product used in \cite{MeyerRoyWoronowiczII, MeyerRoyWoronowiczI}, introduced in \cite[Section 3.2]{MeyerRoyWoronowiczI} (in \cst-setting) can be realized on various levels of generality, one of which is to consider \cst-algebras endowed with an action of a quasitriangular quantum group. A special case of this is the category of \cst-algebras with an action of the Drinfeld double $\D(\GG)$ (\cite[Section 4]{PodlesWoronowiczQLG}, \cite{Yamanouchi}) of a fixed locally compact quantum group $\GG$. The even more special case relevant to the braided quantum group $\SU_q(2)$ is that of the category of \cst-algebras with an action of $\D(\TT)$ with its canonical $\mathrm{R}$-matrix $\hh{R}\in\M(\C_0(\hh{\D(\TT)})\tens\C_0(\hh{\D(\TT)}))$
\begin{equation}\label{eq-eq25}
\hh{R}=\ww^\TT_{23}
\end{equation}
seen as an element of $\M(\mathrm{c}_0(\ZZ)\tens\C(\TT)\tens\mathrm{c}_0(\ZZ)\tens\C(\TT))=\M(\C_0(\hh{\D(\TT)})\tens\C_0(\hh{\D(\TT)}))$.

Recall that an action of $\D(\TT)$ is equivalent to existence of the so called Yetter-Drinfeld structure (this is proved in \cite[Proposition 3.2]{NestVoigt}, and holds also for von Neumann algebras \cite[Proposition 2.9]{DeCommerKrajczok}), which in this case of the classical, abelian group $\TT$ is equivalent to two commuting actions of $\TT$ and $\ZZ=\hh{\TT}$. In order to keep the notation consistent with the one used in Section \ref{sect-SUq2} for a \cst-algebra $\sX$ we will write $\brho^\sX$ for the action of $\D(\TT)$ and $\rho^\sX$, $\tilde{\rho}^\sX$ for the corresponding actions of $\TT$ and $\ZZ$ respectively. Clearly $\rho^\sX$ and $\tilde{\rho}^\sX$ are restrictions of $\brho$ to the subgroups $\TT$ and $\ZZ$ of $\D(\TT)$ respectively.

The authors of \cite{MeyerRoyWoronowiczI, MeyerRoyWoronowiczII} use a different definition than von Neumann-algebraic framework \cite{DeCommerKrajczok}, however in \cite[Section 4]{MeyerRoyWoronowiczII} there is a convenient description of the braided tensor product, which allows us to pass to left conventions used in \cite{DeCommerKrajczok} and \cite{BraidedSU2}. Let $\sX$, $\sY$ be \cst-algebras with (left) $\D(\TT)=\TT\times\ZZ$ actions and $(\ph,U^{\sX})$, $(\psi,U^{\sY})$ faithful covariant representations on Hilbert spaces $\sH$ and $\sK$ respectively, i.e.~$\ph$ and $\psi$ are faithful, non-degenerate representations of $\sX$ and $\sY$ on $\sH$ and $\sK$ respectively and $U^\sX$, $U^\sY$ are representations of $\D(\TT)$ implementing the actions (hence $U^\sX=(\id\tens\phi_\sX)\wW^{\D(\TT)}$, $U^\sY=(\id\tens\phi_\sY)\wW^{\D(\TT)}$ for uniquely determined representations of $\C_0^\uu(\hh{\D(\TT)})=\mathrm{c}_0(\ZZ)\tens\C(\TT)$ on $\sH$ and $\sK$). Then we define $\sX\boxtimes\sY$ as
\[
\sX\boxtimes\sY=\overline{\iota_\sX(\sX)\iota_\sY(\sY)}\subseteq\B(\sH\tens\sK),
\]
where
\begin{equation}\label{eq-eq7}
\begin{aligned}
\iota_\sX(x)
&=\bigl((\phi_\sX\tens\phi_\sY)\hh{R}^\uu\bigr)(x\tens\I)\bigl((\phi_\sX\tens\phi_\sY)\hh{R}^\uu\bigr)^*,&\quad x&\in\sX,\\
\iota_\sY(y)&=\I\tens y,&\quad y&\in\sY
\end{aligned}
\end{equation}
(the apparent difference between \eqref{eq-eq7} and formulas found in \cite{MeyerRoyWoronowiczII} stem from the left/right difference, here we have translated conventions to the left setting, see the discussion in \cite[Section 9.2]{DeCommerKrajczok}). We note that faithful covariant representations always exist\footnote{Recall that \cst-algebraic actions are always assumed to be injective, and satisfy Podle\'s condition.} and in our case of $\D(\TT)$ we have $\hh{R}^u=\hh{R}$.

In what follows, all results are translated to the setting of left actions etc. As described at the end of \cite[Section 5.2]{MeyerRoyWoronowiczI}, an action $\brho^\sX\in\Mor(\sX,\C_0(\D(\TT))\tens\sX)$ can be seen as an equivariant morphism (with action on $\C_0(\D(\TT))\tens\sX$ given by $ \Delta_{D(\TT)}\tens\id$), and then by \cite[Proposition 5.6]{MeyerRoyWoronowiczI} we have an embedding
\[
\resizebox{\textwidth}{!}{\ensuremath{\displaystyle
\brho^\sX\boxtimes\brho^\sY\colon\sX\boxtimes\sY\longrightarrow
\M\bigl((\C_0(\D(\TT))\tens\sX)\boxtimes(\C_0(\D(\TT))\tens\sY)\bigr)
\simeq\M\bigl((\C_0(\D(\TT))\boxtimes\C_0(\D(\TT)))\tens\sX\tens\sY\bigr),
}}
\]
where the isomorphism (implemented by the flip on legs $2$ and $3$) holds because the action of $\D(\TT)$ is trivial on the second factor.
Concretely
\begin{equation}\label{eq-dodatkowe}
\begin{aligned}
(\brho^\sX\boxtimes\brho^\sY)(\iota_\sX(x)\iota_\sY(y))
&=\iota_{\C_0(\D(\TT))\tens\sX}\bigl(\brho^\sX(x)\bigr)
\iota_{\C_0(\D(\TT))\tens\sY}\bigl(\brho^\sY(y)\bigr)\\
&=\hh{R}_{13}\brho^\sX(x)_{12}\hh{R}^*_{13}\brho^\sY(y)_{34}
\longmapsto\hh{R}_{12}\brho^\sX(x)_{13}\hh{R}^*_{12}\brho^\sY(y)_{24}\\
&=\bigl((\biota_1\tens\id)\brho^\sX(x)\bigr)_{123}\bigl((\biota_2\tens\id)\brho^\sY(y)\bigr)_{124}.
\end{aligned}
\end{equation}
 Here we write $\biota_1,\biota_2\colon \C_0(D(\TT))\to\C_0(D(\TT))\boxtimes \C_0(D(\TT))$ for the canonical embeddings. Thus
\begin{equation}\label{eq-eq8}
\sX\boxtimes\sY\simeq\overline{\bigl((\biota_1\tens\id)\brho^\sX(\sX)\bigr)_{123}
\bigl((\biota_2\tens\id)\brho^\sY(\sY)\bigr)_{124}}.
\end{equation}

In the special case of $\sX=\sY=\C_0(\D(\TT))$ note that the action of $\D(\TT)$ on $\C_0(\D(\TT))$ is given by the comultiplication, so $\phi_{\C_0(\D(\TT))}=\id$ and using \eqref{eq-eq25}
\begin{align*}
\C_0(\D(\TT))&\!\!\;\boxtimes\C_0(\D(\TT))\\
&=\operatorname{\overline{span}}
\bigl\{
\ww^{\TT}_{23}(x\tens\I\tens\I){\ww^{\TT}_{23}}^*(\I\tens\I\tens y)
\,\bigr|\bigl.\,
x,y\in\C(\TT)\tens\mathrm{c}_0(\ZZ)
\bigr\}\\
&=\C(\TT)\tens
\operatorname{\overline{span}}
\bigl\{
\ww^{\TT}(x\tens\I){\ww^\TT}^*(\I\tens y)
\,\bigr|\bigl.\,
x\in\mathrm{c}_0(\ZZ),\:y\in\C(\TT)
\bigr\}
\tens\mathrm{c}_0(\ZZ).
\end{align*}

Now let $\sX$ and $\sY$ be $\TT$-\cst-algebras. We will describe how to turn them into $\D(\TT)$-\cst-algebras in a canonical way, so that the general braided tensor product $\boxtimes$ can be identified with $\boxtimes_\zeta$ (see Remark \ref{rem1}). Let us fix $\zeta\in\TT$. Recall from Section \ref{sect-boson} that $\bn$ denotes the canonical generator of $\mathrm{c}_0(\ZZ)$.

\begin{proposition}\label{prop-lemma1}
There is a unique action $\tilde{\rho}^\sX$ on of $\ZZ$ on $\sX$ such that $\tilde{\rho}^{\sX}(x)=\zeta^{\deg(x)\bn}\tens x$ for all homogeneous $x\in\sX$. The actions $\rho$ and $\tilde{\rho}$ of $\TT$ and $\ZZ$ on $\sX$ commute.
\end{proposition}

\begin{proof}
Homogeneous elements are dense, hence if the claimed action of $\ZZ$ exists, it is unique. We let $\tilde{\rho}^\sX$ be the map taking $x\in\sX$ to the element
\[
\ZZ\ni n\longmapsto(\operatorname{ev}_{\zeta^n}\tens\id)\rho^\sX(x)\in\sX
\]
of $\M(\mathrm{c}_0(\ZZ)\tens\sX)$.
This is a well defined non-degenerate $*$-homomorphism and it defines an action of $\ZZ$ on $\sX$. Clearly actions of $\ZZ$ and $\TT$ commute and $\tilde{\rho}^\sX(x)=\zeta^{\deg(x)\bn}\tens x$ for homogeneous $x$.
\end{proof}

The action $\brho$ of $\D(\TT)$ introduced via Proposition \ref{prop-lemma1} has a very simple expression on homogeneous elements:
\[
\brho^\sX(x)
=\bz^{\deg(x)}\tens\zeta^{\deg(x)\bn}\tens x\in\M\bigl(\C(\TT)\tens\mathrm{c}_0(\ZZ)\tens\sX\bigr),
\]
where, as defined in Section \ref{sect-SUq2}, $\bz\in\C(\TT)$ is the identity function.

For further discussion we will need the isomorphism $\C(\TT^2_\zeta)\simeq\C(\TT)\boxtimes\C(\TT)$ established it in Proposition \ref{prop-lemma2} below. Here the action of $\TT$ on both copies of $\C(\TT)$ is given by the comultiplication, and the $\ZZ$-action is the one from Proposition \ref{prop-lemma1}. Then, by the description above
\begin{equation}\label{eq10}
\resizebox{0.9\textwidth}{!}{\ensuremath{\displaystyle
\C(\TT)\!\!\:\boxtimes\C(\TT)=\operatorname{\overline{span}}
\bigl\{
\bigl((\phi_{\C(\TT)}\tens\phi_{\C(\TT)})\ww^\TT_{23}\bigr)(x\tens\I)
\bigl((\phi_{\C(\TT)}\tens\phi_{\C(\TT)}){\ww^\TT_{23}}\bigr)^*(\I\tens y)
\,\bigr|\bigl.\,
x,y\in\C(\TT)
\bigr\},
}}
\end{equation}
where the map $\phi_{\C(\TT)}\colon\C_0^\uu(\hh{\D(\TT)})=\mathrm{c}_0(\ZZ)\tens\C(\TT)\to\B(\Ltwo(\TT))$ is associated with the action of $\D(\TT)$ on $\C(\TT)$. More precisely
\[
\brho^{\C(\TT)}(x)=\bigl((\id\tens\phi_{\C(\TT)})\wW^{\D(\TT)}\bigr)^*(\I\tens x)
\bigl((\id\tens\phi_{\C(\TT)})\wW^{\D(\TT)}\bigr),\qquad x\in\C(\TT).
\]
It can be shown by direct calculations that $\phi_{\C(\TT)}$ is the unique $*$-homomorphism taking $f\tens\bz$ to $f\zeta^{-\bn}$ for any $f\in\mathrm{c}_0(\ZZ)$. Denoting the morphism from $\C(\TT)$ to $\mathrm{c}_0(\ZZ)$ taking $\bz$ to $\zeta^{-\bn}$ by $\theta$ we can write $\phi_{\C(\TT)}(f\tens\bz)$ as $f\theta(\bz)$ and using this we can make \eqref{eq10} more explicit:
\begin{equation}\label{eq-eq12}
\C(\TT)\boxtimes\C(\TT)
=\operatorname{\overline{span}}
\bigl\{
\bigl((\theta\tens\id)\ww^{\TT}\bigr)
(x\tens\I)
\bigl((\theta\tens\id)\ww^{\TT}\bigr)^*
(\I\tens y)
\,\bigr|\bigl.\,
x,y\in\C(\TT)
\bigr\}.
\end{equation}

\begin{proposition}\label{prop-lemma2}
There is an isomorphism of \cst-algebras $\C(\TT)\boxtimes\C(\TT)\simeq\C(\TT^2_\zeta)$ which maps the canonical generators of copies of $\C(\TT)$ inside $\C(\TT)\boxtimes\C(\TT)$ to the standard generators of $\C(\TT^2_\zeta)$. In particular $\iota_1(\bz)\iota_2(\bz)=\zeta\,\iota_2(\bz)\iota_1(\bz)$ in $\C(\TT)\boxtimes\C(\TT)$.
\end{proposition}

\begin{proof}
First let us check that the generating unitaries in $\C(\TT)\boxtimes\C(\TT)$ satisfy the commutation relations of standard generators of $\C(\TT^2_\zeta)$: using the expression\footnote{It is easier to see that $\ww^\ZZ=(\I\tens\bz)^{(\bn\tens\I)}$, so $\ww^\TT=\Sigma(\ww^\ZZ)^*\Sigma=(\bz\tens\I)^{-(\I\tens\bn)}$, see \cite[Introduction]{WoronowiczE2}.} $\ww^\TT=(\bz\tens\I)^{-(\I\tens\bn)}$ we calculate
\begin{equation}\label{eq-eq18}
\iota_1(\bz)=\bigl((\theta\tens\id)\ww^{\TT}\bigr)(\bz\tens\I)\bigl((\theta\tens\id)\ww^{\TT}\bigr)^*=\zeta^{\bn\tens\bn}(\bz\tens\I)\zeta^{-\bn\tens\bn}=\bz\tens\zeta^{\bn}
\end{equation}
and hence
\[
\iota_1(\bz)\iota_2(\bz)=(\bz\tens\zeta^{\bn})(\I\tens\bz)
=\zeta(\I\tens\bz)(\bz\tens\zeta^{\bn})
=\zeta\,\iota_2(\bz)\iota_1(\bz).
\]
Consequently $\C(\TT)\boxtimes\C(\TT)$ is a quotient of $\C(\TT^2_\zeta)$ via the map sending $U$ to $\iota_1(\bz)$ and $V$ to $\iota_2(\bz)$.

On the other hand one easily sees that $U\mapsto\iota_1(\bz)$ and $V\mapsto\iota_2(\bz)$ coincides with the regular representation (\cite[Section 2.2]{Williams}) of the dynamical system defining the rotation algebra $\C(\TT^2_\zeta)$ which is well known to be faithful.
\end{proof}

\begin{theorem}\label{thm2}
Let $\sX$ and $\sY$ be $\TT$-\cst-algebras and equip them with an action of $\ZZ$ (and hence of $\D(\TT)$) as described above in Proposition \ref{prop-lemma1}. Then the \cst-algebras $\sX\boxtimes_\zeta\sY$ and $\sX\boxtimes\sY$ are isomorphic. Furthermore, this isomorphism can be chosen to be $\TT$-equivariant, and to map $\jmath_\sX(x)\jmath_\sY(y)\in\sX\boxtimes_\zeta\sY$ to $\iota_\sX(x)\iota_\sY(y)\in\sX\boxtimes\sY$ for all $x\in\sX$, $y\in\sY$.
\end{theorem}

\begin{remark}\label{rem1}
 Once the parameter $\zeta$ is fixed, via Proposition \ref{prop-lemma1}, the category of $\mathcal{C}^*_{\TT}$ of \cst-algebras equipped with $\TT$-action becomes a full subcategory of $\mathcal{C}^*_{\D(\TT)}$, the category of \cst-algebras equipped with an action of $\D(\TT)$. Then Theorem \ref{thm2} tells us that on this subcategory constructions $\boxtimes$ and $\boxtimes_{\zeta}$ can be identified.

\end{remark}

\begin{proof}
We first note that an isomorphism $\sX\boxtimes_\zeta\sY\to\sX\boxtimes\sY$ mapping $\jmath_\sX(x)\jmath_\sY(y)\in\sX\boxtimes_\zeta\sY$ to $\iota_\sX(x)\iota_\sY(y)\in\sX\boxtimes\sY$ for all $x\in\sX$, $y\in\sY$ is automatically $\TT$-equivariant.

We have
\begin{align*}
\sX\boxtimes_\zeta\sY
&=\operatorname{\overline{span}}
\bigl\{\jmath_{\sX}(x)\jmath_{\sY}(y)\,\bigr|\bigl.\,x\in\sX,\:y\in\sY\bigr\}\\
&\simeq\operatorname{\overline{span}}
\bigl\{
\bigl((\iota_1\tens\id)\rho^\sX(x)\bigr)_{123}
\bigl((\iota_2\tens\id)\rho^\sY(y)\bigr)_{124}
\,\bigr|\bigl.\,
x\in\sX,\:y\in\sY\bigr\}\\
&=\operatorname{\overline{span}}
\bigl\{
\bigl((\theta\tens\id)\ww^{\TT}\bigr)_{12}
\rho^\sX(x)_{13}
\bigl((\theta\tens\id)\ww^{\TT}\bigr)^*_{12}
\rho^\sY(y)_{24}
\,\bigr|\bigl.\,
x\in\sX,\:y\in\sY\bigr\}\\
&\subseteq\M\bigl((\C(\TT)\boxtimes\C(\TT))\tens\sX\tens\sY\bigr)
\subseteq
\M\bigl(\cK(\Ltwo(\TT))\tens\cK(\Ltwo(\TT))\tens\sX\tens\sY\bigr)
\end{align*}
where the above isomorphism follows from Proposition \ref{prop-lemma2} and the next equality from \eqref{eq-eq12}. Let us denote this isomorphism
\[
\sX\boxtimes_\zeta\sY\longrightarrow\operatorname{\overline{span}}
\bigl\{
((\iota_1\tens\id)\rho^\sX(x))_{123}
((\iota_2\tens\id)\rho^\sY(y))_{124}
\,\bigr|\bigl.\,
x\in\sX,\:y\in\sY\bigr\}
\]
by $\Upsilon_1$.

In what follows $x$ and $y$ will be homogeneous elements of $\sX$ and $\sY$ respectively. We will make a more precise calculation of $\Upsilon_1(\jmath_{\sX}(x)\jmath_{\sY}(y))=((\iota_1\tens\id)\rho^\sX(x))_{123}
((\iota_2\tens\id)\rho^\sY(y))_{124}$ and then apply an injective morphism $\Upsilon_2$ to the \cst-algebra
\[
\bigl(\C_0(\D(\TT))\boxtimes\C_0(\D(\TT))\bigr)\tens\sX\tens\sY
\subseteq
\M\bigl(
\C(\TT)\tens\mathrm{c}_0(\ZZ)\tens\cK(\Ltwo(\TT))\tens\mathrm{c}_0(\ZZ)\tens\sX\tens\sY\bigr).
\]
Next we will show that the image of $\Upsilon_1(\jmath_{\sX}(x)\jmath_{\sY}(y))$ coincides there with the image of $(\brho^\sX\boxtimes\brho^\sY)(\iota_\sX(x)\iota_\sY(y))$ which by \eqref{eq-eq8} means that $\sX\boxtimes_\zeta\sY$ and $\sX\boxtimes\sY$ are isomorphic. Since the isomorphism maps homogeneous elements to homogeneous elements (of the same degree) it is equivariant. In the following calculations we will, among others, use the morphism $\theta$ introduced before Proposition \ref{prop-lemma2}.

We have
\begin{align*}
\Upsilon_1\bigl(\jmath_{\sX}(x)\jmath_{\sY}(y)\bigr)
&=\bigl((\theta\tens\id)\ww^\TT\bigr)_{12}\rho^\sX(x)_{13}
\bigl((\theta\tens\id)\ww^\TT\bigr)^*_{12}\rho^\sY(y)_{24}\\
&=\zeta^{\bn\tens\bn\tens\I\tens\I}(\bz^{\deg(x)}\tens\I\tens x\tens\I)\zeta^{-\bn\tens\bn\tens\I\tens\I}(\I\tens\bz^{\deg(y)}\tens\I\tens y)\\
&=\bigl(\zeta^{\bn\tens\bn}(\bz^{\deg(x)}\tens\I)\zeta^{-\bn\tens\bn}(\I\tens\bz^{\deg(y)})\bigr)\tens x\tens y\\
&=\bz^{\deg(x)}\tens\zeta^{\deg(x)\bn}\bz^{\deg(y)}\tens{x}\tens{y}
\end{align*}
(cf.~\eqref{eq-eq18}).

Now let $\tilde{\theta}$ be the morphism from $\C(\TT)$ to $\mathrm{c}_0(\ZZ)$ mapping $\bz$ to $\zeta^\bn$. We introduce two more morphisms $\Theta_1\in\Mor(\C(\TT),\C(\TT)\tens\mathrm{c}_0(\ZZ))$ and $\Theta_2\in\Mor(\cK(\Ltwo(\TT)),\cK(\Ltwo(\TT))\tens\mathrm{c}_0(\ZZ))$ by $\Theta_1=(\id\tens\tilde{\theta})\comp\Delta_\TT$ and
\[
\Theta_2(b)=(\id\tens\tilde{\theta})\bigl(\vv^\TT(b\tens\I){\vv^\TT}^*\bigr),\qquad{b}\in\cK(\Ltwo(\TT)),
\]
where $\vv^{\TT}$ is the right regular representation of $\TT$ implementing the comultiplication via $\Delta_{\TT}(f)=\vv^{\TT}(f\tens\I){\vv^{\TT}}^*$ for $f\in \C(\TT)$. Finally let $\Upsilon_2$ be the injective map $\Upsilon_2=(\Theta_1\tens\Theta_2\tens\id\tens\id)$.

We have $\Delta_{\TT}(\bz)=\bz\tens\bz$, hence
\[
\resizebox{\textwidth}{!}{\ensuremath{\displaystyle
\begin{aligned}
\Theta_2\bigl(\zeta^{\deg(x)\bn}\bz^{\deg(y)}\bigr)
&=(\id\tens\tilde{\theta})\bigl((\zeta^{\deg(x)\bn}\tens\I)\vv^\TT(\bz^{\deg(y)}\tens\I){\vv^\TT}^*\bigr)\\
&=(\zeta^{\deg(x)\bn}\tens\I)\bigl(
(\id\tens\tilde{\theta})\bigl(\bz^{\deg(y)}\tens\bz^{\deg(y)}\bigr)\bigr)
=\zeta^{\deg(x)\bn}\bz^{\deg(y)}\tens\zeta^{\deg(y)\bn}.
\end{aligned}
}}
\]
It follows that
\begin{equation}\label{eq-exactly}
\Upsilon_2\bigl(\Upsilon_1(\jmath_{\sX}(x)\jmath_{\sY}(y))\bigr)
=\bz^{\deg(x)}\tens\zeta^{\deg(x)\bn}\tens\zeta^{\deg(x)\bn}\bz^{\deg(y)}\tens\zeta^{\deg(y)\bn}\tens x\tens y.
\end{equation}

Now according to \eqref{eq-eq25} (cf.~also the first part of \eqref{eq-dodatkowe})
\[
\resizebox{\textwidth}{!}{\ensuremath{\displaystyle
\begin{aligned}
 \bigl(\brho^\sX\boxtimes\brho^\sY)&\bigl(\iota_\sX(x)\iota_\sY(y)\bigr)
\simeq \ww^\TT_{23}\brho^\sX(x)_{125}\bigl(\ww^\TT_{23}\bigr)^*\brho^\sY(y)_{346}\\
&=(\I\tens\bz\tens\I\tens\I\tens\I\tens\I)^{-\I\tens\I\tens\bn\tens\I\tens\I}
(\bz^{\deg(x)}\tens\zeta^{\deg(x)\bn}\tens\I\tens\I\tens x\tens\I)\\
&\qquad\quad(\I\tens\bz\tens\I\tens\I\tens\I\tens\I)^{\I\tens\I\tens\bn\tens\I\tens\I}
(\I\tens\I\tens\bz^{\deg(y)}\tens\zeta^{\deg(y)\bn}\tens\I\tens y)\\
&=\bz^{\deg(x)}\tens
\bigl(
(\bz\tens\I)^{-\I\tens\bn}(\zeta^{\deg(x)\bn}\tens\I)(z\tens\I)^{\I\tens\bn}(\I\tens\bz^{\deg(y)})
\bigr)
\tens\zeta^{\deg(y)\bn}\tens x\tens y\\
&=\bz^{\deg(x)}\tens\zeta^{\deg(x)\bn}\tens\zeta^{\deg(x)\bn}\bz^{\deg(y)}\tens\zeta^{\deg(y)\bn}\tens x\tens y
\end{aligned}
}}
\]
which is exactly the right-hand side of \eqref{eq-exactly}.
\end{proof}

Having established equivalence of the two definitions of the braided tensor product used in the description of the braided quantum group $\SU_q(2)$ we can freely use the results about tensoring various types of maps such as those contained in \cite[Section 5.1]{MeyerRoyWoronowiczI} as well as \cite[Proposition 5.1]{DeCommerKrajczok} (also on von Neumann algebra level). In particular we have

\begin{corollary}\label{cor-idbth}
There exist equivariant u.c.p.~maps $\bh\boxtimes\id,\id\boxtimes\bh\colon A\boxtimes A\to A$ such that
\[
(\bh\boxtimes\id)\bigl(\iota_1(a)\iota_2(b)\bigr)=\bh(a)b\quad\text{and}\quad
(\id\boxtimes\bh)\bigl(\iota_1(a)\iota_2(b)\bigr)=\bh(b)a
\]
for all $a,b\in A$.
\end{corollary}

At the end of this section let us quote the result from \cite{MeyerRoyWoronowiczII} which comes useful e.g.~in the proof of Proposition \ref{prop-hLeftInv} or in Section \ref{sect-BosonBoson} below. The version of this result given here differs from the original because of the left/right switch in notational conventions.

\begin{proposition}[{\cite[Proposition 6.3]{MeyerRoyWoronowiczII}}]\label{prop-prop2Cstar}
Let $\sA$ and $\sB$ be $\D(\GG)$-\cst-algebras. Then there exists a unique injective $*$-homomorphism
\[
\Psi\colon\sA\boxtimes\sB\boxtimes\C_0(\GG)\longrightarrow
\bigl(\sA\boxtimes\C_0(\GG)\bigr)\tens
\bigl(\sB\boxtimes\C_0(\GG)\bigr)
\]
such that
\begin{align*}
\\
\Psi\bigl(\iota_{\sA}(a)\bigr)
&=\iota_{\sA}(a)\tens\I,\\
\Psi\bigl(\iota_{\sB}(b)\bigr)
&=(\iota_{\C_0(\GG)}\tens\iota_{\sB})\rho^{\sB}(b),\\
\Psi\bigl(\iota_{\C_0(\GG)}(c)\bigr)
&=(\iota_{\C_0(\GG)}\tens\iota_{\C_0(\GG)})\Delta_{\GG}(c)
\end{align*}
for all $c\in\C_0(\GG)$, $a\in\sA$, and $b\in\sB$.
\end{proposition}

For completeness we note the corresponding von Neumann algebraic version of this proposition which can be proved in the von Neumann algebra framework of \cite{DeCommerKrajczok}:

\begin{proposition}
Let $\sM,\sN$ be $\D(\GG)$-von Neumann algebras. There is a unique normal injective $*$-homomorphism
\[
\Psi\colon\sM\,\overline{\boxtimes}\,\sN\,\overline{\boxtimes}\,\Linf(\GG)\to
\bigl(\sM\,\overline{\boxtimes}\,\Linf(\GG)\bigr)\bar{\tens}
\bigl(\sN\,\overline{\boxtimes}\,\Linf(\GG)\bigr)
\]
satisfying
\begin{align*}
\Psi\bigl(\iota_{\sM}(m)\bigr)
&=\iota_{\sM}(m)\tens\I,\\
\Psi\bigl(\iota_{\sN}(n)\bigr)
&=(\iota_{\Linf(\GG)}\tens\iota_{\sN})\rho^{\sN}(n),\\
\Psi\bigl(\iota_{\Linf(\GG)}(x)\bigr)
&=(\iota_{\Linf(\GG)}\tens\iota_{\Linf(\GG)})\Delta_{\GG}(x)
\end{align*}
for all $x\in\Linf(\GG)$, $m\in\sM$, and $n\in\sN$.
\end{proposition}

\subsection{Two approaches to the bosonization of \texorpdfstring{$\SU_q(2)$}{SUq(2)}}\label{sect-BosonBoson}

On the purely algebraic level bosonization of a braided Hopf algebra was introduced by S.~Majid (see \cite[Chapter 16]{QGPrimer} and also \cite[Section 3.8]{HeckenbergerSchneider}). We will compare two ways this procedure is viewed in the \cst-algebraic context for the braided quantum group $\SU_q(2)$. The first of these procedures was already described in Section \ref{sect-boson}. The second one comes e.g.~from \cite[Theorem 6.4]{MeyerRoyWoronowiczII} (see also \cite{RoyBraided}).

Translated to our setting (taking into account the left/right conventions) the \cst-algebra describing the bosonization is $\widetilde{B}=A\boxtimes\C(\TT)$, where the structure of a $\D(\TT)$-\cst-algebra on $A$ is provided by the action \eqref{eq-eq5} of $\TT$ considered from the beginning and the $\ZZ$ action introduced in Proposition \ref{prop-lemma1}, while the action of $\TT$ on $\C(\TT)$ is given by the comultiplication and $\ZZ$ acts on $\C(\TT)$ by the adjoint action $\tilde{\rho}^{\C(\TT)}(b)={\ww^\ZZ}^* (\I\tens b)\ww^\ZZ$ ($b\in\C(\TT)$). Explicitly
\[
\rho^{\C(\TT)}(\bz)=\bz\tens\bz,\quad
\tilde{\rho}^{\C(\TT)}(\bz)={\ww^{\ZZ}}^*(\I\tens\bz)\ww^{\ZZ}=\I\tens\bz.
\]
In particular, the action of $\ZZ$ is trivial. This combines to the action of $\D(\TT)=\TT\times\ZZ$ given by $\brho^{\C(\TT)}(\bz)=\bz\tens\I\tens\bz$.

\begin{proposition}\label{prop-BB}
We have $\widetilde{B}\simeq B$.
\end{proposition}

\begin{proof}
By construction, using $\hh{R}^\uu=\ww^\TT_{23}$ we have
\[
\resizebox{\textwidth}{!}{\ensuremath{\displaystyle
\begin{aligned}
A\boxtimes\C(\TT)
&=\operatorname{\overline{span}}
\bigl\{
(\bphi_A\tens\bphi_{\C(\TT)})(\ww^{\TT}_{23})(a\tens\I)
(\bphi_A\tens\bphi_{\C(\TT)})(\ww^{\TT}_{23})^*(\I\tens b)
\,\bigr|\bigl.\,
a\in A,\:b\in\C(\TT)\bigr\}\\
&=\operatorname{\overline{span}}
\bigl\{
(\tilde{\phi}_A\tens\phi_{\C(\TT)})(\ww^{\TT})(a\tens\I)
(\tilde{\phi}_A\tens\phi_{\C(\TT)})(\ww^{\TT})^*(\I\tens b)
\,\bigr|\bigl.\,
a\in A,\:b\in\C(\TT)\bigr\}\\
&=\operatorname{\overline{span}}
\bigl\{
(\tilde{\phi}_A\tens\id)(\ww^{\TT})(a\tens\I)
(\tilde{\phi}_A\tens\id)(\ww^{\TT})^*(\I\tens b)
\,\bigr|\bigl.\,
a\in A,\:b\in\C(\TT)\bigr\},
\end{aligned}
}}
\]
hence applying flip gives
\begin{align*}
\widetilde{B}
&=A\boxtimes\C(\TT)
\simeq\operatorname{\overline{span}}
\bigl\{
(\id\tens\tilde{\phi}_A)(\ww^{\ZZ})^*(\I\tens a)
(\id\tens\tilde{\phi}_A)(\ww^{\ZZ})(b\tens\I)
\,\bigr|\bigl.\,a\in A,\:b\in\C(\TT)\bigr\}\\
&=\ZZ\ltimes_{\tilde{\rho}}A\simeq B
\end{align*}
(see the proof of Proposition \ref{prop-injkap}). Note that this isomorphism maps generators to generators. Here $\ltimes_{\tilde{\rho}}$
is the crossed product by the action $\tilde{\rho}$ of $\ZZ$ on $A$ given by the restriction of action $\brho$ of $\D(\TT)$ on $A$ to $\ZZ$, i.e.~the one given by Proposition \ref{prop-lemma1}. Concretely $\tilde{\rho}(\alpha)=\I\tens\alpha$ and $\tilde{\rho}(\gamma)=\zeta^{\bn}\tens\gamma$. This action coincides with $\tilde{\rho}^A$ introduced in the proof of Proposition \ref{prop-injkap} and thus we conclude $\tilde{B}\simeq B$ and both versions of bosonization agree as \cst-algebras.
\end{proof}

Now we will use the map $\Psi$ described in Proposition \ref{prop-prop2Cstar} to define $\Delta_{\widetilde{B}}$ as the composition
\[
\xymatrix{
\widetilde{B}=A\boxtimes\C(\TT)\ar[rr]^{\Delta_A\boxtimes\id}&&A\boxtimes A\boxtimes\C(\TT)\ar[r]^(.32)\Psi
&\bigl(A\boxtimes\C(\TT)\bigr)\tens\bigl(A\boxtimes\C(\TT)\bigr)=\widetilde{B}\tens\widetilde{B},
}
\]
i.e.~$\Delta_{\widetilde{B}}=\Psi\comp(\Delta_A\boxtimes\id)$ (see \cite[Theorem 6.4]{MeyerRoyWoronowiczII}, cf.~also the proof of Proposition \ref{prop-hLeftInv}).

\begin{proposition}\label{prop_BBDel}
The isomorphism in Proposition \ref{prop-BB} respects comultiplications, hence we have $(B,\Delta_B)\simeq(\widetilde{B},\Delta_{\widetilde{B}})$ as compact quantum groups.
\end{proposition}

\begin{proof}
We will use a combination of the two versions of the notation for embeddings of \cst-algebras into braided tensor products as described in \eqref{eq-notconv}: we will write $\iota_A$ and $\iota_{\C(\TT)}$ for the canonical embeddings of the factors into $A\boxtimes\C(\TT)$ and $\iota_1,\iota_2$ and $\iota_3$ for the embeddings of the factors into $A\boxtimes A\boxtimes\C(\TT)$. Finally we will denote by $\iota_{12}$ the embedding of $A\boxtimes A$ into $A\boxtimes A\boxtimes\C(\TT)$ and $\jmath_1$ and $\jmath_2$ will be the embeddings of $A$ onto the factors of the braided tensor product $A\boxtimes A$.

We have
\[
\Delta_{\widetilde{B}}\bigl(\iota_{\C(\TT)}(\bz))
=\Psi\bigl(\iota_3(\bz)\bigr)
=(\iota_{\C(\TT)}\tens\iota_{\C(\TT)})\bigl(\Delta_{\TT}(\bz)\bigr)
=\iota_{\C(\TT)}(\bz)\tens\iota_{\C(\TT)}(\bz)
\]
and
\[
\begin{aligned}
\Delta_{\widetilde{B}}\bigl(\iota_A(\alpha)\bigr)
&=(\Psi\comp\iota_{12})\bigl(\jmath_1(\alpha)\jmath_2(\alpha)-q\jmath_1(\gamma)^*\jmath_2(\gamma)\bigr)
=\Psi\bigl(\iota_1(\alpha)\iota_2(\alpha)-q\iota_1(\gamma)^*\iota_2(\gamma)\bigr)\\&=\bigl(\iota_A(\alpha)\tens\I\bigr)\bigl((\iota_{\C(\TT)}\tens\iota_A)\rho^A(\alpha)\bigr)
-q\bigl(\iota_A(\gamma)^*\tens\I\bigr)\bigl((\iota_{\C(\TT)}\tens\iota_A)\rho^A(\gamma)\bigr)\\
&=\iota_A(\alpha)\tens\iota_A(\alpha)-q\iota_A(\gamma^*)\iota_{\C(\TT)}(\bz)\tens\iota_A(\gamma),\\
\Delta_{\widetilde{B}}\bigl(\iota_A(\gamma)\bigr)
&=(\Psi\comp\iota_{12})\bigl(\jmath_1(\gamma)\jmath_2(\alpha)+\jmath_1(\alpha)^*\jmath_2(\gamma)\bigr)
=\Psi\bigl(\iota_1(\gamma)\iota_2(\alpha)+\iota_1(\alpha)^*\iota_2(\gamma)\bigr)\\
&=\bigl(\iota_A(\gamma)\tens\I\bigr)\bigl((\iota_{\C(\TT)}\tens\iota_A)\rho^A(\alpha)\bigr)
+\bigl(\iota_A(\alpha^*)\tens\I\bigr)
\bigl((\iota_{\C(\TT)}\tens\iota_A)\rho^A(\gamma)\bigr)\\
&=\iota_A(\gamma)\tens\iota_A(\alpha)+\iota_A(\alpha^*)\iota_{\C(\TT)}(\bz)\tens\iota_A(\gamma).
\end{aligned}
\]
Comparing these formulas with \eqref{eq-eq32} and noting that $\iota_{\C(\TT)}(\bz)$ is the generator $z$ of $B$ (via the isomorphism of Proposition \ref{prop-BB}) we see that $(B,\Delta_B)\simeq (\widetilde{B},\Delta_{\widetilde{B}})$ as compact quantum groups.
\end{proof}

\subsection{Properties of the braided flip}\label{app-BraidedFlip}

It this section we will prove Proposition \ref{prop1} thus showing that the braided flip $\chi^\boxtimes$ extends to a completely bounded map on $A\boxtimes A$ precisely when $\zeta=q/\overline{q}\in\TT$ is a root of unity. Our reasoning uses the notion of completely bounded multipliers which we will now briefly recall. For the proofs of the facts listed below and additional information see \cite{HuNeufangRuan,JungeNeufangRuan,CrannInner,DKV24}. We follow the notation of \cite{DKV24}.

Let $\GG$ be a locally compact quantum group. The \emph{Fourier algebra} of $\GG$ is defined as $\A(\GG)=\bigl\{(\omega\tens\id)\ww^{\hh{\GG}}\,\bigr|\bigl.\omega\in \Linf(\hh{\GG})_*\bigr\}\subseteq\C_0(\GG)$ (see e.g.~\cite[Section 1]{DKSS}). When equipped with the multiplication of $\C_0(\GG)$ and the operator space structure of $\Linf(\hh{\GG})_*$, it becomes a completely contractive Banach algebra. Next, one defines the space $\M^l_{\cb}(\A(\GG))$ of \emph{(left) completely bounded multipliers}, as the set of elements $b\in \Linf(\GG)$ satisfying the following condition: $ba\in\A(\GG)$ whenever $a\in\A(\GG)$ and the resulting map $\A(\GG)\to\A(\GG)$ is completely bounded. Since $\A(\GG)\simeq\Linf(\hh{\GG})_*$, the dual of the map $a\mapsto{ba}$ is a normal c.b.~map $\Linf(\hh{\GG})\to\Linf(\hh{\GG})$ which we denote by $\Theta^l(b)$.

The space $\M^l_{\cb}(\A(\GG))$ equipped with the norm $\|b\|_{\cb}=\|\Theta^l(b)\|_{\cb}$ and the multiplication of $\Linf(\GG)$ becomes a completely contractive Banach algebra and we have $\A(\GG)\subseteq \M^l_{\cb}(\A(\GG))\subseteq \M(\C_0(\GG))$. We note that for each $b\in\M^l_{\cb}(\A(\GG))$ the map $\Theta^l(b)$ restricts to $\C_0(\hh{\GG})$ and we have $\|b\|_{\cb}=\bigl\|\bigl.\Theta^l(b)\bigr|_{\C_0(\hh{\GG})}\bigr\|_{\cb}$.

A particularly nice family of completely bounded multipliers forms the \emph{Fourier-Stieltjes algebra} of $\GG$, defined as $\mathcal{B}(\GG)=\bigl\{(\omega\tens\id)\Ww^{\hh{\GG}}\,\bigr|\bigl.\,\omega\in\C_0^\uu(\hh{\GG})^*\bigr\}\subseteq\M(\C_0(\GG))$ (also c.f.~\cite[Section 1]{DKSS}). It is a completely contractive Banach algebra, when equipped with the multiplication of $\M(\C_0(\GG))$ and the operator space structure of $\C_0^\uu(\hh{\GG})^*$. Furthermore, $\mathcal{B}(\GG)\subseteq\M^l_{\cb}(\A(\GG))$ and we have a concrete formula for the associated c.b.~maps: if $b=(\omega\tens\id)(\Ww^{\hh{\GG}})\in\mathcal{B}(\GG)$, then $\bigl.\Theta^l(b)\bigr|_{\C_0(\hh{\GG})}=(\omega\tens\id)\comp\Delta^\ur_{\hh{\GG}}$, where
\[
\Delta_{\hh{\GG}}^{\ur}\in\Mor\bigl(\C_0(\hh{\GG}),\C_0^\uu(\hh{\GG})\tens\C_0(\hh{\GG})\bigr)
\]
is the half-lifted version of the comultiplication (\cite[Section 2]{DawsKrajczokVoigt}). In general the inclusion $\mathcal{B}(\GG)\subseteq\M^l_{\cb}(\A(\GG))$ is strict, but if $\GG$ is amenable, then $\mathcal{B}(\GG)=\M^l_{\cb}(\A(\GG))$ and this identification is completely isometric (see \cite[Theorem 7.2]{CrannInner} for the general and \cite{Bozejko} for the classical case).

To prove Proposition \ref{prop1} we start with a lemma concerning completely bounded multipliers on $\ZZ^2$.

\begin{lemma}\label{lemma2}
Fix $\zeta\in\TT$ and $z,z_1,z_2\in\CC$ with $|z|,|z_1|,|z_2|\le 1$. Let $g_z\colon\ZZ\to\CC$ be defined by
\[
g_z(n)=\begin{cases}
z^n&n\geq 0\\
\overline{z}^{-n}&n<0
\end{cases}.
\]
Furthermore let $g_{z_1,z_2}\colon\ZZ^2\ni(n,m)\mapsto g_{z_1}(n)g_{z_2}(m)\in\CC$ and $f_\zeta\colon\ZZ^2\ni(n,m)\mapsto\zeta^{nm}\in\CC$. Then
\begin{enumerate}
\item\label{lemma2-1} The function $g_z$ belongs to $\M^l_{\cb}(\A(\ZZ))$ and we have $\|g_z\|_{\cb}\le \tfrac{1+|z|}{1-|z|}$ for $|z|<1$ and $\|g_z\|_{\cb}=1$ when $|z|=1$,
\item\label{lemma2-2} the function $g_{z_1,z_2}$ belongs to $\M^l_{\cb}(\A(\ZZ^2))$ and $\|g_{z_1,z_2}\|_{\cb}=\|g_{z_1}\|_{\cb}\|g_{z_2}\|_{\cb}$,
\item\label{lemma2-3} if $\zeta\in\TT$ is a root of unity with $\zeta=\ee^{2\pi\ii n/d}$ with $\gcd(n,d)=1$, then $f_\zeta\in\M^l_{\cb}(\A(\ZZ^2))$ and $\|f_\zeta\|_{\cb}=d$,
\item\label{lemma2-4} if $\zeta\in\TT$ is not a root of unity, then $f_{\zeta}\not\in\M^l_{\cb}(\A(\ZZ^2))$.
\end{enumerate}
\end{lemma}

\begin{proof}
Since $\ZZ$ is amenable, we can work with the Fourier-Stieltjes algebra $\mathcal{B}(\ZZ)$ which is completely isometric to $\C(\TT)^*$, identified with the space of Borel, regular, complex measures on $\TT$. Unraveling these identifications, we see that a measure $\rho\in \C(\TT)^*$ viewed as an element of $\mathcal{B}(\ZZ)$ corresponds to the function $\ZZ\ni n\mapsto \rho(\bz^{-n})\in\CC$, where $\bz$ is the canonical coordinate on $\TT$. An analogous observation applies to $\ZZ^2$ and $\TT^2$.

Ad \eqref{lemma2-1}. If $|z|=1$, then $g_z(n)=z^n=\delta_{\overline{z}}(\bz^{-n})$ hence $\|g_z\|_{\cb}=\|\delta_z\|=1$. In the case $|z|<1$, consider the measure $\rho=\sum\limits_{n\in\ZZ} g_z(n)\bh_{\TT}(\bz^{n}\cdot)$ ($\bh_\TT$ is the normalized Haar measure on $\TT$). Since $\rho(\bz^{-n})=g_z(n)$ and $\|\rho\|\le\sum\limits_{n\in \ZZ}\bigl|g_z(n)\bigr|=\tfrac{1+|z|}{1-|z|}$, the claim follows.

Point \eqref{lemma2-2} follows from the fact that $\Theta^l(g_{z_1,z_2})=\Theta^l(g_{z_1})\tens\Theta^l(g_{z_2})$.

Ad \eqref{lemma2-3}. Assume that $\zeta=\ee^{2\pi\ii n/d}$, where $n\in\ZZ_+$, $d\in\NN$ and $\gcd(n,d)=1$. Define the measure $\rho=\tfrac{1}{d}\sum\limits_{k,l=0}^{d-1}\zeta^{-kl}\delta_{(\zeta^k,\zeta^l)}$ ($\delta_{(\zeta^k,\zeta^l)}$ is the Dirac measure at the point $(\zeta^k,\zeta^l)\in\TT^2$). Since
\[
\rho(\bz_1^{-n} \bz_2^{-m})=
\tfrac{1}{d}
\sum_{k,l=0}^{d-1} \zeta^{-kl} \zeta^{-kn} \zeta^{-ml}=
\sum_{\substack{0\leq k\leq d-1\\k\equiv-m\,\text{mod}\,d}}
\zeta^{-kn}
=\zeta^{nm}=f_\zeta(n,m)
\]
for $(n,m)\in\ZZ^2$, we see that $\rho\in\C(\TT^2)^*$ corresponds to $f_\zeta\in\mathcal{B}(\ZZ^2)$ and $\|f_\zeta\|_{\cb}=\|\rho\|$. The upper bound $\|\rho\|\le d$ is  trivial, for the converse inequality choose a function $u\in \C(\TT^2)$ with $\|u\|=1$ and $u(\zeta^k,\zeta^l)=\zeta^{kl}$ for $0\le k,l\le d-1$, then $\rho(u)=d$.

Ad \eqref{lemma2-4}.
Assume by contradiction that $f_\zeta\in\mathcal{B}(\ZZ^2)$. Then $f_\zeta$ is weakly almost periodic (\cite[Proposition D.8]{KerrLi}), i.e.~the set $\bigl\{f_\zeta((n,m)+\cdot)\,\bigr|\bigl.\,(n,m)\in\ZZ^2\bigr\}\subseteq\ell^{\infty}(\ZZ^2)$ has compact closure in the weak topology. Hence Grothendieck's double limit criterion \cite[Theorem D.4]{KerrLi} (applied to the Stone-\v{C}ech compactification $\beta\ZZ^2$ of $\ZZ^2$) implies that $(n_k)_{k\in\NN}$ and $(m_l)_{l\in\NN}$ are sequences of integers such that the limits
\begin{equation}\label{eq3}
\lim_{k\to\infty}\lim_{l\to\infty}f_\zeta(n_k,m_l),\quad 
\lim_{l\to\infty}\lim_{k\to\infty}f_\zeta(n_k,m_l)
\end{equation}
exist, then they are equal.

We will now construct sequences $(n_k)_{k\in\NN}$, $(m_l)_{l\in\NN}$ for which the limits in \eqref{eq3} exist, but are different. More precisely, we will construct sequences such that the following statement
\[
\left.\begin{array}{l}
n_R\in\NN\\[2pt]
m_R\in 2\ZZ_++1\\[2pt]
\forall\:s\in\{1,\dotsc,R-1\}\ \bigl|\zeta^{n_sm_R}-1\bigr|\leq\tfrac{1}{R}\\[2pt]
\forall\:t\in\{1,\dotsc,R\}\ \bigl|\zeta^{n_Rm_t}+1\bigr|\leq\tfrac{1}{R}
\end{array}\right\}
\tag{$\star_R$}
\]
is true for all $R\ge 1$.

For $R=1$, set $m_1=1$ and choose $n_1\in\NN$ such that $|\zeta^{n_1}+1|\le 1$. Next, assume that we have constructed $m_1,\dotsc,m_R$ and $n_1,\dotsc,n_R$ such that the conditions $(\star_1),\dotsc,(\star_R)$ hold. Choose $p\in\NN$ in such a way that $\bigl|\zeta^{2p}-\zeta^{-1}\bigr|\leq\tfrac{1}{(R+1)\max\{n_1,\dotsc,n_R\}}$ and define $m_{R+1}=2p+1\in 2\NN+1$. For $1\le s \le R$ we have
\[
\bigl|\zeta^{n_sm_{R+1}}-1\bigr|
=\bigl|(\zeta^{m_{R+1}})^{n_s}-1^{n_s}\bigr|
\leq n_s\bigl|\zeta^{m_{R+1}}-1\bigr|
=n_s\bigl|\zeta^{2p+1}-1\bigr|
\leq\tfrac{1}{R+1}.
\]
Next choose $n_{R+1}\in\NN$ so that $\bigl|\zeta^{n_{R+1}}+1\bigr|\leq\tfrac{1}{(R+1) \max\{m_1,\dotsc,m_{R+1}\}}$. Then for $1\leq t\leq R+1$
\[
\bigl|\zeta^{n_{R+1}m_t}+1\bigr|
=\bigl|(\zeta^{n_{R+1}})^{m_t}-(-1)^{m_t}\bigr|
\leq m_t\bigl|\zeta^{n_{R+1}}+1\bigr|\leq\tfrac{1}{R+1},
\]
because each $m_t$ is odd. This proves $(\star_{R+1})$ and ends the construction. Using conditions $(\star_R)$ we see that
\[
\lim_{k\to\infty}\lim_{l\to\infty}f_\zeta(n_k,m_l)
=\lim_{k\to\infty}\lim_{l\to\infty}\zeta^{n_km_l}=1
\]
and
\[
\lim_{l\to\infty}\lim_{k\to\infty}f_\zeta(n_k,m_l)
=\lim_{l\to\infty}\lim_{k\to\infty}\zeta^{n_km_l}=-1,
\]
hence we obtain a contradiction.
\end{proof}

We will need one more simple lemma, which is a consequence of the proof of \cite[Theorem 2.3]{BraidedSU2} and \cite[Equation (2.6)]{KrajczokSoltan}.

\begin{lemma}\label{lemma3}
For $q\in\CC$ with $0<|q|<1$ we have $\operatorname{Sp}(\gamma)=\{0\}\cup\bigl\{\lambda|q|^{k}\,\bigr|\bigl.\,\lambda\in\TT,\:k\in\ZZ_+\bigr\}$.
\end{lemma}

With these preliminaries, we can prove Proposition \ref{prop1}. 

\begin{proof}[Proof of Proposition \ref{prop1}]
Assume first that $\zeta$ is a root of unity and write $\zeta=\ee^{2\pi\ii n/d}$ for $n\in\ZZ_+$, $d\in\NN$ with $0\le n\le d-1$ and $\gcd(n,d)=1$. If $\zeta=1$ (i.e.~$q\in \RR$) the claim is trivial, hence assume that $d\ge 2$.
 
As explained in Section \ref{sect-SUq2} we can realize $\cA\algboxtimes\cA\subseteq A\boxtimes A\subseteq\C(\TT^2_\zeta)\tens A\tens A$. Recall (equation \eqref{eq1}) that $\chi^\boxtimes$ acts as
\[
\chi^\boxtimes\colon\iota_1(a)\iota_2(b)
=U^{\deg(a)}V^{\deg(b)}\tens a\tens b\longmapsto\iota_2(a)\iota_1(b)
=V^{\deg(a)}U^{\deg(b)}\tens b\tens a
\]
for homogeneous $a,b\in\cA$. Since $\zeta$ is a root of unity, we can use \cite[Corollary 1.12]{Boca} and identify $\C(\TT^2_\zeta)$ with $\C^*(\bz_1\tens U_0,\bz_2\tens V_0)\subseteq\C(\TT^2)\tens\Mat_d(\CC)$, where
\[
U_0=\begin{bmatrix}
0 & 1 & 0 & \cdots & 0\\
0 & 0 & 1 & \cdots & 0 \\
\vdots & \vdots & \vdots &
\ddots & \vdots\\
0 & 0 & 0 & \cdots & 1 \\
1 & 0 & 0 & \cdots & 0
\end{bmatrix},\quad
V_0=\begin{bmatrix}
1 & 0 & \cdots & 0\\
0 & \zeta & \cdots & 0 \\
\vdots & \vdots & \ddots & \vdots\\
0 & 0 & \cdots & \zeta^{d-1}
\end{bmatrix}.
\]
Let $\{e_k\}_{k=0,\dotsc,d-1}$ be the standard basis of $\CC^d$ (it will be convenient to treat $k$ as a cyclic variable in $\ZZ_d$). Define linear maps $\mathcal{O},\widetilde{\mathcal{O}}\colon\CC^d\to\CC^d$ by
\[
\mathcal{O}e_k=\tfrac{1}{\sqrt{d}}\sum_{l=0}^{d-1}\zeta^{k l}e_l,\qquad
\widetilde{\mathcal{O}}e_k=e_{-k}\in\CC^d
\]
(c.f.~\cite[Equation (1.11)]{Boca}). It is straightforward to check that $\mathcal{O}$ and $\widetilde{\mathcal{O}}$ are unitary and
\[
\mathcal{O} U_0 \mathcal{O}^*=V_0^{-1},\quad
\mathcal{O}V_0 \mathcal{O}^*=U_0,\quad
\widetilde{\mathcal{O}}U_0\widetilde{\mathcal{O}}^*=U_0^*,\quad
\widetilde{\mathcal{O}}V_0\widetilde{\mathcal{O}}^*=V_0^*.
\]
Next we define a linear map $\phi_0\colon\Mat_d(\CC)\to\Mat_d(\CC)$
\[
\phi_0(x)=\widetilde{\mathcal{O}}(\mathcal{O}x\mathcal{O}^*)^\top
\widetilde{\mathcal{O}}^*,
\]
where $\top$ denotes the transposition. Then $\phi_0$ is completely bounded with $\|\phi_0\|_{\cb}=d$ (\cite[Proposition 2.2.7]{EffrosRuan}). Furthermore we have
\begin{equation}\label{eq4}
\begin{aligned}
\phi_0(U_0^n V_0^m)
&=\widetilde{\mathcal{O}}
(\mathcal{O} U_0^n V_0^m \mathcal{O}^*)^t \widetilde{\mathcal{O}}^*\\
&=\widetilde{\mathcal{O}} (U_0^m)^t (V_0^{-n})^{t} \widetilde{\mathcal{O}}^*
=\widetilde{\mathcal{O}} U_0^{-m} V_0^{-n} \widetilde{\mathcal{O}}^*
=U_0^m V_0^n
\end{aligned}
\end{equation}
for any $n,m\in\ZZ$.

Next let $F\colon\C(\TT^2)\to\C(\TT^2)$ be the automorphism associated to the homeomorphism $\TT^2\ni(z_1,z_2)\mapsto(z_2,z_1)\in\TT^2$ and put $\widetilde{\phi}=F\tens\phi_0\colon\C(\TT^2)\tens\Mat_d(\CC)\to\C(\TT^2)\tens\Mat_d(\CC)$.

Since $F$ is a $*$-homomorphism, the map $\widetilde{\phi}$ is completely bounded with $\|\widetilde{\phi}\|_{\cb}=d$. Denoting the coordinate functions on $\TT^2$ by $\bz_1$ and $\bz_2$ and using \eqref{eq4} we obtain
\begin{align*}
\widetilde{\phi}(U^nV^m)
=\widetilde{\phi}(\bz_1^n\bz_2^m\tens U_0^nV_0^m)
&=\bz_2^n\bz_1^m\tens U_0^mV_0^n\\
&=\bz_1^m\bz_2^n\tens U_0^mV_0^n
=U^mV^n=\zeta^{nm}V^nU^m
\end{align*}
for all $n,m\in\ZZ$. By Lemma \ref{lemma2}\eqref{lemma2-3} the map $\Theta^l(f_{\overline{\zeta}})\colon\C(\TT^2)\to\C(\TT^2)$ acting by
\[
\bz_1^n\bz_2^m\longmapsto
f_{\overline{\zeta}}(-n,-m)\bz_1^n\bz_2^m=\zeta^{-nm}\bz_1^n\bz_2^m
\]
satisfies $\|\Theta^l(f_{\overline{\zeta}})\|_{\cb} =d$. It follows that $(\Theta^l(f_{\overline{\zeta}})\tens\id)\circ\widetilde\phi$ restricts to a completely bounded map $\phi\colon\C(\TT^2_{\zeta})\to\C(\TT^2_{\zeta})$ which satisfies
\[
\phi(U^nV^m)=V^nU^m,\qquad n,m\in\ZZ
\]
with $\|\phi\|_{\cb}\le d^2$. Finally observe that for homogeneous $a,b\in\cA$ we have
\begin{align*}
(\phi\tens\chi)\bigl(\iota_1(a)\iota_2(b)\bigr)
&=(\phi\tens\chi)\bigl(U^{\deg(a)}V^{\deg(b)}\tens a\tens b\bigr)\\
&=V^{\deg(a)}U^{\deg(b)}\tens b\tens a
=\iota_2(a)\iota_1(b)=\chi^{\boxtimes}\bigl(\iota_1(a)\iota_2(b)\bigr),
\end{align*}
where $\chi$ is the standard tensor flip on $A\tens A$. It follows that the restriction of $\phi\tens\chi$ to $A\boxtimes A$ defines an extension of $\chi^\boxtimes$ with c.b.~norm $\le d^2$. This ends the proof of the first part.

Assume now that $\zeta$ is not a root of unity, and assume by contradiction that $\chi^{\boxtimes}$ extends to a c.b.~map on $A\boxtimes A$ which we will denote by the same symbol. Then the map $(\chi^{\boxtimes})^2$ on $A\boxtimes A$ is completely bounded and acts by
\[
A\boxtimes A\ni\iota_1(a)\iota_2(b)\longmapsto
\zeta^{-2\deg(a)\deg(b)}\iota_1(a)\iota_2(b)\in A\boxtimes A
\]
for homogeneous $a,b\in\cA$. Lemma \ref{lemma2} tells us that $f_{\zeta^{-2}}\not\in \M^l_{\cb}(\A(\ZZ^2))$ and hence \cite[Proposition 7.2]{QuantumTori} implies that the map $ U^nV^m\mapsto f_{\zeta^{-2}}(n,m) U^nV^m$ does not extend to a linear, completely bounded map on $\C(\TT^2_\zeta)$. Thus for a fixed $N\ge 2$ we can find matrix $a=[a_{k,l}]_{k,l=1}^{p}\in\Mat_p(\C(\TT^2_\zeta))$ with $\|a\|_{\Mat_p(\C(\TT^2_\zeta))}=1$ with each matrix entry
\begin{equation}\label{eq5}
a_{k,l}=\sum_{n,m\in\ZZ} a_{k,l}^{n,m} U^n V^m
\end{equation}
(each $a_{k,l}$ is a finite linear combination of generators $U^nV^m$) and such that
\begin{equation}\label{eq6}
\left\|\biggl[\,
\sum_{n,m\in\ZZ}
 \zeta^{-2nm} a_{k,l}^{n,m} U^n V^m
\biggr]_{k,l=1}^{p}\right\|_{\Mat_p(\C(\TT^2_\zeta))}\ge N.
\end{equation}
Next, let us define
\[
b=[b_{k,l}]_{k,l=1}^{p}=
\biggl[\,
\sum_{n,m\in\ZZ}
a^{n,m}_{k,l} U^n V^m\tens\tilde{\gamma}^n\tens\tilde{\gamma}^m
\biggr]_{k,l=1}^{p}\in\Mat_p\bigl(\C(\TT^2_\zeta)\tens A\tens A\bigr),
\]
where we use the convenient notation $\tilde{\gamma}^n=\gamma^n$ if $n\ge 0$ and $\tilde{\gamma}^n=\gamma^{ * |n|}$ if $n<0$.

Consider now
\[
\mathsf{B}=\overline{\operatorname{span}}
\bigl\{\iota_1(\tilde{\gamma}^n)\iota_2(\tilde{\gamma}^m)\,\bigr|\bigl.\,n,m\in\ZZ\bigr\}
\]
which is a unital \cst-subalgebra in $A\boxtimes A$. We can write $b_{k,l}=\sum\limits_{n,m\in\ZZ} a^{n,m}_{k,l}\iota_1(\tilde{\gamma}^n)\iota_2(\tilde{\gamma}^m)$, hence in fact we have
\[
b\in\Mat_p(\mathsf{B})\subseteq\Mat_p\bigl(\C(\TT^2_\zeta)\tens\cst(\I,\gamma)\tens \cst(\I,\gamma)\bigr)\simeq\Mat_p\bigl(\C(\TT^2_\zeta)\tens\C(\Sp(\gamma))\tens\C(\Sp(\gamma))\bigr),
\]
where $\operatorname{Sp}(\gamma)=\{0\}\cup\bigl\{\lambda|q|^{k}\,\bigr|\bigl.\,\lambda\in\TT,\:k\in\ZZ_+\bigr\}$ (Lemma \ref{lemma3}).

Next, we can identify
\[
\Mat_p\bigl(\C(\TT^2_\zeta)\tens\C(\Sp(\gamma))\tens\C(\Sp(\gamma))\bigr)\simeq\C\bigl(\Sp(\gamma)\times \Sp(\gamma),\Mat_p(\C(\TT^2_\zeta))\bigr)
\]
(cf.~\cite[Theorem 6.4.17]{Murphy})
 and hence
\[
\begin{aligned}
\|b\|&=\sup\Bigl\{\bigl\|(\id\tens\ev_{z_1}\tens\ev_{z_2})b\bigr\|\,\Bigr|\Bigl.\,z_1,z_2\in\Sp(\gamma)\Bigr\}\\
&=\sup\left\{
\left\|
 \biggl[
  \sum_{n,m\in\ZZ}a^{n,m}_{k,l}U^nV^m\ev_{z_1}(\tilde{\gamma}^n)\ev_{z_2}(\tilde{\gamma}^m)
 \biggr]_{k,l=1}^{p}
\right\|_{\Mat_p(\C(\TT^2_\zeta))}
\,\vline\,
z_1,z_2\in\Sp(\gamma)
\right\}.
\end{aligned}
\]
Now we note that $\ev_{z}(\tilde{\gamma}^n)=\begin{cases}z^n,&n\ge 0\\ \overline{z}^{-n},& n<0\end{cases}$ which is exactly $g_z(n)$, where the function $g_z$ was introduced in Lemma \ref{lemma2}. Using again \cite[Proposition 7.2]{QuantumTori} and the fact that $\|a\|=1$ we obtain
\[
\begin{aligned}
\|b\|&=\sup
\left\{
 \left\|
 \biggl[\,
 \sum_{n,m\in\ZZ}a^{n,m}_{k,l}U^nV^mg_{z_1,z_2}(n,m)
 \biggr]_{k,l=1}^{p}
 \right\|_{\Mat_p(\C(\TT^2_\zeta))}
\,\vline\,
z_1,z_2\in\Sp(\gamma)
\right\}\\
&\le
\sup_{z_1,z_2\in\Sp(\gamma)}
\|g_{z_1,z_2}\|_{\cb}
=\left(\tfrac{1+|q|}{1-|q|}\right)^2.
\end{aligned}
\]
On the other hand, applying $p\times p$ matrix amplification of $(\chi^{\boxtimes})^2$ we have
\[
\bigl((\chi^{\boxtimes})^2\bigr)^{(p)}(b)
=\biggl[\,
\sum_{n,m\in\ZZ}
\zeta^{-2 nm}a_{n,m}^{k,l}U^nV^m\tens\tilde{\gamma}^n\tens\tilde{\gamma}^m
\biggr]_{k,l=1}^{p}
\!\!\in\Mat_p\bigl(\C(\TT^2_\zeta)\tens\cst(\I,\gamma)\tens\cst(\I,\gamma)\bigr)
\]
and we can bound the norm of this matrix from below using \eqref{eq6} as follows:
\[
\begin{aligned}
\bigl\|((\chi^{\boxtimes})^2)^{(p)}(b)\bigr\|
&=\sup\left\{
\left\|(\id\tens\ev_{z_1}\tens\ev_{z_2})\bigl(((\chi^{\boxtimes})^2)^{(p)}(b)\bigr)
\right\|_{\Mat_p(\C(\TT^2_\zeta))}
\,\vline\,
z_1,z_2\in \Sp(\gamma)
\right\}\\
&\ge 
\bigl\|(\id\tens\ev_{1}\tens\ev_{1})\bigl(((\chi^{\boxtimes})^2)^{(p)}(b)\bigr)
\bigr\|_{\Mat_p(\C(\TT^2_\zeta))}\\
&=\left\|
\biggl[\,
\sum_{n,m\in\ZZ}
\zeta^{-2nm}
a_{n,m}^{k,l}U^n V^m
\biggr]_{k,l=1}^{p}
\right\|_{\Mat_p(\C(\TT^2_\zeta))}\ge N.
\end{aligned}
\]
This shows that the c.b.~norm of $(\chi^\boxtimes)^2$ is not smaller than $N\left(\tfrac{1 +|q|}{1-|q|}\right)^{-2}$. Since $N$ was arbitrary, we obtain a contradiction.
\end{proof}

\end{appendices}

\section{Acknowledgements}

The work of the first author was partially supported by FWO grant 1246624N. The second author was partially supported by NCN (National Science Centre, Poland) grant no.~2022/47/B/ST1/00582. Additionally the research presented in this paper is based upon work supported by the Swedish Research Council under grant no.~2021-06594 while the second author was in residence at Institut Mittag-Leffler in Djursholm, Sweden during the year 2026.

The authors wish to thank Pawe{\l} Kasprzak for many stimulating discussions on the subject of braided quantum groups.

\end{document}